\documentclass[oneside,english,american]{amsart}
\usepackage[T1]{fontenc}
\usepackage[utf8]{inputenc}
\usepackage{prettyref}
\usepackage{float}
\usepackage{enumitem}
\usepackage{amstext}
\usepackage{amsthm}
\usepackage{amssymb}
\usepackage{geometry}
\makeatletter
\numberwithin{equation}{section}
\numberwithin{figure}{section}
\theoremstyle{plain}
\newtheorem{thm}{\protect\theoremname}[section]
\theoremstyle{definition}
\newtheorem{defn}[thm]{\protect\definitionname}
\theoremstyle{remark}
\newtheorem{rem}[thm]{\protect\remarkname}
\theoremstyle{plain}
\newtheorem{lem}[thm]{\protect\lemmaname}
\theoremstyle{plain}
\newtheorem{prop}[thm]{\protect\propositionname}
\theoremstyle{definition}
\newtheorem{example}[thm]{\protect\examplename}
\theoremstyle{plain}
\newtheorem{cor}[thm]{\protect\corollaryname}
\newlist{casenv}{enumerate}{4}
\setlist[casenv]{leftmargin=*,align=left,widest={iiii}}
\setlist[casenv,1]{label={{\itshape\ \casename} \arabic*.},ref=\arabic*}
\setlist[casenv,2]{label={{\itshape\ \casename} \roman*.},ref=\roman*}
\setlist[casenv,3]{label={{\itshape\ \casename\ \alph*.}},ref=\alph*}
\setlist[casenv,4]{label={{\itshape\ \casename} \arabic*.},ref=\arabic*}
\theoremstyle{remark}
\newtheorem*{claim*}{\protect\claimname}

\usepackage{tikz}
\usepackage{tikz-cd}
\usetikzlibrary{arrows.meta}
\usepackage{hyperref}

\newrefformat{fig}{\hyperref[#1]{Figure \ref{#1}}}
\newrefformat{tab}{\hyperref[#1]{Table \ref{#1}}}
\newrefformat{eq}{\hyperref[#1]{Equation (\ref{#1})}}
\newrefformat{def}{\hyperref[#1]{Definition \ref{#1}}}
\newrefformat{prop}{\hyperref[#1]{Proposition \ref{#1}}}
\newrefformat{exa}{\hyperref[#1]{Example \ref{#1}}}
\newrefformat{lem}{\hyperref[#1]{Lemma \ref{#1}}}
\newrefformat{cor}{\hyperref[#1]{Corollary \ref{#1}}}
\newrefformat{thm}{\hyperref[#1]{Theorem \ref{#1}}}
\newrefformat{assumptions}{\hyperref[#1]{Inductive Assumptions}}

\theoremstyle{plain}
\newtheorem*{assumptions*}{Inductive Assumptions}

\makeatother

\usepackage{babel}
\addto\captionsamerican{\renewcommand{\casename}{Case}}
\addto\captionsamerican{\renewcommand{\claimname}{Claim}}
\addto\captionsamerican{\renewcommand{\corollaryname}{Corollary}}
\addto\captionsamerican{\renewcommand{\definitionname}{Definition}}
\addto\captionsamerican{\renewcommand{\examplename}{Example}}
\addto\captionsamerican{\renewcommand{\lemmaname}{Lemma}}
\addto\captionsamerican{\renewcommand{\propositionname}{Proposition}}
\addto\captionsamerican{\renewcommand{\remarkname}{Remark}}
\addto\captionsamerican{\renewcommand{\theoremname}{Theorem}}
\addto\captionsenglish{\renewcommand{\casename}{Case}}
\addto\captionsenglish{\renewcommand{\claimname}{Claim}}
\addto\captionsenglish{\renewcommand{\corollaryname}{Corollary}}
\addto\captionsenglish{\renewcommand{\definitionname}{Definition}}
\addto\captionsenglish{\renewcommand{\examplename}{Example}}
\addto\captionsenglish{\renewcommand{\lemmaname}{Lemma}}
\addto\captionsenglish{\renewcommand{\propositionname}{Proposition}}
\addto\captionsenglish{\renewcommand{\remarkname}{Remark}}
\addto\captionsenglish{\renewcommand{\theoremname}{Theorem}}
\providecommand{\casename}{Case}
\providecommand{\claimname}{Claim}
\providecommand{\corollaryname}{Corollary}
\providecommand{\definitionname}{Definition}
\providecommand{\examplename}{Example}
\providecommand{\lemmaname}{Lemma}
\providecommand{\propositionname}{Proposition}
\providecommand{\remarkname}{Remark}
\providecommand{\theoremname}{Theorem}

\begin{document}
\title{$C^{*}$-correspondences for ordinal graphs}
\author{Benjamin Jones}
\date{February 16, 2026}
\begin{abstract}
We introduce a family of $C^{*}$-correspondences $X_{\alpha}$ naturally
associated to every ordinal graph $\Lambda$. When $\Lambda$ is a
directed graph, $X_{0}$ is isomorphic to the usual $C^{*}$-correspondence
associated to a graph. We show that ordinal graphs satisfying a weak
assumption have the property that the $C^{*}$-algebra of $\Lambda_{\alpha+1}$
is isomorphic to the Cuntz-Pimsner algebra of $X_{\alpha}$. As a
consequence, the $C^{*}$-algebra of $\Lambda$ may be constructed
starting from $c_{0}\left(\Lambda_{0}\right)$ by iteratively applying
the Cuntz-Pimsner construction and inductive limits. We apply this
result to strengthen the author's previous Cuntz-Krieger uniqueness
theorem.
\end{abstract}

\maketitle

\section{Introduction}

Since the introduction of Cuntz-Krieger algebras in \cite{CUNTZ-KRIEGER,CUNTZ-KRIEGER2},
the study of these $C^{*}$-algebras and their generalizations has
flourished, largely due to the strong uniqueness theorems these algebras
possess. As Cuntz-Krieger algebras are universal for generators and
relations, creating representations is simply a matter of finding
examples of operators satisfying the relations. A priori, one would
not expect to easily prove representations of such a large family
of algebras defined this way are faithful, but as Cuntz and Krieger
proved in their seminal paper, it is often sufficient to check that
the image of only a handful generators is non-zero.

Shortly after the introduction of Cuntz-Krieger algebras, Fujii and
Watatani observed in \cite{GRAPH-CK,ADJOINT-GRAPHS} that the $0,1$-matrices
used to construct Cuntz-Krieger algebras may be interpreted as an
edge incidence matrix for a directed graph whose structure can be
used to study the corresponding algebras. Expanding on this idea,
Kumjian, Pask, and Raeburn eventually defined in \cite{DIRECTED-GRAPH-ALGS}
the $C^{*}$-algebra $C^{*}\left(E\right)$ of a row-finite directed
graph. Subsequently, most work generalizing Cuntz and Krieger's original
algebras regarded them as algebras of the associated graph.

Cuntz-Pimsner algebras, which were introduced in \cite{PIMSNER-ALG}
and refined in \cite{CP-CONSTRUCT}, provide ways of understanding
the uniqueness theorems for graph algebras in a more abstract setting.
In \cite{KATSURA1}, Katsura showed how a topological graph, a continuous
analogue of graphs for which edges and vertices are topological spaces,
may be assigned a $C^{*}$-algebra by constructing an appropriate
$C^{*}$-correspondence and computing the Cuntz-Pimsner algebra. Generalizing
the notion of a graph has been a popular method for expanding the
theory of graph $C^{*}$-algebras. As another example, Kumjian and
Pask in \cite{KGRAPH} defined the $C^{*}$-algebra of a $k$-graph
as the $C^{*}$-algebra of an associated groupoid. Instead of the
lengths of paths being a natural number as in a directed graph, a
$k$-graph is defined as its category of paths such that each path
has a degree, or length, valued in $\mathbb{N}^{k}$. Similar constructions
allow the category of paths to have lengths valued in a category \cite{CONDUCHEFIBRATIONS,LYDIADEWOLF}
or allow one to pick an arbitrary left-cancellative category as the
category of paths \cite{LCSC}. 

In \cite{ORDGRAPH}, the author initiated the study of $C^{*}$-algebras
of ordinal graphs, categories of paths for which the lengths of paths
are ordinals. The purpose of this paper is to improve on many of the
results in this previous work by constructing new $C^{*}$-correspondences
for ordinal graphs which the author believes are of independent interest.
Since ordinal addition is left-cancellative, ordinal graphs are automatically
left cancellative as well, so we study the Cuntz-Krieger algebras
defined using groupoids by Spielberg in \cite{LCSC,CAT-PATH}. 

The structure of ordinals has some interesting consequences for the
study of ordinal graph algebras. Most importantly, ordinal addition
is not right cancellative, hence many natural examples of ordinal
graphs lack right cancellation. This means that uniqueness theorems
such as \cite[Theorem 10.12, Theorem 10.13]{CAT-PATH} and \cite[Corollary 5.7]{CONDUCHEFIBRATIONS}
do not apply. In \prettyref{cor:cku}, we improve the author's previous
Cuntz-Krieger uniqueness theorem for ordinal graph algebras. The existence
of this theorem for ordinal graphs suggests that the assumption of
right cancellation in the other theorems could be weakened. Interestingly,
the techniques in this paper largely benefit from the lack of right
cancellation of ordinal addition. The fact that for $\alpha<\beta$,
$\omega^{\alpha}+\omega^{\beta}=\omega^{\beta}$ allows us to construct
in \prettyref{def:correspondences} a hierarchy of $C^{*}$-correspondences
$X_{\alpha}$ for $\alpha\in\mathrm{Ord}$, each of which generalizing
the usual correspondence defined for a directed graph. Simultaneously,
the proof of our main result, \prettyref{thm:main-thm}, relies primarily
on a weaker notion of right cancellation in \prettyref{def:cancellative}
satisfied by paths which determine the Katsura ideals. The idea of
using a tower of $C^{*}$-correspondences is also being applied by
Deaconu, Kaliszewski, Paulovicks, and Quigg to study $k$-graph algebras
in work which is in progress.

In section 2, we establish conventions and provide some preliminary
definitions and theorems. Then we cover some basic definitions and
results for ordinal graph algebras in section 3. Section 4 is devoted
to the definition of the $C^{*}$-correspondences $X_{\alpha}$ and
characterizing the intersection between the image of the left action
and the compact operators. Section 5 contains our main result, \prettyref{thm:main-thm},
as well as an application to \prettyref{cor:cku}, the Cuntz-Krieger
uniqueness theorem. The proof of \prettyref{thm:main-thm} leaves
out some key lemmas required for completing the inductive step, so
we devote section 6 to the proof of these more technical results.

\section{Preliminaries}

Throughout the paper, we use arithmetic of ordinals. Our main reference
for ordinals is \cite{CARDINALORDINAL}. We denote the class of ordinals
by $\mathrm{Ord}$. There are many possible ways of defining ordinals,
but one intuitive method is to regard each ordinal as an order isomorphism
class of a well-ordered set. If $[A]$ and $[B]$ are the ordinals
represented by well-ordered sets $A$ and $B$, then the sum of ordinals
$[A]+[B]$ is the ordinal $[A\sqcup B]$, where $A\sqcup B$ is ordered
such that $a<b$ for all $a\in A$ and $b\in B$. The product $[A]\cdot[B]$
is obtained by well-ordering $A\times B$ such that $\left(a_{1},b_{1}\right)\leq\left(a_{2},b_{2}\right)$
if $b_{1}\leq b_{2}$, or if $b_{1}=b_{2}$, $a_{1}\leq a_{2}$. The
last operation we define is the exponentiation $[A]^{[B]}$, which
we define as the order isomorphism class of
\[
Z\left([A],[B]\right)=\left\{ f\in A^{B}:\left\{ x\in B:f(x)\not=\min A\right\} \text{ is finite}\right\} 
\]
where for $f,g\in Z\left([A],[B]\right)$, $f\leq g$ if there is
$b_{0}\in B$ such that $f\left(b_{0}\right)\leq g\left(b_{0}\right)$
and for all $b_{1}>b_{0}$, $f\left(b_{1}\right)=g\left(b_{1}\right)$.

We use $\omega$ to denote $[\mathbb{N}]$, the order isomorphism
class of the natural numbers. We identify $\mathbb{N}$ with the set
of finite ordinals $[0,\omega)$. The division algorithm holds in
$\mathrm{Ord}$, and in particular, it may be used to expand an ordinal
in base $\omega$. In particular, every ordinal may be written in
Cantor normal form:
\begin{thm}[{Cantor Normal Form \foreignlanguage{english}{\cite[Chapter XIV 19.2]{CARDINALORDINAL}}}]
Every $\alpha\in\mathrm{Ord}$ may be expressed uniquely as
\[
\alpha=\omega^{\beta_{1}}\cdot\gamma_{1}+\omega^{\beta_{2}}\cdot\gamma_{2}+\dots+\omega^{\beta_{n}}\cdot\gamma_{n}
\]
for $n,\gamma_{k}\in[0,\omega)$ and $\beta_{k}\in\mathrm{Ord}$ satisfying
$\beta_{1}\geq\beta_{2}\geq\dots\geq\beta_{n}$.
\end{thm}

Applying this theorem and the fact that $\omega^{\alpha}+\omega^{\beta}=\omega^{\beta}$
for $\alpha<\beta$ allows one to compute the sum of arbitrary ordinals
written in Cantor normal form. For example, if $\alpha=\omega^{\omega}\cdot2+\omega\cdot3+2$
and $\beta=\omega^{\omega}+\omega^{3}$, then
\begin{align*}
\alpha+\beta & =\omega^{\omega}\cdot2+\omega\cdot3+\omega^{0}\cdot2+\omega^{\omega}+\omega^{3}\\
 & =\omega^{\omega}+\omega^{\omega}+\omega+\omega+\omega+\omega^{0}+\omega^{0}+\omega^{\omega}+\omega^{3}\\
 & =\omega^{\omega}+\omega^{\omega}+\omega^{\omega}+\omega^{3}\\
 & =\omega^{\omega}\cdot3+\omega^{3}
\end{align*}

As we make extensive use of $C^{*}$-correspondences, we now provide
the relevant definitions. For a reference on Hilbert modules, we suggest
\cite{HILMOD}.
\begin{defn}
A \emph{$C^{*}$-correspondence} over a $C^{*}$-algebra $A$ is a
pair $\left(X,\varphi\right)$ where $X$ is a right-Hilbert $A$-module
and $\varphi:A\rightarrow\mathcal{L}\left(X\right)$ is a {*}-homomorphism
into the adjointable operators of $X$. We denote by $a\cdot x$ the
vector $\varphi\left(a\right)x$, making $X$ into a left $A$ module.
The \emph{Katsura ideal} is the closed, two-sided ideal $J_{X}$ of
$A$ defined as
\[
J_{X}=\left(\ker\varphi\right)^{\perp}\cap\varphi^{-1}\left(\mathcal{K}\left(X\right)\right)
\]
\end{defn}

We will usually write $X$ instead of the full data $\left(X,\varphi\right)$
for a $C^{*}$-correspondence.
\begin{defn}
If $X$ is a $C^{*}$-correspondence over $A$, a \emph{representation}
into a $C^{*}$-algebra $B$ is a pair $\left(\psi,\pi\right)$ such
that $\psi:X\rightarrow B$ is linear, $\pi:A\rightarrow B$ is a
{*}-homomorphism, and
\begin{enumerate}
\item $\psi\left(x\right)\pi\left(a\right)=\psi\left(x\cdot a\right)$
\item $\pi\left(a\right)\psi\left(x\right)=\psi\left(\varphi\left(a\right)x\right)$
\item $\psi\left(x\right)^{*}\psi\left(y\right)=\pi\left(\left\langle x,y\right\rangle _{A}\right)$
\end{enumerate}
We write $\left(\psi,\pi\right):\left(X,A\right)\rightarrow B$ to
denote the fact that $\left(\psi,\pi\right)$ is a representation
of $X$ into $B$.
\end{defn}

\begin{defn}
For a representation $\left(\psi,\pi\right):\left(X,A\right)\rightarrow B$,
we denote by $\left(\psi,\pi\right)^{(1)}$ the {*}-homomorphism $\left(\psi,\pi\right)^{(1)}:\mathcal{K}\left(X\right)\rightarrow B$
uniquely determined by
\[
\left(\psi,\pi\right)^{(1)}\left(\theta_{x,y}\right)=\psi\left(x\right)\psi\left(y\right)^{*}
\]
\end{defn}

\begin{defn}
A representation $\left(\psi,\pi\right):\left(X,A\right)\rightarrow B$
is \emph{covariant} if for all $a\in J_{X}$, $\left(\psi,\pi\right)^{(1)}\left(\varphi\left(a\right)\right)=\pi\left(a\right)$.
\end{defn}

\begin{defn}[{\cite[Definition 3.5]{KATIDEAL2}}]
The \emph{Cuntz-Pimsner algebra} $\mathcal{O}\left(X\right)$ of
a $C^{*}$-correspondence $X$ is the $C^{*}$-algebra which is universal
for covariant representations of $X$. In particular, there exists
a covariant representation $\left(\psi_{u},\pi_{u}\right):\left(X,A\right)\rightarrow\mathcal{O}\left(X\right)$
such that for all covariant representations $\left(\psi,\pi\right):\left(X,A\right)\rightarrow B$
there exists a unique {*}-homomorphism $\psi\times\pi:\mathcal{O}\left(X\right)\rightarrow B$
such that $\psi=\psi\times\pi\circ\psi_{u}$ and $\pi=\psi\times\pi\circ\pi_{u}$. 
\end{defn}

\begin{rem}
One can check that if $\left(\psi,\pi\right)$ is a covariant representation
of $X$ and $z\in\mathbb{T}$, $\left(\psi_{z},\pi\right)$ is also
a representation of $X$ where $\psi_{z}\left(x\right)=z\psi\left(x\right)$.
Therefore there is an action $\gamma:\mathbb{T}\rightarrow\mathrm{Aut}\left(\mathcal{O}\left(X\right)\right)$
defined by $\gamma_{z}\left(\pi_{u}\left(a\right)\right)=\pi_{u}\left(a\right)$
and $\gamma_{z}\left(\psi_{u}\left(x\right)\right)=z\psi_{u}\left(x\right)$
called the \emph{gauge action}.
\end{rem}

For the rest of the paper, all homomorphisms will be {*}-homomorphisms,
and all ideals will be closed and two-sided. Finally, in section 4
we make brief use of multiplier algebras. Conveniently, multiplier
algebras may also be defined using the technology of Hilbert modules
(see \cite[Chapter 2]{HILMOD}).

\section{Ordinal Graphs}

In the following definition, we regard a category to be its collection
of morphisms. Then $s,r:\Lambda\rightarrow\Lambda$ map into the collection
of identity morphisms, which may be identified with objects in the
category. A small category is a category for which the morphisms are
small enough to be a set.
\begin{defn}
\label{def:ordinal-graph}An \emph{ordinal graph} is a pair $\left(\Lambda,d\right)$
where $\Lambda$ is a small category and $d:\Lambda\rightarrow\mathrm{Ord}$
is a functor into the ordinals with the following factorization property:
\begin{quote}
For every $e\in\Lambda$ and $\alpha\in\mathrm{Ord}$ with $\alpha\leq d\left(e\right)$,
there exist unique $f,g\in\Lambda$ such that $d\left(f\right)=\alpha$
and $e=fg$.
\end{quote}
If $e=fg$ and $d\left(f\right)=\alpha$, we denote $f$ by $e\left(\alpha\right)$
and $g$ by $e\left(\alpha\right)^{-1}e$. Then by definition, we
have $e\left(\alpha\right)e\left(\alpha\right)^{-1}e=e$ and $d\left(e\left(\alpha\right)\right)=\alpha$.
For each $\alpha\in\mathrm{Ord}$, define
\begin{align*}
\Lambda^{\alpha} & =\left\{ f\in\Lambda:d\left(f\right)=\alpha\right\} \\
\Lambda_{\alpha} & =\left\{ f\in\Lambda:d\left(f\right)<\omega^{\alpha}\right\} \\
\Lambda^{\omega^{*}} & =\left\{ f\in\Lambda:d\left(f\right)=\omega^{\alpha}\text{ for some }\alpha\in\mathrm{Ord}\right\} 
\end{align*}
\end{defn}

Note that if $\alpha,\beta<\omega^{\gamma}$, then the properties
of ordinal arithmetic imply $\alpha+\beta<\omega^{\gamma}$. Thus
if $f,g\in\Lambda_{\alpha}$ with $s\left(f\right)=r\left(g\right)$,
$fg\in\Lambda_{\alpha}$. This makes $\Lambda_{\alpha}$ into an ordinal
graph. On the other hand, $\Lambda^{\alpha}$ is not an ordinal graph
unless $\alpha=0$, in which case $\Lambda^{0}=\Lambda_{0}$ is simply
the set of vertices. The set of vertices is also the range of the
maps $s,r$. One direction follows by observing that for each $e\in\Lambda$,
$s\left(e\right)s\left(e\right)=s\left(e\right)$ and $r\left(e\right)r\left(e\right)=r\left(e\right)$.
Hence $d\left(s\left(e\right)\right)+d\left(s\left(e\right)\right)=d\left(s\left(e\right)\right)=0$,
and similarly for $r\left(e\right)$. Likewise, if $d\left(e\right)=0$,
then since $r\left(e\right)e=es\left(e\right)$, unique factorization
implies $e=r\left(e\right)=s\left(r\left(e\right)\right)$.

By the factorization property, each ordinal graph is a left cancellative
category: if $ef=eg$ are two factorizations of the same path, then
$f=g$. Then by \foreignlanguage{english}{\cite[Theorem 10.15]{LCSC}},
the $C^{*}$-algebra $\mathcal{O}\left(\Lambda\right)$ has a natural
presentation in terms of generators and relations. $\mathcal{O}\left(\Lambda\right)$
is generated by elements $\left\{ T_{e}:e\in\Lambda\right\} $, and
in particular, $T_{e}T_{f}=T_{ef}$ if $e$ and $f$ are composable
in $\Lambda$. Since each ordinal is a sum of powers of $\omega$,
each generator $T_{e}$ for which $d\left(e\right)\not=0$ is a product
of some generators $T_{e_{n}}$ where the length of each $e_{n}$
is a power of $\omega$. Thus $\mathcal{O}\left(\Lambda\right)$ is
generated by $\left\{ T_{e}:e\in\Lambda^{\omega^{\alpha}}\text{ for some }\alpha\in\mathrm{Ord}\right\} \cup\left\{ T_{v}:v\in\Lambda_{0}\right\} $,
and the relations for these generators are determined by the following
result, which will be convenient for us.
\begin{thm}[{\cite[Proposition 4.20]{ORDGRAPH}}]
\label{thm:gen-relations}$\mathcal{O}\left(\Lambda\right)$ is universal
for partial isometries
\[
\bigcup_{\alpha\in\mathrm{Ord}}\left\{ T_{e}:e\in\Lambda^{\omega^{\alpha}}\right\} \cup\left\{ T_{v}:v\in\Lambda_{0}\right\} 
\]
which satisfy the following relations:
\begin{enumerate}
\item $T_{e}^{*}T_{e}=T_{s\left(e\right)}$
\item $T_{e}T_{f}=T_{ef}$ if $s\left(e\right)=r\left(f\right)$ and $d\left(e\right)<d\left(f\right)$
\item $T_{e}^{*}T_{f}=0$ if $e\Lambda\cap f\Lambda=\emptyset$
\item $T_{v}=\sum_{e\in\Lambda^{\omega^{\alpha}}}T_{e}T_{e}^{*}$ if $v\in\Lambda_{0}$
is $\alpha$-regular
\end{enumerate}
\end{thm}

We call any family of operators $\left\{ S_{e}:e\in\Lambda^{\omega^{\alpha}}\text{ for }\alpha\in\mathrm{Ord}\right\} \cup\left\{ S_{v}:v\in\Lambda_{0}\right\} $
satisfying relations (1) through (4) a \emph{Cuntz-Krieger }family
for $\Lambda$. Relation (2) is well-defined because if $d\left(e\right)=\omega^{\alpha}<\omega^{\beta}=d\left(f\right)$,
$d\left(ef\right)=\omega^{\alpha}+\omega^{\beta}=\omega^{\beta}$.
We now proceed to define $\alpha$-regular vertices; note, in particular,
that the sum in relation (4) is non-empty and finite as a consequence
of the definitions.
\begin{defn}[{\cite[Definition 4.1]{ORDGRAPH}}]
A vertex $v\in\Lambda_{0}$ is \emph{$\alpha$-source-regular} for
$\alpha\in\mathrm{Ord}$ if for every $e\in v\Lambda_{\alpha}$, there
exists $f\in s\left(e\right)\Lambda^{\omega^{\alpha}}$.
\end{defn}

In particular, if $\Lambda$ is a directed graph (which is equivalent
to $d\left(\Lambda\right)\subseteq[0,\omega)$), then $v$ is $0$-source-regular
if $v$ is not a source. By a source, we mean a vertex $v$ such that
$r\left(f\right)=v$ implies $d\left(f\right)=0$.
\begin{defn}[{\cite[Definition 4.3]{ORDGRAPH}}]
A vertex $v\in\Lambda_{0}$ is \emph{$\alpha$-row-finite} for $\alpha\in\mathrm{Ord}$
if $\left|\left\{ f\in v\Lambda^{\omega^{\alpha}}\right\} \right|<\infty$.
We define $v$ to be \emph{$\alpha$-regular} if $v$ is $\alpha$-source-regular
and $\alpha$-row-finite.
\end{defn}

In fact, $\alpha$-regularity of a vertex $v\in\Lambda_{0}$ is quite
a strong condition, as the following lemma demonstrates.
\begin{lem}
\label{lem:regular-hereditary}If $v\in\Lambda_{0}$ is $\alpha$-regular,
then
\begin{enumerate}
\item $s\left(f\right)$ is $\alpha$-regular for each $f\in v\Lambda_{\alpha}$
\item $v$ is $\beta$-regular for all $\beta<\alpha$
\end{enumerate}
\end{lem}

\begin{proof}
Let $\alpha$-regular $v\in\Lambda_{0}$ and $f\in v\Lambda_{\alpha}$
be given. If $g\in s\left(f\right)\Lambda_{\alpha}$, then $d\left(fg\right)=d\left(f\right)+d\left(g\right)<\omega^{\alpha}$,
hence by $\alpha$-source-regularity of $v$, $s\left(fg\right)\Lambda^{\omega^{\alpha}}=s\left(g\right)\Lambda^{\omega^{\alpha}}\not=\emptyset$,
and $s\left(f\right)$ is $\alpha$-source-regular. Moreover, if $h\in s\left(f\right)\Lambda^{\omega^{\alpha}}$,
then $d\left(fh\right)=d\left(f\right)+\omega^{\alpha}=\omega^{\alpha}$,
so $fh\in v\Lambda^{\omega^{\alpha}}$. By left cancellation, the
function $h\mapsto fh$ is injective, and since $v\Lambda^{\omega^{\alpha}}$
is finite, $s\left(f\right)\Lambda^{\omega^{\alpha}}$ is finite.
Thus $s\left(f\right)$ is $\alpha$-regular.

Now we will prove $v$ is $\beta$-regular if $\beta<\alpha$. If
$g\in v\Lambda_{\beta}\subseteq v\Lambda_{\alpha}$, there exists
$h\in s\left(g\right)\Lambda^{\omega^{\alpha}}$, and $h\left(\omega^{\beta}\right)\in s\left(g\right)\Lambda^{\omega^{\beta}}$.
Thus $v$ is $\beta$-source-regular. Moreover, if $h\in v\Lambda^{\omega^{\beta}}$,
then there exists $p_{h}\in s\left(h\right)\Lambda^{\omega^{\alpha}}$,
and $hp_{h}\in v\Lambda^{\omega^{\alpha}}$. By the factorization
property in \prettyref{def:ordinal-graph}, the map $h\mapsto hp_{h}$
from $v\Lambda^{\omega^{\beta}}$ to $v\Lambda^{\omega^{\alpha}}$
is injective, and thus $v\Lambda^{\beta}$ is finite. 
\end{proof}
The following facts are useful for calculations with relation (3).
\begin{prop}[{\cite[Lemma 3.17]{ORDGRAPH}}]
\label{prop:relation-3}For an ordinal graph $\Lambda$, $e\Lambda\cap f\Lambda\not=\emptyset$
if and only if $e\in f\Lambda$ or $f\in e\Lambda$.
\end{prop}

\begin{prop}[{\cite[Corollary 3.21]{ORDGRAPH}}]
\label{prop:ordgraph-closed-span}If $\Lambda$ is an ordinal graph,
then
\[
\mathcal{O}\left(\Lambda\right)=\overline{\mathrm{span}}\left\{ T_{e}T_{f}^{*}:e,f\in\Lambda\right\} 
\]
\end{prop}

Throughout this paper, we will determine relationships between the
algebras $\mathcal{O}\left(\Lambda_{\alpha}\right)$. Despite the
similar generators and relations, we will see shortly that these algebras
are not necessarily contained in one another; nevertheless, we denote
by $\left\{ T_{e}:e\in\Lambda_{\alpha}\right\} $ the generators of
$\mathcal{O}\left(\Lambda_{\alpha}\right)$, regardless of the value
of $\alpha$. We are careful to avoid the ambiguity of which algebra
$T_{e}$ belongs to by specifying domains and codomains for all involved
maps. 

We do, however, have connecting maps $\rho_{\alpha}^{\beta}:\mathcal{O}\left(\Lambda_{\alpha}\right)\rightarrow\mathcal{O}\left(\Lambda_{\beta}\right)$
between these algebras when $\alpha\leq\beta$. To see this, it suffices
to verify that every instance of relations (1)-(4) in $\Lambda_{\alpha}$
continues to hold in $\Lambda_{\beta}$. It is straightforward to
check that this is true for relations (1), (2), and (4), but there
is a subtlety with verifying relation (3). Suppose $e,f\in\Lambda_{\alpha}$
and $e\Lambda_{\alpha}\cap f\Lambda_{\alpha}=\emptyset$. We must
verify $e\Lambda_{\beta}\cap f\Lambda_{\beta}=\emptyset$. Using \prettyref{prop:relation-3},
Suppose for the sake of contradiction that $e=fg$ for some $g\in\Lambda$.
Then $d\left(e\right)=d\left(f\right)+d\left(g\right)$, hence $d\left(g\right)=-d\left(f\right)+d\left(e\right)<\omega^{\alpha}$,
and $g\in\Lambda_{\alpha}$. This contradicts our assumption, and
the case for $f\in e\Lambda$ is similar. We summarize these observations
and prove more in the following result.
\begin{prop}
\label{prop:inductive-limit}For each $\alpha\leq\beta$ there exists
a {*}-homomorphism $\rho_{\alpha}^{\beta}:\mathcal{O}\left(\Lambda_{\alpha}\right)\rightarrow\mathcal{O}\left(\Lambda_{\beta}\right)$
such that $\rho_{\alpha}^{\beta}\left(T_{e}\right)=T_{e}$. Moreover,
for each limit ordinal $\beta>0$, we have the following:
\[
\mathcal{O}\left(\Lambda_{\beta}\right)\cong\lim_{\rightarrow}\left(\mathcal{O}\left(\Lambda_{\alpha}\right),\rho_{\alpha}^{\gamma}\right)_{\alpha<\gamma<\beta}
\]
\end{prop}

\begin{proof}
The existence of $\rho_{\alpha}^{\beta}$ follows from the discussion
above, so we verify the isomorphism. Let $\beta$ be a limit ordinal,
that is, an ordinal not of the form $\gamma+1$. Let $\mathcal{A}$
be a $C^{*}$-algebra with {*}-homomorphisms $\epsilon_{\gamma}:\mathcal{O}\left(\Lambda_{\gamma}\right)\rightarrow\mathcal{A}$
for each $\gamma<\beta$. Suppose that for each $\alpha<\gamma<\beta$,
we have $\rho_{\alpha}^{\gamma}\circ\epsilon_{\gamma}=\epsilon_{\alpha}$.
Define
\[
S=\bigcup_{\gamma<\beta}\left\{ \epsilon_{\gamma}\left(T_{e}\right):e\in\Lambda^{\omega^{\alpha}}\cup\Lambda_{0}\text{ for some }\alpha\in\mathrm{Ord}\right\} 
\]
We claim $S$ is a Cuntz-Krieger family for $\Lambda_{\beta}$. Let
$S_{e}$ denote $\epsilon_{\gamma}\left(T_{e}\right)$. Any of the
relations in \prettyref{thm:gen-relations} involves only finitely
many generators, say $\left\{ S_{e_{1}},S_{e_{2}},\dots S_{e_{n}}\right\} $
for $e_{n}\in\Lambda_{\gamma_{n}}$ with $\gamma_{n}<\beta$. Then
$\gamma=\max\left\{ \gamma_{1},\dots\gamma_{n}\right\} <\beta$, and
the relation holds in $\mathcal{O}\left(\Lambda_{\gamma}\right)$
for generators $\left\{ T_{e_{1}},T_{e_{2}},\dots T_{e_{n}}\right\} $.
Since $\epsilon_{\gamma}$ is a homomorphism and $\epsilon_{\gamma_{k}}=\rho_{\gamma_{k}}^{\gamma}\circ\epsilon_{\gamma}$,
the relation also holds for $\left\{ S_{e_{1}},\dots S_{e_{n}}\right\} $
in $\mathcal{A}$. Therefore there exists a {*}-homomorphism $\epsilon:\mathcal{O}\left(\Lambda_{\beta}\right)\rightarrow\mathcal{A}$
such that $\epsilon\left(T_{e}\right)=S_{e}$. This implies, in particular,
$\rho_{\gamma}^{\beta}\circ\epsilon=\epsilon_{\gamma}$. If $e\in\Lambda_{\beta}$,
then $d\left(e\right)<\omega^{\beta}$, and in particular, $d\left(e\right)<\omega^{\gamma}<\omega^{\beta}$
for some $\gamma<\beta$. Then $\mathcal{O}\left(\Lambda_{\beta}\right)$
is the closure of $\cup_{\gamma<\beta}\:\rho_{\gamma}^{\beta}\left(\mathcal{O}\left(\Lambda_{\gamma}\right)\right)$,
and hence by continuity, $\epsilon$ is the only {*}-homomorphism
satisfying $\rho_{\gamma}^{\beta}\circ\epsilon=\epsilon_{\gamma}$
for each $\gamma<\beta$. Since $\mathcal{O}\left(\Lambda_{\beta}\right)$
satisfies the universal property for the inductive limit, $\mathcal{O}\left(\Lambda_{\beta}\right)$
is isomorphic to the inductive limit.
\end{proof}
Now we provide a concrete definition of the boundary paths of an ordinal
graph which in hindsight agrees with \cite[Definition 10.2]{LCSC}.
It is paramount that our definition of boundary path is weak enough
to allow us to build rich representations, and restricting ourselves
to maximal paths would prevent us from applying our strategy in the
proof of \prettyref{lem:injective}. The original definition of boundary
paths of a directed graph is attributed to folklore, but for a view
on how these definitions work for directed graphs, see for example
\cite[Definition 3.3]{NAIMARK-GRAPH}. In particular, boundary paths
of a directed graph are only allowed to terminate at a vertex which
is not regular. In the following definition, a boundary path which
encounters an $\alpha$-regular vertex is required to be at least
longer by $\omega^{\alpha}$. Hence a boundary path which encounters
a $0$-regular vertex is required to be longer by at least another
edge, and this definition is consistent in the case $\Lambda$ is
a directed graph.
\begin{defn}[{cf. \cite[Definition 10.2]{LCSC}}]
\label{def:boundary-def}If $\Lambda$ is an ordinal graph, define
\[
\Lambda^{*}=\left\{ f\in\prod_{\beta<\alpha}\Lambda^{\beta}:\alpha\in\mathrm{Ord},\alpha>0,\text{ and }f\left(\beta\right)\in f\left(\gamma\right)\Lambda\text{ for all }\gamma\leq\beta\right\} 
\]
For each $f\in\Lambda^{*}\cap\prod_{\beta<\alpha}\Lambda^{\beta}$,
define $L\left(f\right)=\alpha$. Also, define the set of \emph{boundary
paths} as
\[
\partial\Lambda=\left\{ f\in\Lambda^{*}:\text{for all }\gamma<L\left(f\right),\text{ if }s\left(f\left(\gamma\right)\right)\text{ is }\alpha\text{-regular, then }L\left(f\right)>\gamma+\omega^{\alpha}\right\} 
\]
\end{defn}

If $\Lambda$ is a directed graph, then $\Lambda^{*}$ can be regarded
as the set of all finite and infinite ``paths''. However, an ordinal
graph in general already contains paths with infinite lengh ($e\in\Lambda$
with $d\left(e\right)\geq\omega$). In this context, it's helpful
to instead regard $\Lambda^{*}$ to be the set of all ``paths''
which may or may not have a source vertex. Then we may naturally identify
$\Lambda$ as a subset of $\Lambda^{*}$ by mapping $f\in\Lambda$
to $g\in\Lambda^{*}$ defined by $g\left(\beta\right)=f\left(\beta\right)$
for all $\beta<L\left(g\right)=d\left(g\right)+1$. Since the notation
for $f\left(\beta\right)$ is consistent when $f\in\Lambda$ is viewed
as a member of $\Lambda^{*}$, we make no distinction members of $\Lambda$
and their image in $\Lambda^{*}$. 

We may define the range of a member $f$ in $\Lambda^{*}$ as $f\left(0\right)$.
Then, if $s\left(e\right)=r\left(f\right)$ and $e\in\Lambda$, we
define $ef\in\Lambda^{*}$ such that $L\left(ef\right)=d\left(e\right)+L\left(f\right)$
and $\left(ef\right)\left(\beta\right)=ef\left(-d\left(e\right)+\beta\right)$
for $\beta\geq d\left(e\right)$. Since $\left(ef\right)\left(\gamma\right)=\left(ef\right)\left(\beta\right)\left(\gamma\right)$
for $\gamma<\beta$, this uniquely determines $ef\in\Lambda^{*}$.
Likewise, if $f\in\Lambda^{*}$ and $f\left(d\left(e\right)\right)=e$
for some $e\in\Lambda$ with $d\left(e\right)<L\left(f\right)$, then
we define $e^{-1}f\in\Lambda^{*}$ by defining $L\left(e^{-1}f\right)=-d\left(e\right)+L\left(f\right)$
and $\left(e^{-1}f\right)\left(\beta\right)=e^{-1}f\left(d\left(e\right)+\beta\right)$.
This notation for members in $\Lambda^{*}$ is compatible with the
way the notation is used for paths in $\Lambda$. For example, it
follows that if $e,f\in\Lambda$ with $d\left(f\right)\geq d\left(e\right)$
and $g\in\Lambda^{*}$, $e^{-1}\left(fg\right)=\left(e^{-1}f\right)g$.

$\Lambda^{*}$ also has a natural partial order, where $f\leq g$
if $L\left(f\right)\leq L\left(g\right)$ and $f\left(\beta\right)=g\left(\beta\right)$
for each $\beta<L\left(f\right)$. As we see in the next result, this
partial order gives us a method for constructing boundary paths.
\begin{lem}
\label{lem:boundary-paths-exist}If $f\in\Lambda^{*}$ is maximal,
then $f\in\partial\Lambda$. Furthermore, for each $v\in\Lambda_{0}$
there exists $f\in v\partial\Lambda$ which is maximal in $\Lambda^{*}$.
\end{lem}

\begin{proof}
Suppose $f\in\Lambda^{*}$ is maximal, $\beta<L\left(f\right)$, $w=s\left(f\left(\beta\right)\right)$
is $\alpha$-regular, and $L\left(f\right)\leq\beta+\omega^{\alpha}$.
Define a function $h:\left[0,-\beta+L\left(f\right)\right)\rightarrow\mathbb{N}$
such that $h\left(\epsilon\right)=\left|f\left(\beta+\epsilon\right)\Lambda^{\omega^{\alpha}}\right|$.
Then $h$ is decreasing because $p\Lambda^{\omega^{\alpha}}\subseteq q\Lambda^{\omega^{\alpha}}$
if $p,q\in\Lambda_{\alpha}$ and $p\in q\Lambda$. Thus $h$ is eventually
constant for inputs $\epsilon\geq\eta$. Then $\eta<-\beta+L\left(f\right)\leq\omega^{\alpha}$,
and $f\left(\beta\right)^{-1}f\left(\beta+\eta\right)\in w\Lambda_{\alpha}$.
Hence by $\alpha$-source-regularity of $w$, $f\left(\beta+\eta\right)\Lambda^{\omega^{\alpha}}$
is non-empty, and $h\left(\eta\right)\geq1$. Choose $g\in f\left(\beta+\eta\right)\Lambda^{\omega^{\alpha}}$
arbitrarily, and note that $f\leq g$ since $h\left(\epsilon\right)$
is constant for $\epsilon\geq\eta$. However, $L\left(g\right)=d\left(g\right)+1=\beta+\eta+\omega^{\alpha}+1=\beta+\omega^{\alpha}+1>L\left(f\right)$,
and this contradicts the maximality of $f$. Thus if $s\left(f\left(\beta\right)\right)$
is $\alpha$-regular, then $L\left(f\right)>\beta+\omega^{\alpha}$,
and $f\in\partial\Lambda$.

Now, given $v\in\Lambda_{0}$, we may construct an element of $v\partial\Lambda$
by choosing a maximal element of $v\Lambda^{*}$ using Zorn's lemma.
If $C\subseteq v\Lambda^{*}$ is a chain, then we let $\alpha=\sup_{f\in C}L\left(f\right)$.
Define $g\in\Lambda^{*}$ such that $L\left(g\right)=\alpha$ and
$g\left(\epsilon\right)=f\left(\epsilon\right)$ for arbitrary $f\in C$
with $L\left(f\right)>\epsilon$. Since $C$ is totally ordered, $g$
is well-defined, and we have constructed $g$ to be an upper bound
of $C$. Therefore, there exists a maximal element $f\in v\Lambda^{*}$,
which by the above argument must also belong to $\partial\Lambda$.
\end{proof}
\begin{lem}
\label{lem:compose-boundary}If $e\in\Lambda$ and $f\in\Lambda^{*}$
with $r\left(f\right)=s\left(e\right)$, then $ef\in\partial\Lambda$
if and only if $f\in\partial\Lambda$.
\end{lem}

\begin{proof}
Since every ordinal is a finite sum of powers of $\omega$, every
path $e\in\Lambda$ is either a vertex or a composition of finitely
many paths whose lengths are powers of $\omega$. Therefore, it suffices
to consider the case where $d\left(e\right)=\omega^{\beta}$ for some
$\beta\in\mathrm{Ord}$. We will begin with the foward direction.
Then let $\gamma<L\left(ef\right)=d\left(e\right)+L\left(f\right)$
such that $v=s\left(\left(ef\right)\left(\gamma\right)\right)$ is
$\alpha$-regular. We wish to show $L\left(ef\right)>\gamma+\omega^{\alpha}$. 

We will verify this in cases. First suppose $\gamma\geq d\left(e\right)$.
Then $\left(ef\right)\left(\gamma\right)=f\left(-d\left(e\right)+\gamma\right)$,
so $s\left(f\left(-d\left(e\right)+\gamma\right)\right)$ is $\alpha$-regular.
Since $f\in\partial\Lambda$, this implies $L\left(f\right)>-d\left(e\right)+\gamma+\omega^{\alpha}$,
and adding $d\left(e\right)$ to both sides, we see $L\left(ef\right)>\gamma+\omega^{\alpha}$.
Next, consider the case where $\gamma<d\left(e\right)$ and $\beta\geq\alpha$.
Then $L\left(ef\right)\geq d\left(e\right)+1=\omega^{\beta}+1=\gamma+\omega^{\beta}+1>\gamma+\omega^{\alpha}$.
Finally, if $\gamma<d\left(e\right)$ and $\beta<\alpha$, then $\left(ef\right)\left(\gamma\right)=e\left(\gamma\right)$
and $v=s\left(e\left(\gamma\right)\right)$ is $\alpha$-regular.
Since $e\in\Lambda_{\alpha}$, \prettyref{lem:regular-hereditary}
implies $s\left(e\left(\gamma\right)^{-1}e\right)=s\left(e\right)=r\left(f\right)$
is $\alpha$-regular, and $f\in\partial\Lambda$ implies $L\left(ef\right)\geq L\left(f\right)>\omega^{\alpha}=\gamma+\omega^{\alpha}$.

For the other direction, we suppose $ef\in\partial\Lambda$ and $\gamma<L\left(f\right)$
with $s\left(f\left(\gamma\right)\right)$ $\alpha$-regular. Then
$s\left(f\left(\gamma\right)\right)=s\left(\left(ef\right)\left(d\left(e\right)+\gamma\right)\right)$,
so $L\left(ef\right)=d\left(e\right)+L\left(f\right)>d\left(e\right)+\gamma+\omega^{\alpha}$,
and subtracting $d\left(e\right)$ from both sides, $L\left(f\right)>\gamma+\omega^{\alpha}$.
\end{proof}
\begin{prop}
\label{prop:boundary-path-rep}For each ordinal graph $\Lambda$,
there is a representation $\tau:\mathcal{O}\left(\Lambda\right)\rightarrow B\left(\ell^{2}\left(\partial\Lambda\right)\right)$
defined by
\[
\tau\left(T_{e}\right)\xi_{f}=\begin{cases}
\xi_{ef} & s\left(e\right)=r\left(f\right)\\
0 & s\left(e\right)\not=r\left(f\right)
\end{cases}
\]
\end{prop}

\begin{proof}
We verify that the operators $\left\{ \tau\left(T_{e}\right):e\in\Lambda\right\} $
satisfy the relations of \prettyref{thm:gen-relations}, which is
sufficient by universality. First note that the adjoints are defined
by
\[
\tau\left(T_{e}\right)^{*}\xi_{f}=\begin{cases}
\xi_{e^{-1}f} & f\in e\partial\Lambda\\
0 & f\not\in e\partial\Lambda
\end{cases}
\]
Relation (1) follows because $e^{-1}ef=f$ for all $f\in s\left(e\right)\partial\Lambda$.
Similarly, $\left(ef\right)g=e\left(fg\right)$ for $e\in r\left(f\right)\Lambda$,
$f\in r\left(g\right)\Lambda$, and $g\in\partial\Lambda$, so relation
(2) holds. For relation (3), note that if $e\not\in f\Lambda$ and
$d\left(e\right)\leq d\left(f\right)$, then for $g\in s\left(f\right)\partial\Lambda$,
$fg\not\in e\partial\Lambda$, otherwise $\left(fg\right)\left(d\left(e\right)\right)=f\left(d\left(e\right)\right)=e$.
Finally, relation (4) follows from the definition of $\partial\Lambda$.
If $v\in\Lambda_{0}$ is $\alpha$-regular and $f\in v\partial\Lambda$,
then $L\left(f\right)>\omega^{\alpha}$. Thus $f=f\left(\omega^{\alpha}\right)f\left(\omega^{\alpha}\right)^{-1}f$,
and $\tau\left(T_{f\left(\omega^{\alpha}\right)}\right)\tau\left(T_{f\left(\omega^{\alpha}\right)}\right)^{*}\xi_{f}=\xi_{f}$.
\end{proof}
Next we prove some identities involving the factorization of paths
in $\Lambda$.
\begin{lem}
\label{lem:path-factor}If $e,f\in\Lambda$, then
\begin{enumerate}
\item $\left(ef\right)\left(\alpha\right)=e\left(\alpha\right)$ if $\alpha\leq d\left(e\right)$.
\item $\left(ef\right)\left(\alpha\right)=ef\left(-d\left(e\right)+\alpha\right)$
if $\alpha\geq d\left(e\right)$.
\end{enumerate}
\end{lem}

\begin{proof}
Let $e,f\in\Lambda$ be given, and suppose $\alpha\leq d\left(e\right)$.
Then we have
\[
e\left(\alpha\right)e\left(\alpha\right)^{-1}ef=ef
\]
where $d\left(e\left(\alpha\right)\right)=\alpha$. By \prettyref{def:ordinal-graph},
this implies $\left(ef\right)\left(\alpha\right)=e\left(\alpha\right)$.
Now suppose $\alpha\geq d\left(e\right)$, in which case
\[
ef\left(-d\left(e\right)+\alpha\right)f\left(-d\left(e\right)+\alpha\right)^{-1}f=ef
\]
and 
\[
d\left(ef\left(-d\left(e\right)+\alpha\right)\right)=d\left(e\right)+d\left(f\left(-d\left(e\right)+\alpha\right)\right)=d\left(e\right)-d\left(e\right)+\alpha=\alpha
\]
Therefore $\left(ef\right)\left(\alpha\right)=ef\left(-d\left(e\right)+\alpha\right)$.
\end{proof}
We end this section by noting the existence of certain actions which
generalize the gauge action for graph algebras.
\begin{lem}
\label{lem:alg-auts}For $\alpha\in\mathrm{Ord}$, there is an action
$\Gamma_{\alpha}:\mathbb{T}\rightarrow\mathrm{Aut}\left(\mathcal{O}\left(\Lambda_{\alpha+1}\right)\right)$
such that
\[
\Gamma_{\alpha,z}\left(T_{e}\right)=\begin{cases}
zT_{e} & d\left(e\right)=\omega^{\alpha}\\
T_{e} & d\left(e\right)<\omega^{\alpha}
\end{cases}
\]
Moreover, $z\mapsto\Gamma_{\alpha,z}\left(a\right)$ is continuous
for each $a\in\mathcal{O}\left(\Lambda_{\alpha+1}\right)$.
\end{lem}

\begin{proof}
Let $\left\{ T_{e}:e\in\Lambda^{\omega^{\beta}},\beta<\alpha+1\right\} \cup\left\{ T_{v}:v\in\Lambda_{0}\right\} $
be the generators of $\mathcal{O}\left(\Lambda_{\alpha+1}\right)$
from \prettyref{thm:gen-relations}. Let $z\in\mathbb{T}$, and for
$e\in\Lambda^{\omega^{\alpha}}$, define $S_{e}=zT_{e}$. Likewise,
if $\beta<\alpha$ and $e\in\Lambda^{\omega^{\beta}}\cup\Lambda_{0}$,
define $S_{e}=T_{e}$. We claim $\left\{ S_{e}:e\in\Lambda^{\omega^{\beta}},\beta<\alpha+1\right\} \cup\left\{ S_{v}:v\in\Lambda_{0}\right\} $
is a Cuntz-Krieger family for $\Lambda_{\alpha+1}$. Relations (1),
(3), and (4) are immediate. Relation (2) follows because if $d\left(e\right)<d\left(f\right)$
for some $e\in\Lambda^{\omega^{\beta}}$ with $d\left(f\right)$ a
power of $\omega$, then $\beta<\alpha$ and $ef\in\Lambda^{\omega^{\alpha}}$
if and only if $f\in\Lambda^{\omega^{\alpha}}$. The universal property
then induces automorphisms $\Gamma_{\alpha,z}$ which satisfy the
necessary requirements.

To see that $z\mapsto\Gamma_{\alpha,z}\left(a\right)$ is continuous,
it suffices to see that $z\mapsto\Gamma_{\alpha,z}\left(b\right)$
is continuous for each $b$ in the $*$-subalgebra generated by $\left\{ T_{e}:e\in\Lambda^{\omega^{*}}\right\} \cup\left\{ T_{v}:v\in\Lambda_{0}\right\} $
and apply a standard $\varepsilon/3$ argument. Since this map is
continuous when $b$ is one of these generators and $\Gamma_{\alpha,z}$
is an automorphism, indeed each such $z\mapsto\Gamma_{\alpha,z}\left(b\right)$
is continuous.
\end{proof}

\section{$C^{*}$-correspondences}

For the rest of the paper, $\left(\Lambda,d\right)$ will be a fixed
ordinal graph. To each $\alpha\in\mathrm{Ord}$, there is an associated
$C^{*}$-correspondence we construct from the following Hilbert module.
\begin{defn}
\label{def:correspondences}Let $X_{\alpha}$ denote the Hilbert module
over $\mathcal{O}\left(\Lambda_{\alpha}\right)$ which is the completion
of the following module
\[
X_{\alpha}^{\circ}=\left\{ f\in c_{c}\left(\Lambda^{\omega^{\alpha}},\mathcal{O}\left(\Lambda_{\alpha}\right)\right):T_{s\left(e\right)}f\left(e\right)=f\left(e\right)\text{ for all }e\in\Lambda^{\omega^{\alpha}}\right\} 
\]
with the operations for $a\in\mathcal{O}\left(\Lambda_{\alpha}\right)$
and $x,y\in X_{\alpha}^{\circ}$ defined by
\[
\left(x\cdot a\right)\left(e\right)=x\left(e\right)a
\]
\[
\left\langle x,y\right\rangle =\sum_{e\in\Lambda^{\omega^{\alpha}}}x\left(e\right)^{*}y\left(e\right)
\]
For $e\in\Lambda^{\omega^{\alpha}}$, define $\delta_{e}\in X_{\alpha}^{\circ}$
by
\[
\delta_{e}\left(f\right)=\begin{cases}
T_{s\left(e\right)} & e=f\\
0 & e\not=f
\end{cases}
\]
\end{defn}

If $x\in X_{\alpha}^{\circ}$, then $\left\langle x,x\right\rangle $
is a sum of positive elements which is zero if and only if $x=0$.
Thus $\left\langle \cdot,\cdot\right\rangle $ is an inner product,
and the completion is indeed a Hilbert module. We can form a $C^{*}$-correspondence
by giving $X_{\alpha}$ the following left action by $\mathcal{O}\left(\Lambda_{\alpha}\right)$.
\begin{prop}
\label{prop:left-action}There exists a {*}-homomorphism $\varphi_{\alpha}:\mathcal{O}\left(\Lambda_{\alpha}\right)\rightarrow\mathcal{L}\left(X_{\alpha}\right)$
such that for $e\in\Lambda^{\omega^{\alpha}}$, $g\in\Lambda_{\alpha}$,
and $x\in X_{\alpha}^{\circ}$,
\[
\left(\varphi_{\alpha}\left(T_{g}\right)x\right)\left(e\right)=\begin{cases}
x\left(g^{-1}e\right) & e\in g\Lambda\\
0 & \text{otherwise}
\end{cases}
\]
\[
\left(\varphi_{\alpha}\left(T_{g}^{*}\right)x\right)\left(e\right)=\begin{cases}
x\left(ge\right) & s\left(g\right)=r\left(e\right)\\
0 & s\left(g\right)\not=r\left(e\right)
\end{cases}
\]
\end{prop}

\begin{proof}
First, note that $\left\langle \varphi_{\alpha}\left(T_{g}\right)x,\varphi_{\alpha}\left(T_{g}\right)x\right\rangle \leq\left\langle x,x\right\rangle $,
so $\varphi_{\alpha}\left(T_{g}\right)$ is bounded on $X_{\alpha}^{\circ}$
and extends to $X_{\alpha}$. By the above formulas, we also have
\[
\varphi_{\alpha}\left(T_{g}\right)\delta_{e}=\begin{cases}
\delta_{ge} & s\left(g\right)=r\left(e\right)\\
0 & \text{otherwise}
\end{cases}
\]
\[
\varphi_{\alpha}\left(T_{g}^{*}\right)\delta_{e}=\begin{cases}
\delta_{g^{-1}e} & e\in g\Lambda\\
0 & \text{otherwise}
\end{cases}
\]
Now we check $\varphi_{\alpha}\left(T_{g}\right)^{*}=\varphi_{\alpha}\left(T_{g}^{*}\right)$.
Since $X_{\alpha}^{\circ}$ is the $\mathcal{O}\left(\Lambda_{\alpha}\right)$-span
of $\left\{ \delta_{e}:e\in\Lambda^{\omega^{\alpha}}\right\} $, it
suffices to check $\left\langle \varphi_{\alpha}\left(T_{g}\right)\delta_{e},\delta_{f}\right\rangle =\left\langle \delta_{e},\varphi_{\alpha}\left(T_{g}^{*}\right)\delta_{f}\right\rangle $
for all $e,f\in\Lambda^{\omega^{\alpha}}$. If $f=ge$, then indeed
\[
\left\langle \varphi_{\alpha}\left(T_{g}\right)\delta_{e},\delta_{f}\right\rangle =\left\langle \delta_{e},\varphi_{\alpha}\left(T_{g}^{*}\right)\delta_{f}\right\rangle =T_{s\left(e\right)}
\]
Otherwise, both inner products are zero. Thus each $\varphi_{\alpha}\left(T_{g}\right)$
is adjointable, and all that remains is to check that $\left\{ \varphi_{\alpha}\left(T_{g}\right):g\in\Lambda_{\alpha}\right\} $
is a Cuntz-Krieger $\Lambda_{\alpha}$-family, which we will do by
verifying the relations in \prettyref{thm:gen-relations}. For relation
(1), we must verify $\varphi_{\alpha}\left(T_{e}\right)^{*}\varphi_{\alpha}\left(T_{e}\right)=\varphi_{\alpha}\left(T_{s(e)}\right)$.
If $g\in s\left(e\right)\Lambda^{\omega^{\alpha}}$, then
\[
\varphi_{\alpha}\left(T_{e}\right)^{*}\varphi_{\alpha}\left(T_{e}\right)\delta_{g}=\varphi_{\alpha}\left(T_{e}\right)^{*}\delta_{eg}=\delta_{g}=\varphi_{\alpha}\left(T_{s(e)}\right)\delta_{g}
\]
and the relation is satisfied. Relation (2) follows, since if $s\left(e\right)=r\left(f\right)$
and $g\in s\left(f\right)\Lambda^{\omega^{\alpha}}$,
\[
\varphi_{\alpha}\left(T_{e}\right)\varphi_{\alpha}\left(T_{f}\right)\delta_{g}=\varphi_{\alpha}\left(T_{e}\right)\delta_{fg}=\delta_{efg}=\varphi_{\alpha}\left(T_{ef}\right)\delta_{g}
\]
Now we verify relation (3). Let $e,f\in\Lambda_{\alpha}$ with $e\Lambda_{\alpha}\cap f\Lambda_{\alpha}=0$.
By taking adjoints if necessary and applying \prettyref{prop:relation-3},
we may assume without loss of generality that $d\left(e\right)\leq d\left(f\right)$.
Also, $\varphi_{\alpha}\left(T_{e}\right)^{*}\varphi_{\alpha}\left(T_{f}\right)$
is non-zero only if there exists $g\in s\left(g\right)\Lambda^{\omega^{\alpha}}$
and $fg\in e\Lambda$, in which case
\[
\varphi_{\alpha}\left(T_{e}\right)^{*}\varphi_{\alpha}\left(T_{f}\right)\delta_{g}=\delta_{e^{-1}fg}
\]
Since $d\left(e\right)\leq d\left(f\right)$, this implies $f\in e\Lambda_{\alpha}\cap f\Lambda_{\alpha}$,
which is a contradiction. Thus $\varphi_{\alpha}\left(T_{e}\right)^{*}\varphi_{\alpha}\left(T_{f}\right)=0$,
and relation (3) is satisfied. Finally, we show relation (4) is satisfied.
Suppose $v\in\Lambda_{0}$ is $\beta$-regular for some $\beta<\alpha$.
Then for $e\in\Lambda^{\omega^{\beta}}$ and $g\in e\Lambda^{\omega^{\alpha}}$,
\[
\varphi_{\alpha}\left(T_{e}\right)\varphi_{\alpha}\left(T_{e}\right)^{*}\delta_{g}=\varphi_{\alpha}\left(T_{e}\right)\delta_{e^{-1}g}=\delta_{g}
\]
If $g\not\in e\Lambda^{\omega^{\alpha}}$, then $\varphi_{\alpha}\left(T_{e}\right)\varphi_{\alpha}\left(T_{e}\right)^{*}\delta_{g}=0$.
Since each $g\in\Lambda^{\omega^{\alpha}}$ belongs to $g\left(\omega^{\beta}\right)\Lambda^{\omega^{\alpha}}$,
we have
\[
\sum_{e\in v\Lambda^{\omega^{\beta}}}\varphi_{\alpha}\left(T_{e}\right)\varphi_{\alpha}\left(T_{e}\right)^{*}=\varphi_{\alpha}\left(T_{v}\right)
\]
and relation (4) is satisfied.
\end{proof}
\begin{example}
\label{exa:graph-example}Suppose $\Lambda=\Lambda_{1}$, i.e. $\Lambda$
is a directed graph. Then $X_{0}$ is a $C^{*}$-correspondence over
$\mathcal{O}\left(\Lambda_{0}\right)$. $\Lambda_{0}$ is the set
of vertices of $\Lambda$, so $\mathcal{O}\left(\Lambda_{0}\right)$
is the $C^{*}$-algebra which is universal for a family of mutually
orthogonal projections, one for each vertex. Hence $\mathcal{O}\left(\Lambda_{0}\right)$
is isomorphic to $c_{0}\left(\Lambda_{0}\right)$. Moreover, $X_{0}^{\circ}$
is spanned by functions $f\in c_{c}\left(\Lambda^{1},c_{0}\left(\Lambda_{0}\right)\right)$
satisfying $f\left(e\right)=T_{s\left(e\right)}f\left(e\right)$.
Thus the support of $f\left(e\right)$ is contained in $\left\{ e\right\} $.
Define a $c_{0}\left(\Lambda_{0}\right)$-correspondence $Y$ as the
closure of $c_{c}\left(\Lambda^{1}\right)$ with the following operations
defined for $f,g\in c_{c}\left(\Lambda^{1}\right)$ and $h\in c_{0}\left(\Lambda_{0}\right)$
\begin{align*}
\left(h\cdot f\right)\left(e\right) & =h\left(r\left(e\right)\right)f\left(e\right)\\
\left(f\cdot h\right)\left(e\right) & =f\left(e\right)h\left(s\left(e\right)\right)\\
\left\langle f,g\right\rangle \left(v\right) & =\sum_{s\left(e\right)=v}\overline{f\left(e\right)}g\left(e\right)
\end{align*}
Define $\psi:X_{0}\rightarrow Y$ by
\[
\psi\left(\delta_{e}\right)\left(f\right)=\begin{cases}
1 & e=f\\
0 & \text{otherwise}
\end{cases}
\]
Applying the condition $f\left(e\right)=T_{s\left(e\right)}f\left(e\right)$,
it is not hard to see that $\psi$ is an isomorphism of $c_{0}\left(\Lambda_{0}\right)$-correspondences.
In particular, $X_{0}$ is isomorphic to the usual $C^{*}$-correspondence
defined for directed graphs, and by \cite[Proposition 3.10]{CP-CONSTRUCT},
$\mathcal{O}\left(X_{0}\right)$ is isomorphic to the graph algebra
$\mathcal{O}\left(\Lambda\right)$.
\end{example}

The correspondences $X_{\alpha}$ are constructed so that there exists
the following representation of $X_{\alpha}$ into $\mathcal{O}\left(\Lambda_{\alpha+1}\right)$.
\begin{prop}
\label{prop:correspondence-representation}There exist maps $\psi_{\alpha}:X_{\alpha}\rightarrow\mathcal{O}\left(\Lambda_{\alpha+1}\right)$
such that for $x\in X_{\alpha}^{\circ}$,
\[
\psi_{\alpha}\left(x\right)=\sum_{e\in\Lambda^{\omega^{\alpha}}}T_{e}\rho_{\alpha}^{\alpha+1}\left(x\left(e\right)\right)
\]
Moreover, $\left(\psi_{\alpha},\rho_{\alpha}^{\alpha+1}\right):X_{\alpha}\rightarrow\mathcal{O}\left(\Lambda_{\alpha+1}\right)$
is a representation of the correspondence $\left(X_{\alpha},\varphi_{\alpha}\right)$.
\end{prop}

\begin{proof}
We define $\psi_{\alpha}$ first on $X_{\alpha}^{\circ}$ by the formula
above. To extend the domain of $\psi_{\alpha}$ to $X_{\alpha}$,
we show $\psi_{\alpha}$ is continuous. If $x,y\in X_{\alpha}^{\circ}$,
then
\[
\left\Vert \psi_{\alpha}\left(x\right)^{*}\psi_{\alpha}\left(y\right)\right\Vert =\left\Vert \sum_{e,f\in\Lambda^{\omega^{\alpha}}}\rho_{\alpha}^{\alpha+1}\left(x\left(e\right)\right)^{*}T_{e}^{*}T_{f}\rho_{\alpha}^{\alpha+1}\left(y\left(f\right)\right)\right\Vert =\left\Vert \sum_{e\in\Lambda^{\omega^{\alpha}}}\rho_{\alpha}\left(x\left(e\right)\right)^{*}T_{s\left(e\right)}\rho_{\alpha}\left(y\left(e\right)\right)\right\Vert 
\]
\[
=\left\Vert \rho_{\alpha}^{\alpha+1}\left(\left\langle x,y\right\rangle \right)\right\Vert \leq\left\Vert \left\langle x,y\right\rangle \right\Vert 
\]
In particular, $\left\Vert \psi_{\alpha}\left(x\right)\right\Vert \leq\left\Vert x\right\Vert $,
and $\psi_{\alpha}$ extends by continuity to a map $\psi_{\alpha}:X_{\alpha}\rightarrow\mathcal{O}\left(\Lambda_{\alpha+1}\right)$.
Now we show $\left(\psi_{\alpha},\rho_{\alpha}^{\alpha+1}\right)$
is a representation. If $x\in X_{\alpha}$ and $a\in\mathcal{O}\left(\Lambda_{\alpha}\right)$,
then
\[
\psi_{\alpha}\left(x\right)\rho_{\alpha}^{\alpha+1}\left(a\right)=\sum_{e\in\Lambda^{\omega^{\alpha}}}T_{e}\rho_{\alpha}^{\alpha+1}\left(x\left(e\right)a\right)=\psi_{\alpha}\left(x\cdot a\right)
\]
\[
\rho_{\alpha}^{\alpha+1}\left(T_{e}\right)\psi_{\alpha}\left(\delta_{f}\right)=\rho_{\alpha}^{\alpha+1}\left(T_{e}\right)T_{f}=\psi_{\alpha}\left(\varphi_{\alpha}\left(T_{e}\right)\delta_{f}\right)
\]
Thus by linearity, $\psi_{\alpha}$ respects the left action.
\end{proof}
These representations of the correspondences will be a valuable tool
for us later. We wish to know when $\left(\psi_{\alpha},\rho_{\alpha}^{\alpha+1}\right)$
is the universal covariant representation for $X_{\alpha}$. In particular,
this would imply $\mathcal{O}\left(\Lambda_{\alpha+1}\right)$ is
isomorphic to the Cuntz-Pimsner algebra $\mathcal{O}\left(X_{\alpha}\right)$.
The following example shows that this isn't always the case.

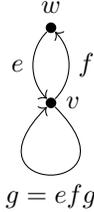
\begin{figure}[H]

\begin{center}
\begin{tikzpicture}[every loop/.style={looseness=40}]
	\draw[->] node[circle, fill, inner sep=0, minimum size=4pt, label=right:$v$] (v) {} edge[in=225, out=315, loop] node[below] {$g=efg$} ();
	\draw (0, 1) node[circle, fill, inner sep=0, minimum size=4pt, label=above:$w$] (w) {};
	\draw[->] (v) edge[in=315, out=45] node[right] {$f$} (w);
	\draw[->] (w) edge[out=225, in=135] node[left] {$e$} (v);
\end{tikzpicture}
\end{center}\caption{A category generated by two objects $v,w$, two morphisms $e,f$,
and a morphism $g$ such that $g=efg$}\label{fig:example1-fig}

\end{figure}

\begin{example}
Consider the category $\Lambda$ pictured in \prettyref{fig:example1-fig};
it is generated by two objects $v$ and $w$, two morphisms $e$ and
$f$, and a morphism $g$ satisfying $g=efg$. $\Lambda$ is an ordinal
graph when we define $d\left(v\right)=d\left(w\right)=0$, $d\left(e\right)=d\left(f\right)=1$,
and $d\left(g\right)=\omega$. Then the only two paths of length $\omega$
are $g$ and $fg$, so $v$ and $w$ are $1$-regular. Since $g$
is the only member of $v\Lambda^{\omega}$, we have $T_{g}T_{g}^{*}=T_{v}$
in $\mathcal{O}\left(\Lambda\right)$. Thus 
\[
T_{ef}=T_{ef}T_{g}T_{g}^{*}=T_{efg}T_{g}^{*}=T_{g}T_{g}^{*}=T_{v}
\]

On the other hand, $\Lambda_{1}$ is a directed graph with a single
cycle $ef$ of length $2$. Let $S_{e},S_{f},S_{v}$, and $S_{w}$
be the generators for $\mathcal{O}\left(\Lambda_{1}\right)$. Then
$\mathcal{O}\left(\Lambda_{1}\right)\cong M_{2}\left(C\left(\mathbb{T}\right)\right)$,
and there is a gauge action $\gamma:\mathbb{T}\rightarrow\mathrm{Aut}\left(\mathcal{O}\left(\Lambda_{1}\right)\right)$
for which $\gamma_{z}\left(S_{e}\right)=zS_{e}$, $\gamma_{z}\left(S_{f}\right)=zS_{f}$,
and $\gamma_{z}\left(S_{v}\right)=S_{v}$. One may directly verify
by calculating in $M_{2}\left(C\left(\mathbb{T}\right)\right)$ that
$S_{e}S_{f}\not=S_{v}$ in $\mathcal{O}\left(\Lambda_{1}\right)$,
but we can also see this using the gauge action. If $S_{e}S_{f}=S_{v}$,
then $\gamma_{z}\left(S_{e}S_{f}\right)=z^{2}S_{e}S_{f}=\gamma_{z}\left(S_{v}\right)=S_{v}$
for all $z\in\mathbb{T}$. In particular, $S_{e}S_{f}=-S_{e}S_{f}$,
which would imply $S_{e}S_{f}=0$. However, a generator $T_{h}$ of
an ordinal graph algebra is never zero, so this is impossible. We
may conclude that $S_{e}S_{f}-S_{v}\in\ker\rho_{1}^{2}\backslash\left\{ 0\right\} $,
and in particular, $\rho_{1}^{2}$ is not injective. Then the representation
$\left(\psi_{1},\rho_{1}^{2}\right)$ is not injective, and therefore
not the universal covariant representation. 

Note that the relation $T_{ef}=T_{v}$ implies $T_{e}^{*}T_{ef}=T_{f}=T_{e}^{*}$,
so $\mathcal{O}\left(\Lambda\right)$ is in fact generated by $T_{f}$
and $T_{g}$. Moreover, the relation $T_{e}T_{f}T_{g}=T_{g}$ follows
from $T_{e}T_{f}=T_{f}^{*}T_{f}=T_{v}$, hence $\mathcal{O}\left(\Lambda\right)$
is isomorphic to $C^{*}\left(F\right)\cong M_{2}\left(C\left(\mathbb{T}\right)\right)$
pictured in \prettyref{fig:example1-fig-1}. We have explicit isomorphisms
from $\mathcal{O}\left(\Lambda_{1}\right)$ and $\mathcal{O}\left(\Lambda\right)$
into $M_{2}\left(C\left(\mathbb{T}\right)\right)$ defined by
\begin{align*}
S_{e} & \mapsto\begin{pmatrix}0 & 0\\
z & 0
\end{pmatrix} & S_{f} & \mapsto\begin{pmatrix}0 & 1\\
0 & 0
\end{pmatrix}\\
T_{e} & \mapsto\begin{pmatrix}0 & 0\\
1 & 0
\end{pmatrix} & T_{f} & \mapsto\begin{pmatrix}0 & 1\\
0 & 0
\end{pmatrix} & T_{g} & \mapsto\begin{pmatrix}z & 0\\
0 & 0
\end{pmatrix}
\end{align*}
where $z$ denotes the identity function in $C\left(\mathbb{T}\right)$.
Under these isomorphisms, $\rho_{1}:\mathcal{O}\left(\Lambda_{1}\right)\rightarrow\mathcal{O}\left(\Lambda\right)$
is the map
\[
\begin{pmatrix}f_{1,1} & f_{1,2}\\
f_{2,1} & f_{2,2}
\end{pmatrix}\mapsto\begin{pmatrix}f_{1,1}\left(1\right)1_{\mathbb{T}} & f_{1,2}\left(1\right)1_{\mathbb{T}}\\
f_{2,1}\left(1\right)1_{\mathbb{T}} & f_{2,2}\left(1\right)1_{\mathbb{T}}
\end{pmatrix}
\]

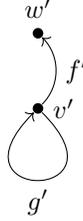
\begin{figure}[H]
\begin{center}
\begin{tikzpicture}[every loop/.style={looseness=40}]
	\draw[->] node[circle, fill, inner sep=0, minimum size=4pt, label=right:$v'$] (v) {} edge[in=225, out=315, loop] node[below] {$g'$} ();
	\draw (0, 1) node[circle, fill, inner sep=0, minimum size=4pt, label=above:$w'$] (w) {};
	\draw[->] (v) edge[in=315, out=45] node[right] {$f'$} (w);
\end{tikzpicture}
\end{center}\caption{A directed graph $F$ with vertices $v'$, $w'$ and edges $f'$,
$g'$}\label{fig:example1-fig-1}
\end{figure}
\end{example}

In the previous example, $\rho_{1}$ is not injective because of the
existence of the path $g$, which is a path of length $\omega$ with
$1$-regular range that factors through the cycle $ef$. More generally,
paths of length $\omega^{\alpha+1}$ factoring through a non-trivial
cycle and ending with a $\alpha+1$-regular vertex obstruct injectivity
of $\rho_{\alpha+1}$, as we show in the following result.
\begin{prop}
\label{prop:hom-injective}$\rho_{\alpha+1}:\mathcal{O}\left(\Lambda_{\alpha+1}\right)\rightarrow\mathcal{O}\left(\Lambda\right)$
is injective only if there are no paths $f\in\Lambda^{\omega^{\alpha+1}}$
such that $r\left(f\right)$ is $\alpha+1$-regular and $f=gf$ for
some $g\in\Lambda_{\alpha+1}\backslash\Lambda_{\alpha}$.
\end{prop}

\begin{proof}
Let $f\in\Lambda^{\omega^{\alpha+1}}$ and $g\in\Lambda_{\alpha+1}\backslash\Lambda_{\alpha}$
such that $r\left(f\right)$ is $\alpha+1$-regular and $f=gf$. Then
in particular, $g$ is a cycle based at an $\alpha+1$-regular vertex
$v=r\left(f\right)$. Since $v$ is $\alpha+1$-row-finite, for each
$h\in v\Lambda^{\omega^{\alpha+1}}$, the set $\left\{ h,gh,g^{2}h,\dots\right\} $
is finite. In particular, there exist distinct $m_{h},k_{h}\in\mathbb{N}$
with $m_{h}>k_{h}$ and $g^{k_{h}}h=g^{m_{h}}h$. Cancelling on the
left, we see $g^{n_{h}}h=h$ for $n_{h}=m_{h}-k_{h}$. Because $v\Lambda^{\omega^{\alpha+1}}$
is finite, we may choose a common multiple $n$ of $\left\{ n_{h}:h\in\Lambda^{\omega^{\alpha+1}}\right\} $.
Then for each $h\in\Lambda^{\omega^{\alpha+1}}$, we have $g^{n}h=h$.
By relation (4), this implies
\[
T_{g}^{n}=T_{g}^{n}T_{v}=T_{g}^{n}\sum_{h\in v\Lambda^{\omega^{\alpha}}}T_{h}T_{h}^{*}=\sum_{h\in v\Lambda^{\omega^{\alpha}}}T_{h}T_{h}^{*}=T_{v}
\]

Since $g\in\Lambda_{\alpha+1}\backslash\Lambda_{\alpha}$, there is
some $m\in\mathbb{N}$ such that $m>0$ and $d\left(g\right)\in\left[\omega^{\alpha}\cdot m,\omega^{\alpha}\cdot\left(m+1\right)\right)$
by the Cantor normal form. Let $\left\{ S_{e}:e\in\Lambda_{\alpha+1}\right\} $
be the generators in $\mathcal{O}\left(\Lambda_{\alpha+1}\right)$.
By \prettyref{lem:alg-auts}, there are automorphisms $\Gamma_{\alpha,z}\in\mathrm{Aut}\left(\mathcal{O}\left(\Lambda_{\alpha+1}\right)\right)$
for each $z\in\mathbb{T}$ such that $\Gamma_{\alpha,z}\left(S_{g}\right)=z^{m}S_{g}$
and $\Gamma_{\alpha,z}\left(S_{v}\right)=S_{v}$. We note that $S_{g}^{n}\not=S_{v}$
in $\mathcal{O}\left(\Lambda_{\alpha+1}\right)$, otherwise $\Gamma_{\alpha,z}\left(S_{g}^{n}\right)=z^{nm}S_{g}=\Gamma_{\alpha,z}\left(S_{v}\right)=S_{v}$
for all $z\in\mathbb{T}$. Selecting a $2nm$-th root of unity, this
would imply $S_{g}=-S_{g}$, or $S_{g}=0$, which is not the case.
Hence $S_{g}^{n}-S_{v}\in\ker\rho_{\alpha+1}\backslash\left\{ 0\right\} $,
and $\rho_{\alpha+1}$ is not injective.
\end{proof}
For the rest of this section, we focus on developing a characterization
of the compact operators $\mathcal{K}\left(X_{\alpha}\right)$ in
the image of $\varphi_{\alpha}$ that we will need later. We wish
to construct a representation of $\sigma$ of $\mathcal{O}\left(\Lambda_{\alpha}\right)$
on a Hilbert space $H$ so that there are vectors $\xi_{e}$ for $e\in\Lambda^{\omega^{\alpha}}$
and $\sigma\left(T_{e}\right)\xi_{f}=\xi_{ef}$ if $f\in s\left(e\right)\Lambda^{\omega^{\alpha}}$.
Since this resembles the fact that $\varphi_{\alpha}\left(T_{e}\right)\delta_{f}=\delta_{ef}$
when $f\in s\left(e\right)\Lambda^{\omega^{\alpha}}$, we might expect
that $\varphi_{\alpha}\left(a\right)\in\mathcal{K}\left(X_{\alpha}\right)$
iff $\sigma\left(a\right)\in K\left(H\right)$. Below we make all
of this precise and show this is the case.
\begin{prop}
\label{prop:left-action-factors}For each $v\in\Lambda_{0}$, let
$H_{\alpha,v}=\ell^{2}\left(\Lambda^{\omega^{\alpha}}v\right)$. There
is a faithful representation $\kappa_{\alpha}:\bigoplus_{v\in\Lambda_{0}}^{\ell^{\infty}}B\left(H_{\alpha,v}\right)\rightarrow\mathcal{L}\left(X_{\alpha}\right)$
and a representation $\sigma_{\alpha}:\mathcal{O}\left(\Lambda_{\alpha}\right)\rightarrow\bigoplus_{v\in\Lambda_{0}}^{\ell^{\infty}}B\left(H_{\alpha,v}\right)$
such that
\begin{enumerate}
\item For every $f,g\in\Lambda^{\omega^{\alpha}}v$,
\[
\kappa_{\alpha}\left(\theta_{\xi_{f},\xi_{g}}\right)=\theta_{\delta_{f},\delta_{g}}
\]
where $\xi_{f},\xi_{g}$ are the canonical orthonormal basis vectors
of $H_{\alpha,v}$ and $\theta_{\xi_{f},\xi_{g}}$ is the rank-one
operator mapping $\xi_{g}$ to $\xi_{f}$. In particular, $\kappa_{\alpha}\left(K\left(H_{\alpha,v}\right)\right)\subseteq\mathcal{K}\left(X_{\alpha}\right)$.
\item The diagram commutes:
\[\begin{tikzcd}
	{\mathcal{O}\left(\Lambda_\alpha\right)} && {\bigoplus_{v\in \Lambda_0}^{\ell^\infty}B\left(H_{\alpha,v}\right)} \\
	\\
	&& {\mathcal{L}\left(X_\alpha\right)}
	\arrow["{\sigma_\alpha}", from=1-1, to=1-3]
	\arrow["{\varphi_\alpha}"', from=1-1, to=3-3]
	\arrow["{\kappa_\alpha}", from=1-3, to=3-3]
\end{tikzcd}\]
\end{enumerate}
\end{prop}

\begin{proof}
First we construct $\kappa_{\alpha,v}:K\left(H_{\alpha,v}\right)\rightarrow\mathcal{L}\left(X_{\alpha}\right)$,
and then we extend using multipliers to build $\kappa_{\alpha}$.
Recall that $K\left(H_{\alpha,v}\right)$ is universal for (rank one)
operators $\left\{ \theta_{\xi_{f},\xi_{g}}:f,g\in\Lambda^{\omega^{\alpha}}v\right\} $
satisfying the relations
\begin{align*}
\theta_{\xi_{f},\xi_{g}}\theta_{\xi_{h},\xi_{k}} & =\begin{cases}
\theta_{\xi_{f},\xi_{k}} & g=h\\
0 & g\not=h
\end{cases}\\
\theta_{\xi_{f},\xi_{g}}^{*} & =\theta_{\xi_{g},\xi_{f}}
\end{align*}
Therefore it suffices to identify such a collection of operators in
$\mathcal{L}\left(X_{\alpha}\right)$. Notice that the operators $\theta_{\delta_{f},\delta_{g}}$
for $f,g\in\Lambda^{\omega^{\alpha}}v$ satisfy the necessary relations,
as we may calculate: 
\[
\theta_{\delta_{f},\delta_{g}}\theta_{\delta_{h},\delta_{k}}x=\delta_{f}\left\langle \delta_{g},\delta_{h}\left\langle \delta_{k},x\right\rangle \right\rangle =\delta_{f}\left\langle \delta_{g},\delta_{h}\right\rangle \left\langle \delta_{k},x\right\rangle 
\]
If $g=h$, then $\left\langle \delta_{g},\delta_{h}\right\rangle =T_{v}\in\mathcal{O}\left(\Lambda_{\alpha}\right)$,
and since $\delta_{f}T_{v}=\delta_{f}$, we conclude $\theta_{\delta_{f},\delta_{g}}\theta_{\delta_{h},\delta_{k}}=\theta_{\delta_{f},\delta_{k}}$.
Otherwise $\left\langle \delta_{g},\delta_{h}\right\rangle =0$, and
$\theta_{\delta_{f},\delta_{g}}\theta_{\delta_{h},\delta_{k}}=0$.
Thus we have constructed $\kappa_{\alpha,v}:K\left(H_{\alpha,v}\right)\rightarrow\mathcal{L}\left(X_{\alpha}\right)$,
which is automatically injective due to simplicity of $K\left(H_{\alpha,v}\right)$,
satisfying $\kappa_{\alpha,v}\left(\theta_{\xi_{f},\xi_{g}}\right)=\theta_{\delta_{f},\delta_{g}}$. 

If $\kappa_{\alpha,v}$ were non-degenerate, then we could immediately
extend $\kappa_{\alpha,v}$ to a faithful representation of $B\left(H_{\alpha,v}\right)$;
however, $\kappa_{\alpha,v}$ is usually degenerate. Instead, note
that for each $f,g,h,k\in\Lambda^{\omega^{\alpha}}$ with $s\left(g\right)\not=s\left(h\right)$,
$\theta_{\delta_{f},\delta_{g}}\theta_{\delta_{h},\delta_{k}}=0$
since $\left\langle \delta_{g},\delta_{h}\right\rangle =0$. Therefore,
we have an injective $*$-homomorphism
\[
\kappa_{\alpha}:\bigoplus_{w\in\Lambda_{0}}K\left(H_{\alpha,w}\right)\rightarrow\mathcal{L}\left(X_{\alpha}\right)
\]
given by $\kappa_{\alpha,v}$ on each summand. We claim $\kappa_{\alpha}$
is non-degenerate. Note that for each $f,g\in\Lambda^{\omega^{\alpha}}v$
and $a,b\in\mathcal{O}\left(\Lambda_{\alpha}\right)$,
\[
\theta_{\delta_{f},\delta_{f}}\theta_{\delta_{f}a,\delta_{g}b}=\theta_{\delta_{f}a,\delta_{g}b}
\]
Since $X_{\alpha}^{\circ}$ is dense in $X_{\alpha}$, the net of
elements $s_{F}$ for finite $F\subseteq\Lambda^{\omega^{\alpha}}$
defined by
\[
s_{F}=\sum_{f\in F}\theta_{\delta_{f},\delta_{f}}
\]
is an approximate identity for $\mathcal{K}\left(X_{\alpha}\right)$
contained in the image of $\kappa_{\alpha}$, proving non-degeneracy.
As $\kappa_{\alpha}$ is injective, the extension to multipliers
\[
\kappa_{\alpha}:\bigoplus_{w\in\Lambda_{0}}^{\ell^{\infty}}B\left(H_{\alpha,w}\right)\rightarrow\mathcal{L}\left(X_{\alpha}\right)
\]
 is also injective.

Now we construct representations $\sigma_{\alpha,v}:\mathcal{O}\left(\Lambda_{\alpha}\right)\rightarrow B\left(H_{\alpha,v}\right)$.
For each $e\in\Lambda_{\alpha}$ and $f\in\Lambda^{\omega^{\alpha}}$,
define
\[
\sigma_{\alpha,v}\left(T_{e}\right)\xi_{f}=\begin{cases}
\xi_{ef} & s\left(e\right)=r\left(f\right)\\
0 & s\left(e\right)\not=r\left(f\right)
\end{cases}
\]
Then it suffices to verify the relations in \prettyref{thm:gen-relations}.
The calculations are virtually identical to \prettyref{prop:left-action},
so we omit them this time. We then define $\sigma_{\alpha}=\bigoplus_{v\in\Lambda_{0}}^{\ell^{\infty}}\sigma_{\alpha,v}$.

All that is left is to verify the equality in statement (2). Since
$\kappa_{\alpha}\circ\sigma_{\alpha}$ and $\varphi_{\alpha}$ are
$*$-homomorphisms, we need only to check equality for $a=T_{e}$
with $e\in\Lambda_{\alpha}$ to show $\left(\kappa_{\alpha}\circ\sigma_{\alpha}\right)\left(a\right)=\varphi_{\alpha}\left(a\right)$.
Let $f\in\Lambda^{\omega^{\alpha}}$ and $e\in\Lambda_{\alpha}$ be
fixed. Then if $s\left(e\right)=r\left(f\right)$,
\begin{align*}
\left(\kappa_{\alpha}\circ\sigma_{\alpha}\right)\left(T_{e}\right)\delta_{f} & =\left(\kappa_{\alpha}\circ\sigma_{\alpha}\right)\left(T_{e}\right)\theta_{\delta_{f},\delta_{f}}\delta_{f}\\
 & =\kappa_{\alpha}\left(\sigma_{\alpha}\left(T_{e}\right)\theta_{\xi_{f},\xi_{f}}\right)\delta_{f}\\
 & =\kappa_{\alpha}\left(\theta_{\xi_{ef},\xi_{f}}\right)\delta_{f}\\
 & =\delta_{ef}\\
 & =\varphi_{\alpha}\left(T_{e}\right)\delta_{f}
\end{align*}
If $s\left(e\right)\not=r\left(f\right)$, then similarly $\left(\kappa_{\alpha}\circ\sigma_{\alpha}\right)\left(T_{e}\right)\delta_{f}=0=\varphi_{\alpha}\left(T_{e}\right)\delta_{f}$.
\end{proof}
\begin{lem}
\label{lem:left-action-isometry}For each $v\in\Lambda_{0}$, define
$C_{v}=\overline{\mathrm{span}}_{\mathbb{C}}\left\{ \delta_{f}:f\in\Lambda^{\omega^{\alpha}}v\right\} $.
For each $\eta\in C_{v}$, we have
\begin{equation}
\left\langle \eta,\eta\right\rangle =\left\Vert \eta\right\Vert ^{2}T_{v}\label{eq:cv_norm}
\end{equation}
Moreover, there exists an isometry $j_{\alpha,v}:C_{v}\rightarrow H_{\alpha,v}$
defined by $j_{\alpha,v}\left(\delta_{f}\right)=\xi_{f}$ intertwining
$\varphi_{\alpha}$ and the representation $\sigma_{\alpha}$ defined
in \prettyref{prop:left-action-factors}:
\begin{equation}
j_{\alpha,v}\left(\varphi_{\alpha}\left(a\right)\eta\right)=\sigma_{\alpha}\left(a\right)j_{\alpha,v}\left(\eta\right)\label{eq:j_intertwines}
\end{equation}
\end{lem}

\begin{proof}
We do our calculations in $\mathrm{span}_{\mathbb{C}}\left\{ \delta_{f}:f\in\Lambda^{\omega^{\alpha}}v\right\} $,
and then the identities follow by continuity. Let $\eta$ in the algebraic
span be given. Then for some $\lambda_{f}\in\mathbb{C}$, we have
\begin{align*}
\left\langle \eta,\eta\right\rangle  & =\left\langle \sum_{f\in\Lambda^{\omega^{\alpha}}v}\lambda_{f}\delta_{f},\sum_{f\in\Lambda^{\omega^{\alpha}}v}\lambda_{f}\delta_{f}\right\rangle \\
 & =\sum_{f,g\in\Lambda^{\omega^{\alpha}}v}\overline{\lambda_{f}}\lambda_{g}\left\langle \delta_{f},\delta_{g}\right\rangle \\
 & =\sum_{f\in\Lambda^{\omega^{\alpha}}v}\left|\lambda_{f}\right|^{2}\left\langle \delta_{f},\delta_{f}\right\rangle \\
 & =\sum_{f\in\Lambda^{\omega^{\alpha}}v}\left|\lambda_{f}\right|^{2}T_{v}\\
 & =\left\Vert \eta\right\Vert ^{2}T_{v}
\end{align*}
It follows that
\[
\left\Vert j_{\alpha,v}\left(\eta\right)\right\Vert ^{2}=\sum_{f\in\Lambda^{\omega^{\alpha}}}\left|\lambda_{f}\right|^{2}=\left\Vert \eta\right\Vert ^{2}
\]
and hence $j_{\alpha,v}$ is an isometry. By \prettyref{prop:ordgraph-closed-span},
it suffices to verify \prettyref{eq:j_intertwines} in the case $a=T_{g}T_{h}^{*}$
and $\eta=\delta_{f}$ for some $g,h\in\Lambda_{\alpha}$ and $f\in\Lambda^{\omega^{\alpha}}v$.
If $s\left(g\right)=s\left(h\right)$ and $f\in h\Lambda$, then
\begin{align*}
j_{\alpha,v}\left(\varphi_{\alpha}\left(T_{g}T_{h}^{*}\right)\delta_{f}\right) & =j_{\alpha,v}\left(\delta_{gh^{-1}f}\right)\\
 & =\xi_{gh^{-1}f}\\
 & =\sigma_{\alpha}\left(T_{g}T_{h}^{*}\right)j_{\alpha,v}\left(\delta_{f}\right)
\end{align*}
\end{proof}
We will frequently employ the following results when we wish to use
or show $\varphi_{\alpha}\left(a\right)\in\mathcal{K}\left(X_{\alpha}\right)$.
\begin{cor}[Criteria for compactness]
\label{cor:compactness}For each $v\in\Lambda_{0}$, let $Q_{v}$
be the projection of $\bigoplus_{w\in\Lambda_{0}}^{\ell^{\infty}}B\left(H_{\alpha,w}\right)$
onto the $B\left(H_{\alpha,v}\right)$ summand and $\sigma_{\alpha}$
be the representation defined in \prettyref{prop:left-action-factors}.
Define a projection $P_{v}\in\mathcal{L}\left(X_{\alpha}\right)$
by
\[
P_{v}\delta_{f}=\begin{cases}
\delta_{f} & s\left(f\right)=v\\
0 & s\left(f\right)\not=v
\end{cases}
\]
Then for each $a\in\mathcal{O}\left(\Lambda_{\alpha}\right)$ we have
the following:
\begin{enumerate}
\item For all $v\in\Lambda_{0}$, $\varphi_{\alpha}\left(a\right)P_{v}\in\mathcal{K}\left(X_{\alpha}\right)$
if and only if $\sigma_{\alpha}\left(a\right)Q_{v}\in K\left(H_{\alpha,v}\right)$.
\item $\varphi_{\alpha}\left(a\right)\in\mathcal{K}\left(X_{\alpha}\right)$
if and only if for each $v\in\Lambda_{0}$, $\varphi_{\alpha}\left(a\right)P_{v}\in\mathcal{K}\left(X_{\alpha}\right)$
and the map $v\mapsto\left\Vert \varphi_{\alpha}\left(a\right)P_{v}\right\Vert $
is in $c_{0}\left(\Lambda_{0}\right)$.
\item If $T=\varphi_{\alpha}\left(a\right)$ or $T=\varphi_{\alpha}\left(a\right)P_{v}$,
$T\in\mathcal{K}\left(X_{\alpha}\right)$ if and only if for every
$\varepsilon>0$ there exists finite $F\subseteq\Lambda^{\omega^{\alpha}}$such
that if $\eta\in\overline{\mathrm{span}}_{\mathbb{C}}\left\{ \delta_{f}:f\in\Lambda^{\omega^{\alpha}}\backslash F\right\} $
and $\left\Vert \eta\right\Vert \leq1$, $\left\Vert \varphi_{\alpha}\left(a\right)\eta\right\Vert <\varepsilon$.
\end{enumerate}
\end{cor}

\begin{rem}
Later it will be important that $\eta$ in part (3) belongs specifically
to the $\mathbb{C}$-span of the vectors $\left\{ \delta_{f}:f\in\Lambda^{\omega^{\alpha}}\backslash F\right\} $,
not the $\mathcal{O}\left(\Lambda_{\alpha}\right)$-span of such vectors.
\end{rem}

\begin{proof}
We begin by proving (1). Suppose $\sigma_{\alpha}\left(a\right)Q_{v}\in K\left(H_{\alpha,v}\right)$.
Then by \prettyref{prop:left-action-factors}, $\kappa_{\alpha}\left(K\left(H_{\alpha,v}\right)\right)\subseteq\mathcal{K}\left(X_{\alpha}\right)$,
and hence
\[
\varphi_{\alpha}\left(a\right)P_{v}=\kappa_{\alpha,v}\left(\sigma_{\alpha,v}\left(a\right)Q_{v}\right)\in\mathcal{K}\left(X_{\alpha}\right)
\]
For the other direction, suppose $\varphi_{\alpha}\left(a\right)P_{v}\in\mathcal{K}\left(X_{\alpha}\right)$
and let $\varepsilon>0$. Since $X_{\alpha}^{\circ}$ is dense in
$X_{\alpha}$, we may choose finite rank 
\[
S=\sum_{k=1}^{n}\theta_{\delta_{f_{k}}a_{k},\delta_{g_{k}}b_{k}}
\]
such that $\left\Vert \varphi_{\alpha}\left(a\right)P_{v}-S\right\Vert <\varepsilon$.
Define $F=\left\{ g_{k}:1\leq k\leq n,s\left(g_{k}\right)=v\right\} $
and note that for $\eta\in\overline{\mathrm{span}}_{\mathbb{C}}\left\{ \delta_{f}:f\in\Lambda^{\omega^{\alpha}}v\backslash F\right\} $,
$S\eta=0$. Thus if $\left\Vert \eta\right\Vert \leq1$, then
\[
\left\Vert \varphi_{\alpha}\left(a\right)\eta\right\Vert =\left\Vert \left(\varphi_{\alpha}\left(a\right)-S\right)\eta\right\Vert <\varepsilon
\]
Hence we may apply \prettyref{eq:j_intertwines} to see
\[
\left\Vert \sigma_{\alpha}\left(a\right)j_{\alpha,v}\left(\eta\right)\right\Vert <\varepsilon
\]
Therefore, we have proven that for all $\varepsilon>0$ there exists
finite $F\subseteq\Lambda^{\omega^{\alpha}}v$ such that if $\xi\in\overline{\mathrm{span}}\left\{ \xi_{f}:f\in\Lambda^{\omega^{\alpha}}v\backslash F\right\} $
with $\left\Vert \xi\right\Vert \leq1$, $\left\Vert \sigma_{\alpha,v}\left(a\right)\xi\right\Vert <\varepsilon$.
This is equivalent to $\sigma_{\alpha}\left(a\right)Q_{v}\in K\left(H_{\alpha,v}\right)$.

For (2), suppose first $\varphi_{\alpha}\left(a\right)\in\mathcal{K}\left(X_{\alpha}\right)$.
Then clearly $\varphi_{\alpha}\left(a\right)P_{v}\in\mathcal{K}\left(X_{\alpha}\right)$,
so all we must show is $v\mapsto\left\Vert \varphi_{\alpha}\left(a\right)P_{v}\right\Vert $
is a $c_{0}\left(\Lambda_{0}\right)$ function. Let $\varepsilon>0$
be given, and since $\varphi_{\alpha}\left(a\right)\in\mathcal{K}\left(X_{\alpha}\right)$
and $X_{\alpha}^{\circ}$ is dense in $X_{\alpha}$, select finite
rank
\[
S=\sum_{k=1}^{n}\theta_{\delta_{f_{k}a_{k}},\delta_{g_{k}b_{k}}}
\]
such that $\left\Vert \varphi_{\alpha}\left(a\right)-S\right\Vert <\varepsilon$.
Define finite $F=\left\{ s\left(g_{k}\right):1\leq k\leq n\right\} \subseteq\Lambda_{0}$,
and note that for all $1\leq k\leq n$ and $f\in\Lambda^{\omega^{\alpha}}$
with $s\left(f\right)\not\in F$, $f\not=g_{k}$, and hence $\left\langle \delta_{g_{k}},\delta_{f}\right\rangle =0$.
Therefore, if $v\in\Lambda_{0}\backslash F$, then $SP_{v}=0$, and
hence
\[
\left\Vert \varphi_{\alpha}\left(a\right)P_{v}\right\Vert =\left\Vert \left(\varphi_{\alpha}\left(a\right)-S\right)P_{v}\right\Vert <\varepsilon
\]
This proves one direction for statement (2). For the other direction,
suppose each $\varphi_{\alpha}\left(a\right)P_{v}$ is compact and
$v\mapsto\left\Vert \varphi_{\alpha}\left(a\right)P_{v}\right\Vert $
is $c_{0}\left(\Lambda_{0}\right)$. Then injectivity of $\kappa_{\alpha}$
from \prettyref{prop:left-action-factors} implies $\left\Vert \varphi_{\alpha}\left(a\right)P_{v}\right\Vert =\left\Vert \kappa_{\alpha}\left(\sigma_{\alpha}\left(a\right)Q_{v}\right)\right\Vert =\left\Vert \sigma_{\alpha}\left(a\right)Q_{v}\right\Vert $,
and hence $\sigma_{\alpha}\left(a\right)\in\bigoplus_{v\in\Lambda_{0}}K\left(H_{\alpha,v}\right)$.
Again applying \prettyref{prop:left-action-factors}, we have $\varphi_{\alpha}\left(a\right)=\left(\kappa_{\alpha}\circ\sigma_{\alpha}\right)\left(a\right)\in\mathcal{K}\left(X_{\alpha}\right)$.

To prove (3), first suppose $\varphi_{\alpha}\left(a\right)\in\mathcal{K}\left(X_{\alpha}\right)$
(the proof for the case $T=\varphi_{\alpha}\left(a\right)P_{v}$ is
exactly the same). Let $\varepsilon>0$, and note that by (2) there
exists finite $F\subseteq\Lambda_{0}$ such that for all $v\in\Lambda_{0}\backslash F$,
$\left\Vert \varphi_{\alpha}\left(a\right)P_{v}\right\Vert ^{2}<\varepsilon$.
By (1), each $\sigma_{\alpha}\left(a\right)Q_{v}$ is compact, so
by \prettyref{eq:j_intertwines}, for each $v\in F$ there exists
finite $G_{v}\subseteq\Lambda^{\omega^{\alpha}}v$ such that for all
$\eta\in\overline{\mathrm{span}}_{\mathbb{C}}\left\{ \delta_{f}:f\in\Lambda^{\omega^{\alpha}}v\backslash G_{v}\right\} $
with $\left\Vert \eta\right\Vert \leq1$, $\left\Vert \varphi_{\alpha}\left(a\right)\eta\right\Vert ^{2}=\left\Vert \varphi_{\alpha}\left(a\right)P_{v}\eta\right\Vert ^{2}<\varepsilon$.
Define $G=\cup_{v\in F}G_{v}$, which is finite. Then if $\eta\in\mathrm{span}_{\mathbb{C}}\left\{ \delta_{f}:f\in\Lambda^{\omega^{\alpha}}\backslash G\right\} $
with $\left\Vert \eta\right\Vert \leq1$,
\begin{align*}
\left\Vert \varphi_{\alpha}\left(a\right)\eta\right\Vert ^{2} & =\left\Vert \left\langle \sum_{v\in\Lambda_{0}}\varphi_{\alpha}\left(a\right)P_{v}\eta,\sum_{v\in\Lambda_{0}}\varphi_{\alpha}\left(a\right)P_{v}\eta\right\rangle \right\Vert \\
 & =\left\Vert \left\langle \sum_{v\in\Lambda_{0}}P_{v}\varphi_{\alpha}\left(a\right)P_{v}\eta,\sum_{v\in\Lambda_{0}}P_{v}\varphi_{\alpha}\left(a\right)P_{v}\eta\right\rangle \right\Vert \\
 & =\left\Vert \sum_{v\in\Lambda_{0}}\left\langle \varphi_{\alpha}\left(a\right)P_{v}\eta,\varphi_{\alpha}\left(a\right)P_{v}\eta\right\rangle \right\Vert \\
 & =\left\Vert \sum_{v\in\Lambda_{0}}\left\Vert \varphi_{\alpha}\left(a\right)P_{v}\eta\right\Vert ^{2}T_{v}\right\Vert  & \text{by \prettyref{eq:cv_norm}}\\
 & =\sup_{v\in\Lambda_{0}}\left\Vert \varphi_{\alpha}\left(a\right)P_{v}\eta\right\Vert ^{2}
\end{align*}
For each $v\in\Lambda_{0}$, either $v\not\in F$, in which case $\left\Vert \varphi_{\alpha}\left(a\right)P_{v}\eta\right\Vert ^{2}\leq\left\Vert \varphi_{\alpha}\left(a\right)P_{v}\right\Vert ^{2}<\varepsilon$,
or $v\in F$. For $v\in F$, $P_{v}\eta\in\mathrm{span}_{\mathbb{C}}\left\{ \delta_{f}:f\in\Lambda^{\omega^{\alpha}}v\backslash G_{v}\right\} $,
so $\left\Vert \varphi_{\alpha}\left(a\right)P_{v}\eta\right\Vert ^{2}<\varepsilon$.
In any case, we see $\left\Vert \varphi_{\alpha}\left(a\right)\eta\right\Vert ^{2}<\varepsilon$.
To prove the converse, suppose $\varepsilon>0$ and there exists finite
$F\subseteq\Lambda^{\omega^{\alpha}}$ such that if $\eta\in\overline{\mathrm{span}}_{\mathbb{C}}\left\{ \delta_{f}:f\in\Lambda^{\omega^{\alpha}}\backslash F\right\} $
and $\left\Vert \eta\right\Vert \leq1$, $\left\Vert \varphi_{\alpha}\left(a\right)\eta\right\Vert <\varepsilon$.
If $\eta\in\overline{\mathrm{span}}_{\mathbb{C}}\left\{ \delta_{f}:\Lambda^{\omega^{\alpha}}v\backslash F\right\} $
with $\left\Vert \eta\right\Vert \leq1$, then $\left\Vert \varphi_{\alpha}\left(a\right)\eta\right\Vert =\left\Vert \sigma_{\alpha}\left(a\right)Q_{v}j_{\alpha,v}\left(\eta\right)\right\Vert <\varepsilon$,
so by (1) each $\varphi_{\alpha}\left(a\right)P_{v}$ is compact.
Define finite $G=\left\{ s\left(f\right):f\in F\right\} $. Then for
each $v\in\Lambda_{0}\backslash G$ and $\eta\in\overline{\mathrm{span}}_{\mathbb{C}}\left\{ \delta_{f}:f\in\Lambda^{\omega^{\alpha}}v\right\} $
with $\left\Vert \eta\right\Vert \leq1$, $\left\Vert \varphi_{\alpha}\left(a\right)\eta\right\Vert <\varepsilon$,
so $\left\Vert \varphi_{\alpha}\left(a\right)P_{v}\right\Vert \leq\varepsilon$,
and by (2), $\varphi_{\alpha}\left(a\right)\in\mathcal{K}\left(X_{\alpha}\right)$.
\end{proof}
\begin{cor}
\label{cor:compact-loop}Suppose $\alpha<\zeta$, $a\in\mathcal{O}\left(\Lambda_{\alpha+1}\right)$,
$\left(\varphi_{\zeta}\circ\rho_{\alpha+1}^{\zeta}\right)\left(a\right)\in\mathcal{K}\left(X_{\zeta}\right)$,
$h\in\Lambda_{\alpha+1}\backslash\Lambda_{\alpha}$ with $s\left(h\right)=r\left(h\right)$,
and $g\in\bigcap_{n\in\mathbb{N}}h^{n}\Lambda^{\omega^{\zeta}}$.
Then $\left(\varphi_{\zeta}\circ\rho_{\alpha+1}^{\zeta}\right)\left(a\right)\delta_{g}=0$
or there exists $n>0$ such that $g=h^{n}g$.
\end{cor}

\begin{proof}
Suppose $g\in\bigcap_{n\in\mathbb{N}}h^{n}\Lambda^{\omega^{\zeta}}$,
and let $a\in\mathcal{O}\left(\Lambda_{\alpha+1}\right)$ be a finite
sum of the form
\[
a=\sum_{p,q}\lambda_{p,q}T_{p}T_{q}^{*}
\]
for $p,q\in\Lambda_{\alpha+1}$ with $s\left(p\right)=s\left(q\right)$.
Then for $g\in\Lambda^{\omega^{\zeta}}$ with $s\left(g\right)=v$,
\begin{align*}
\left\Vert \left(\sigma_{\zeta,v}\circ\rho_{\alpha+1}^{\zeta}\right)\left(a\right)\xi_{g}\right\Vert ^{2} & =\left\Vert \sum_{g\in q\Lambda,p,q}\lambda_{p,q}\xi_{pq^{-1}g}\right\Vert ^{2}
\end{align*}
Suppose $p_{1}q_{1}^{-1}g=p_{2}q_{2}^{-1}g$ for indices $p_{1},q_{1},p_{2},q_{2}\in\Lambda_{\alpha+1}$.
Then $g=q_{1}p_{1}^{-1}p_{2}q_{2}^{-1}g$, and since $d\left(h\right)\geq\omega^{\alpha}$,
there exists positive $m\in\mathbb{N}$ such that $d\left(h^{m}\right)=d\left(h\right)\cdot m$
is larger than $d\left(q_{2}\right)+d\left(p_{1}\right)$, and in
particular $u=q_{1}p_{1}^{-1}p_{2}q_{2}^{-1}h^{m}$ is a well-defined
path. If $\beta=d\left(u\right)$, then since $d\left(h\right)\geq\omega^{\alpha}$,
$\beta+d\left(h\right)=d\left(h\right)\cdot k$ for some $k\in\mathbb{N}$.
Hence $g\left(\beta+d\left(h\right)\right)=h^{k}=q_{1}p_{1}^{-1}p_{2}q_{2}^{-1}h^{m+1}$,
and
\begin{align*}
h^{-1}q_{1}p_{1}^{-1}p_{2}q_{2}^{-1}hg & =h^{-1}q_{1}p_{1}^{-1}p_{2}q_{2}^{-1}h^{m+1}hh^{-m-1}g\\
 & =q_{1}p_{1}^{-1}p_{2}q_{2}^{-1}h^{m+1}h^{-m-1}g\\
 & =q_{1}p_{1}^{-1}p_{2}q_{2}^{-1}g\\
 & =g
\end{align*}
Therefore, $p_{1}q_{1}^{-1}hg=p_{2}q_{2}^{-1}hg$. The same equations
yield the reverse implication, so we have $p_{1}q_{1}^{-1}g=p_{2}q_{2}^{-1}g$
if and only if $p_{1}q_{1}^{-1}hg=p_{2}q_{2}^{-1}hg$. In particular,
$\left\Vert \left(\sigma_{\zeta,v}\circ\rho_{\alpha+1}^{\zeta}\right)\left(a\right)\xi_{g}\right\Vert =\left\Vert \left(\sigma_{\zeta,v}\circ\rho_{\alpha+1}^{\zeta}\right)\left(a\right)\xi_{h^{n}g}\right\Vert $
for all $n\in\mathbb{N}$ and $a\in\mathcal{O}\left(\Lambda_{\alpha+1}\right)$
by continuity. Now assume $\left(\varphi_{\zeta}\circ\rho_{\alpha+1}^{\zeta}\right)\left(a\right)\in\mathcal{K}\left(X_{\zeta}\right)$
and that $\varepsilon=\left\Vert \left(\varphi_{\zeta}\circ\rho_{\alpha+1}^{\zeta}\right)\left(a\right)\delta_{g}\right\Vert $
is positive. Choose finite $F\subseteq\Lambda^{\omega^{\zeta}}$ such
that for all $\eta\in\mathrm{span}_{\mathbb{C}}\left\{ \delta_{f}:f\in\Lambda^{\omega^{\zeta}}\backslash F\right\} $
with $\left\Vert \eta\right\Vert \leq1$, $\left\Vert \left(\varphi_{\zeta}\circ\rho_{\alpha+1}^{\zeta}\right)\left(a\right)\eta\right\Vert <\frac{\varepsilon}{2}$.
If $\left\{ h^{m}g:m\in\mathbb{N}\right\} $ is finite, then $h^{m_{1}}g=h^{m_{2}}g$
for some $m_{1}>m_{2}$, from which it follows $h^{m_{1}-m_{2}}g=g$.
Therefore, if no $n>0$ exists such that $h^{n}g=g$, $\left\{ h^{m}g:m\in\mathbb{N}\right\} $
is infinite, and there exists $m\in\mathbb{N}$ such that $h^{m}g\not\in F$.
Then \prettyref{lem:left-action-isometry} and the definition of $F$
imply
\[
\varepsilon=\left\Vert \left(\varphi_{\zeta}\circ\rho_{\alpha+1}^{\zeta}\right)\left(a\right)\delta_{g}\right\Vert =\left\Vert \left(\sigma_{\zeta,v}\circ\rho_{\alpha+1}^{\zeta}\right)\left(a\right)\xi_{g}\right\Vert =\left\Vert \left(\sigma_{\zeta,v}\circ\rho_{\alpha+1}^{\zeta}\right)\left(a\right)\xi_{h^{m}g}\right\Vert <\frac{\varepsilon}{2}
\]
This is a contradiction, so $\varepsilon=0$.
\end{proof}

\section{Main Results}

Our main results are for the case in which $\rho_{\alpha}^{\alpha+1}$
is injective for each $\alpha\in\mathrm{Ord}$. One may guess that
this happens when for each $\alpha$, $\Lambda$ doesn't satisfy the
hypotheses of \prettyref{prop:hom-injective}. We will say such ordinal
graphs satisfy condition (C).
\begin{defn}
\label{def:normal}An ordinal graph $\Lambda$ satisfies \emph{condition
(C)} if for every $\alpha\in\mathrm{Ord}$, $f\in\Lambda^{\omega^{\alpha}}$,
and $g\in\Lambda$ such that $r\left(f\right)$ is $\alpha$-regular,
$gf=f$ implies $d\left(g\right)=0$.
\end{defn}

\begin{rem}
If $\Lambda=\Lambda_{1}$, i.e. $\Lambda$ is a directed graph, then
$\Lambda$ automatically satisfies condition (C). This is because
$gf=f$ implies $d\left(g\right)+d\left(f\right)=d\left(f\right)$,
which for finite $d\left(f\right),d\left(g\right)$, implies $d\left(g\right)=0$.
\end{rem}

\begin{thm}
\label{thm:main-thm}If $\Lambda$ is an ordinal graph satisfying
condition (C), then for each $\alpha\in\mathrm{Ord}$ the following
hold
\begin{enumerate}
\item $\rho_{\alpha}$ is injective
\item $J_{\alpha}$ is the smallest ideal in $\mathcal{O}\left(\Lambda_{\alpha}\right)$
containing $\left\{ T_{v}:v\in\Lambda_{0},v\text{ is }\alpha\text{-regular}\right\} $
\item $\left(\psi_{\alpha},\rho_{\alpha}^{\alpha+1}\right)$ is a covariant
representation of $X_{\alpha}$
\item $\psi_{\alpha}\times\rho_{\alpha}^{\alpha+1}:\mathcal{O}\left(X_{\alpha}\right)\rightarrow\mathcal{O}\left(\Lambda_{\alpha+1}\right)$
is an isomorphism
\end{enumerate}
\end{thm}

We delegate lemmas used for the inductive step in the following proof
to the next section.
\begin{proof}
We prove this by transfinite induction and a base case for $\alpha=0$.
If $\alpha=0$, then $\mathcal{O}\left(\Lambda_{\alpha}\right)\cong c_{0}\left(\Lambda_{0}\right)$.
Since $\rho_{0}\left(T_{v}\right)\not=0$, this implies $\rho_{0}$
is injective. Then (2) through (4) follow from \prettyref{exa:graph-example}
since $\Lambda_{1}$ is a directed graph.

Now suppose we've proven this result for all $\alpha<\delta$. By
\prettyref{lem:injective}, we have that $\rho_{\delta}$ is injective,
and since $\rho_{\delta}=\rho_{\delta+1}\circ\rho_{\delta}^{\delta+1}$,
$\rho_{\delta}^{\delta+1}$ is injective. Then (2) is proven by \prettyref{prop:katsura-ideal-generators}.
For (3), we must show that for each $\alpha\in J_{\delta}$, $\left(\psi_{\delta},\rho_{\delta}^{\delta+1}\right)^{(1)}\left(\varphi_{\delta}\left(a\right)\right)=\rho_{\delta}^{\delta+1}\left(a\right)$.
Define
\[
\mathcal{I}=\left\{ a\in J_{\delta}:\left(\psi_{\delta},\rho_{\delta}^{\delta+1}\right)^{(1)}\left(\varphi_{\delta}\left(a\right)\right)=\rho_{\delta}^{\delta+1}\left(a\right)\right\} 
\]
Then if $a\in\mathcal{I}$ and $b\in\mathcal{O}\left(\Lambda_{\delta}\right)$,
\[
\left(\psi_{\delta},\rho_{\delta}^{\delta+1}\right)^{(1)}\left(\varphi_{\delta}\left(ab\right)\right)=\left(\psi_{\delta},\rho_{\delta}^{\delta+1}\right)^{(1)}\left(\varphi_{\delta}\left(a\right)\right)\rho_{\delta}^{\delta+1}\left(b\right)=\rho_{\delta}^{\delta+1}\left(ab\right)
\]
Hence $\mathcal{I}$ is a closed, two-sided ideal. By (2), it suffices
to prove $T_{v}\in\mathcal{I}$ for each $v\in\Lambda_{0}$ which
is $\delta$-regular in order to show covariance. This is a consequence
of relation 4 in \prettyref{thm:gen-relations}. If we let $\left\{ S_{e}:e\in\Lambda_{\delta+1}\right\} $
be the generators of $\mathcal{O}\left(\Lambda_{\delta+1}\right)$,
then we have
\[
\varphi_{\delta}\left(T_{v}\right)=\sum_{e\in\Lambda^{\omega^{\delta}}}\theta_{\delta_{e},\delta_{e}}
\]
\[
\left(\psi_{\delta},\rho_{\delta}^{\delta+1}\right)^{(1)}\left(\varphi_{\delta}\left(T_{v}\right)\right)=\sum_{e\in\Lambda^{\omega^{\delta}}}\psi_{\delta}\left(\delta_{e}\right)\psi_{\delta}\left(\delta_{e}\right)^{*}=\sum_{e\in\Lambda^{\omega^{\delta}}}S_{e}S_{e}^{*}=S_{v}=\rho_{\delta}^{\delta+1}\left(T_{v}\right)
\]

By covariance, we have a homomorphism $\psi_{\delta}\times\rho_{\delta}^{\delta+1}:\mathcal{O}\left(X_{\delta}\right)\rightarrow\mathcal{O}\left(\Lambda_{\delta+1}\right)$,
which we now show is an isomorphism. Surjectivity is not hard to see;
all we need is that the image contains $\left\{ S_{e}:e\in\Lambda_{\delta+1}\right\} $.
By Cantor's normal form, each $e\in\Lambda_{\delta+1}$ is a finite
composition of paths in $\Lambda^{\omega^{\delta}}$ and $\Lambda_{\delta}$.
If $e\in\Lambda^{\omega^{\delta}}$, then $\psi_{\delta}\left(\delta_{e}\right)=S_{e}$,
and if $e\in\Lambda_{\delta}$, $\rho_{\delta}^{\delta+1}\left(T_{e}\right)=S_{e}$.
For injectivity, note that by (1) $\rho_{\delta}^{\delta+1}$ is injective,
and by \prettyref{lem:alg-auts}, there is an action $\Gamma_{\delta}:\mathbb{T}\rightarrow\mathrm{Aut}\left(\mathcal{O}\left(\Lambda_{\delta+1}\right)\right)$.
If $\gamma_{\delta}$ is the gauge action on $\mathcal{O}\left(X_{\delta}\right)$,
$\left(\iota,\pi\right)$ is the universal covariant representation
for $X_{\delta}$, $z\in\mathbb{T}$, and $e\in\Lambda^{\omega^{\delta}}$,
then
\[
\left(\psi_{\delta}\times\rho_{\delta}^{\delta+1}\circ\gamma_{\delta,z}\right)\left(\iota\left(\delta_{e}\right)\right)=zS_{e}=\Gamma_{\delta,z}\left(S_{e}\right)=\left(\Gamma_{\delta,z}\circ\psi_{\delta}\times\rho_{\delta}^{\delta+1}\right)\left(\iota\left(\delta_{e}\right)\right)
\]
Likewise, if $e\in\Lambda_{\delta}$, then
\[
\left(\psi_{\delta}\times\rho_{\delta}^{\delta+1}\circ\gamma_{\delta,z}\right)\left(\pi\left(\delta_{e}\right)\right)=S_{e}=\Gamma_{\delta,z}\left(S_{e}\right)=\left(\Gamma_{\delta,z}\circ\psi_{\delta}\times\rho_{\delta}^{\delta+1}\right)\left(\pi\left(\delta_{e}\right)\right)
\]
Thus the gauge-invariant uniqueness theorem \cite[Theorem 6.4]{KATIDEAL2}
implies $\psi_{\delta}\times\rho_{\delta}^{\delta+1}$ is injective.
\end{proof}
As a corollary, we will generalize \cite[Theorem 7.15]{ORDGRAPH}
to allow $\Lambda$ to have 1-regular vertices. To do so, we make
the following definitions.
\begin{defn}[{\cite[Definition 7.2]{ORDGRAPH}}]
If $e\in\Lambda^{\omega^{\alpha}\cdot n}$ for some $n\in[0,\omega)$,
then $e$ is \emph{non-returning} if for all $\gamma\in[\omega^{\alpha},\omega^{\alpha}\cdot n)$
and $\beta\in[0,\omega^{\alpha})$, $e\left(\gamma\right)e\left(\beta\right)^{-1}e\not\in e\Lambda$.
\end{defn}

\begin{rem}
This is equivalent to \cite[Definition 7.2]{ORDGRAPH} because if
$fe\left(\beta\right)^{-1}e\in e\Lambda$ and $d\left(f\right)<d\left(e\right)$,
then $f=e\left(d\left(f\right)\right)$. Setting $\gamma=d\left(f\right)$,
we have $e\left(\gamma\right)e\left(\beta\right)^{-1}e\in e\Lambda$.
It's also important to note that the definition of non-returning depends
on $\alpha$ and $n$, which is uniquely determined by the path $e$.
\end{rem}

\begin{defn}
Let $\sim$ be the minimal equivalence relation on $\Lambda$ satisfying
$f\sim g$ for all $f,g\in\Lambda$ with $s\left(f\right)=r\left(g\right)$.
A \emph{connected component} of $\Lambda$ is an equivalence class
in $\Lambda/\sim$.
\end{defn}

\begin{defn}[{\cite[Definition 7.11]{ORDGRAPH}}]
A path $e\in\Lambda$ is \emph{$\alpha$-full} if for every $v\in\Lambda_{0}$
in the same connected component of $\Lambda_{0}$ as $r\left(e\right)$,
there exist $\beta\in[0,\omega^{\alpha})$ and $f\in\Lambda_{\alpha}$
such that $s\left(f\right)=r\left(e\left(\beta\right)\right)$ and
$r\left(f\right)=v$.
\end{defn}

\begin{defn}
$\Lambda$ satisfies \emph{condition (S)} if for every $\alpha\in\mathrm{Ord}$
such that $\Lambda^{\omega^{\alpha}}\not=\emptyset$, every connected
component $F$ of $\Lambda_{\alpha}$, and every $n\in\mathbb{N}$
there exists non-returning, $\alpha$-full $f\in\Lambda$ such that
$r\left(f\right)\in F$ and $d\left(f\right)\in[\omega^{\alpha}\cdot n,\omega^{\alpha+1})$.
\end{defn}

The proof of \cite[Theorem 7.15]{ORDGRAPH} requires us to know $\mathcal{O}\left(X_{\alpha}\right)\cong\mathcal{O}\left(\Lambda_{\alpha+1}\right)$.
Without 1-regular vertices, the author was able to prove $J_{\alpha}=0$
using \cite[Theorem 3.9]{CKU4CPA}, and this was sufficient to prove
$\mathcal{O}\left(X_{\alpha}\right)\cong\mathcal{O}\left(\Lambda_{\alpha+1}\right)$.
Using \prettyref{thm:main-thm}, we can prove $\mathcal{O}\left(X_{\alpha}\right)\cong\mathcal{O}\left(\Lambda_{\alpha+1}\right)$
by showing that an ordinal graph $\Lambda$ satisfying condition (S)
satisfies condition (C). We accomplish this in the next lemma, and
then we obtain a generalization of \cite[Theorem 7.15]{ORDGRAPH}
as a corollary.
\begin{lem}
\label{lem:condition-S}If $\Lambda$ satisfies condition (S), then
$\Lambda$ is satisfies condition (C).
\end{lem}

\begin{proof}
Let $\Lambda$ be an ordinal graph not satisfying condition (C). Then
there exist some $f\in\Lambda^{\omega^{\beta}}$ and $g\in\Lambda_{\beta}$
such that $r\left(f\right)$ is $\beta$-regular and $gf=f$. We will
simplify the problem before showing $\Lambda$ does not satisfy condition
(S). First we will see that it suffices to consider the case where
$d\left(g\right)$ is a multiple of a power of $\omega$. Suppose
$d\left(g\right)=\omega^{\alpha}\cdot n+\varepsilon$ for some $\varepsilon<\omega^{\alpha}$
and $n\in[1,\omega)$. Then $g=pq$ for $p=g\left(\omega^{\alpha}\cdot n\right)$
and $q\in\Lambda_{\alpha}$. Moreover, $gf=g^{2}f=pqf=pqpqf=f$, and
hence $qf=qpqf$. In particular, $d\left(qp\right)=\varepsilon+\omega^{\alpha}\cdot n=\omega^{\alpha}\cdot n$.
Additionally, we have by \prettyref{lem:regular-hereditary} that
$r\left(qf\right)=s\left(p\right)$ is $\beta$-regular, since $r\left(p\right)$
is $\beta$-regular and $d\left(p\right)<\omega^{\alpha}$. Therefore,
we assume from this point without loss of generality that $d\left(g\right)=\omega^{\alpha}\cdot n$. 

Suppose towards a contradiction that $\Lambda$ satisfies condition
(S). Then there exists $m\in[1,\omega)$ and non-returning, $\alpha$-full
$h\in\Lambda^{\omega^{\alpha}\cdot m}$ such that $m\geq n$ and $r\left(h\right)$
belongs to the same connected component of $\Lambda_{\alpha}$ as
$r\left(g\right)$. Since $h$ is $\alpha$-full, there exist $u\in r\left(g\right)\Lambda_{\alpha}$
and $\delta<\omega^{\alpha}$ such that $s\left(u\right)=s\left(h\left(\delta\right)\right)$.
Since $gf=f$, we have $\omega^{\alpha}\cdot n\leq\omega^{\alpha}\cdot m<\omega^{\beta}$.
Because $r\left(g\right)$ is $\beta$-regular and $uh\left(\delta\right)^{-1}h\in\Lambda_{\beta}$,
\prettyref{lem:regular-hereditary} implies $s\left(h\right)$ is
$\beta$-regular. Thus there exists $y\in s\left(h\right)\Lambda^{\omega^{\beta}}$,
and
\[
\left\{ g^{k}uh\left(\delta\right)^{-1}hy:k\in\mathbb{N}\right\} 
\]
is finite. Define $w=uh\left(\delta\right)^{-1}hy$, and choose distinct
$k_{0},k_{1}\in\mathbb{N}$ such that $g^{k_{0}}w=g^{k_{1}}w$. Then
by left cancellation, $g^{k}w=w$ for some $k>0$. Since $d\left(uh\left(\delta\right)^{-1}h\right)=d\left(u\right)-\delta+d\left(h\right)=\omega^{\alpha}\cdot m$,
\prettyref{lem:path-factor} implies 
\[
w\left(\omega^{\alpha}\cdot n\right)=\left(uh\left(\delta\right)^{-1}h\right)\left(\omega^{\alpha}\cdot n\right)=u\left(h\left(\delta\right)^{-1}h\right)\left(-d\left(u\right)+\omega^{\alpha}\cdot n\right)=u\left(h\left(\delta\right)^{-1}h\right)\left(\omega^{\alpha}\cdot n\right)
\]
\[
=uh\left(\delta\right)^{-1}\left(h\left(\delta\right)h\left(\delta\right)^{-1}h\right)\left(\delta+\omega^{\alpha}\cdot n\right)=uh\left(\delta\right)^{-1}h\left(\omega^{\alpha}\cdot n\right)
\]
Hence $uh\left(\delta\right)^{-1}h\left(\omega^{\alpha}\cdot n\right)=g$.
Our contradiction will follow from the observation that 
\[
h\left(\omega^{\alpha}\cdot n\right)uh\left(\delta\right)^{-1}h\in h\Lambda
\]
which contradicts the fact that $h$ is non-returning. By possibly
replacing $k$ with a larger multiple, suppose without loss of generality
that $d\left(g^{k}\right)=\omega^{\alpha}\cdot nk>d\left(h\right)=\omega^{\alpha}\cdot m$.
Then since $g^{k}w=w$,
\[
\left(uh\left(\delta\right)^{-1}h\left(\omega^{\alpha}\cdot n\right)\right)^{k}w=uh\left(\delta\right)^{-1}hy
\]
\[
h\left(\delta\right)u^{-1}uh\left(\delta\right)^{-1}h\left(\omega^{\alpha}\cdot n\right)g^{k-1}w=hy
\]
\[
h\left(\omega^{\alpha}\cdot n\right)g^{k-1}w\in h\Lambda
\]
Since $d\left(h\left(\omega^{\alpha}\cdot n\right)g^{k-1}\right)=\omega^{\alpha}\cdot nk>d\left(h\right)$,
\prettyref{lem:path-factor} implies $h\left(\omega^{\alpha}\cdot n\right)g^{k-1}\in h\Lambda$.
Let $x,z\in\mathbb{N}$ such that $m=xn+z$, and note that $x>0$
since $m\geq n$. Then
\[
h\left(\omega^{\alpha}\cdot n\right)g^{k-1}=h\left(\omega^{\alpha}\cdot n\right)g^{x-1}g\left(\omega^{\alpha}\cdot z\right)g\left(\omega^{\alpha}\cdot z\right)^{-1}g^{k-x}
\]
\[
d\left(h\left(\omega^{\alpha}\cdot n\right)g^{x-1}g\left(\omega^{\alpha}\cdot z\right)\right)=\omega^{\alpha}\cdot n+\omega^{\alpha}\cdot n\left(x-1\right)+\omega^{\alpha}\cdot z
\]
\[
=\omega^{\alpha}\cdot\left(nx+z\right)=\omega^{\alpha}\cdot m=d\left(h\right)
\]
Thus unique factorization in \prettyref{def:ordinal-graph} implies
$h\left(\omega^{\alpha}\cdot n\right)g^{x-1}g\left(\omega^{\alpha}\cdot z\right)=h$.
Finally, we directly calculate:
\[
h\left(\omega^{\alpha}\cdot n\right)uh\left(\delta\right)^{-1}h=h\left(\omega^{\alpha}\cdot n\right)uh\left(\delta\right)^{-1}h\left(\omega^{\alpha}\cdot n\right)g^{x-1}g\left(\omega^{\alpha}\cdot z\right)
\]
\[
=h\left(\omega^{\alpha}\cdot n\right)g^{x}g\left(\omega^{\alpha}\cdot z\right)=h\left(\omega^{\alpha}\cdot n\right)g^{x-1}g\left(\omega^{\alpha}\cdot z\right)g\left(\omega^{\alpha}\cdot z\right)^{-1}gg\left(\omega^{\alpha}\cdot z\right)
\]
\[
=hg\left(\omega^{\alpha}\cdot z\right)^{-1}gg\left(\omega^{\alpha}\cdot z\right)\in h\Lambda
\]
And this directly contradicts the assertion that $h$ is non-returning.
\end{proof}
\begin{cor}
\label{cor:cku}If $\Lambda$ satisfies condition (S) then each correspondence
$X_{\alpha}$ satisfies condition (S) as defined in \cite{CKU4CPA}.
Moreover, if $\pi:\mathcal{O}\left(\Lambda\right)\rightarrow\mathcal{A}$
is a homomorphism into a $C^{*}$-algebra $\mathcal{A}$, then $\pi$
is injective iff for each $v\in\Lambda_{0}$, $\pi\left(T_{v}\right)\not=0$. 
\end{cor}

\begin{proof}
If $\Lambda$ satisfies condition (S), then by \prettyref{lem:condition-S},
$\Lambda$ satisfies condition (C). Hence by \prettyref{thm:main-thm},
$\mathcal{O}\left(X_{\alpha}\right)\cong\mathcal{O}\left(\Lambda_{\alpha+1}\right)$
for every $\alpha\in\mathrm{Ord}$. The proof of \cite[Theorem 7.15]{ORDGRAPH}
only uses the hypothesis that $\Lambda$ has no 1-regular vertices
in order to prove $\mathcal{O}\left(X_{\alpha}\right)\cong\mathcal{O}\left(\Lambda_{\alpha+1}\right)$,
and therefore the same proof works more generally in this case. At
first glance, the correspondences defined in \cite[Theorem 6.1]{ORDGRAPH}
which are used in the proof of \cite[Theorem 7.15]{ORDGRAPH} appear
to be different than those we've defined in \prettyref{def:correspondences},
but they are actually isomorphic. Indeed, the correspondences defined
in \cite[Theorem 6.1]{ORDGRAPH} are the images $\left(\psi_{\alpha}\left(X_{\alpha}\right),\rho_{\alpha}^{\alpha+1}\left(\mathcal{O}\left(\Lambda_{\alpha}\right)\right)\right)$,
and since $\left(\psi_{\alpha},\rho_{\alpha}^{\alpha+1}\right)$ is
an injective representation of $X_{\alpha}$, these images are isomorphic
to $X_{\alpha}$.
\end{proof}

\section{Inductive Step}

In this section, we complete the inductive step for the proof of \prettyref{thm:main-thm}.
We let $\Lambda$ be a fixed ordinal graph and assume the following
for the rest of this section.

\begin{assumptions*}\label{assumptions:inductive-assumptions}

Assume we have the following:
\begin{enumerate}
\item $\Lambda$ satisfies condition (C) of \prettyref{def:normal}.
\item For each $\alpha<\zeta$, $\rho_{\alpha}=\rho_{\zeta}\circ\rho_{\alpha}^{\zeta}$
is injective. In particular, each $\rho_{\alpha}^{\zeta}:\mathcal{O}\left(\Lambda_{\alpha}\right)\rightarrow\mathcal{O}\left(\Lambda_{\zeta}\right)$
is injective, and we regard $\mathcal{O}\left(\Lambda_{\alpha}\right)$
as a $C^{*}$-subalgebra of $\mathcal{O}\left(\Lambda_{\zeta}\right)$.
\item For each $\alpha<\zeta$, the representation $\left(\psi_{\alpha},\rho_{\alpha}^{\alpha+1}\right):\left(X_{\alpha},\mathcal{O}\left(\Lambda_{\alpha}\right)\right)\rightarrow\mathcal{O}\left(\Lambda_{\alpha+1}\right)$
defined in \prettyref{prop:correspondence-representation} is universal
for covariant representations of $X_{\alpha}$. In particular, we
have an isomorphism $\psi_{\alpha}\times\rho_{\alpha}^{\alpha+1}:\mathcal{O}\left(X_{\alpha}\right)\rightarrow\mathcal{O}\left(\Lambda_{\alpha+1}\right)$
defined by
\begin{align*}
\psi_{\alpha}\left(\delta_{e}\right) & =T_{e} & \text{for }e\in\Lambda^{\omega^{\alpha}}\\
\rho_{\alpha}^{\alpha+1}\left(T_{f}\right) & =T_{f} & \text{for }f\in\Lambda_{\alpha}
\end{align*}
Thus the gauge action on $\mathcal{O}\left(X_{\alpha}\right)$ under
these isomorphisms is $\Gamma_{\alpha}$ defined in \prettyref{lem:alg-auts}.
\item For each $\alpha<\zeta$, the Katsura ideal $J_{\alpha}=\left(\ker\varphi_{\alpha}\right)^{\perp}\cap\varphi_{\alpha}^{-1}\left(\mathcal{K}\left(X_{\alpha}\right)\right)$
is generated in $\mathcal{O}\left(\Lambda_{\alpha}\right)$ by $\left\{ T_{v}:v\in\Lambda_{0}\text{ is }\alpha\text{-regular}\right\} $.
In particular, $J_{\alpha+1}$ is invariant under the gauge action
$\Gamma_{\alpha}$ when $\alpha+1<\zeta$. Moreover, for each $\beta<\alpha<\zeta$,
$\mathcal{O}\left(\Lambda_{\beta}\right)\cap\left(J_{\alpha}+\psi_{\beta}^{(1)}\left(\mathcal{K}\left(X_{\beta}\right)\right)\right)\subseteq J_{\beta}$,
and $\mathcal{O}\left(\Lambda_{0}\right)\cap J_{\alpha}=\left\{ T_{v}:v\in\Lambda_{0}\text{ is }\alpha\text{-regular}\right\} $.
\end{enumerate}
\end{assumptions*}
\begin{defn}
\label{def:cancellative}If $f\in\Lambda^{*}$, then $f$ is \emph{$\alpha$-cancellative}
if for every $0\leq\varepsilon\leq\beta<\omega^{\alpha+1}$, $f=f\left(\beta\right)f\left(\varepsilon\right)^{-1}f$
implies $-\varepsilon+\beta<\omega^{\alpha}$. Define $C_{\alpha}=\left\{ f\in\Lambda^{*}:f\text{ is }\alpha\text{-cancellative}\right\} $.
\end{defn}

We define $\alpha$-cancellative for members of $\Lambda^{*}$ so
that we can apply the definition for paths in $\Lambda$ and elements
of $\partial\Lambda$, which are all elements of $\Lambda^{*}$. We
choose this name because $\alpha$-cancellative is a weak form of
right cancellation. For example, if $f$ is $0$-cancellative, then
for every finite length path $g,h\in\Lambda_{1}r\left(f\right)$,
$gf=hf$ implies $g=h$. Also, $f$ cancels on the right if and only
if $f$ is $\alpha$-cancellative for all $\alpha\in\mathrm{Ord}$.
We won't make use of either of these facts, though.
\begin{lem}
\label{lem:cancellative-length-independent}For each $f\in\Lambda^{*}$,
we have the following implications:
\begin{enumerate}
\item If $L\left(f\right)>\omega^{\alpha+1}$, then $f$ is $\alpha$-cancellative
iff $f\left(\omega^{\alpha+1}\right)$ is $\alpha$-cancellative.
\item If $L\left(f\right)<\omega^{\alpha+1}$, then $f$ is $\alpha$-cancellative.
\end{enumerate}
\end{lem}

\begin{proof}
To prove (1), note that for all $\varepsilon\leq\beta<\omega^{\alpha+1}$,
$f=f\left(\beta\right)f\left(\varepsilon\right)^{-1}f$ is equivalent
to $f\left(\omega^{\alpha+1}\right)=f\left(\beta\right)f\left(\varepsilon\right)^{-1}f\left(\omega^{\alpha+1}\right)$.
For (2), suppose $f=f\left(\beta\right)f\left(\varepsilon\right)^{-1}f$.
Then $L\left(f\right)=\beta-\varepsilon+L\left(f\right)$, and if
all values are smaller than $\omega^{\alpha+1}$, then the coefficients
for $\omega^{\alpha}$ in Cantor's normal form on both sides must
be equal. In particular, the coefficient for $\beta$ must equal the
coefficient for $\varepsilon$, and $-\varepsilon+\beta<\omega^{\alpha}$.
\end{proof}
\begin{lem}
\label{lem:returning-loops}If $f\in\Lambda^{*}$ is not $\alpha$-cancellative,
then there exist $g,h\in\Lambda_{\alpha+1}$ such that $r\left(g\right)=r\left(f\right)$,
$s\left(g\right)=r\left(h\right)=s\left(h\right)$, $d\left(h\right)\geq\omega^{\alpha}$,
and $hg^{-1}f=g^{-1}f$.
\end{lem}

\begin{proof}
Suppose $f$ is not $\alpha$-cancellative. Choose $0\leq\varepsilon\leq\beta<\omega^{\alpha+1}$
such that $f\left(\beta\right)f\left(\varepsilon\right)^{-1}f=f$
yet $-\varepsilon+\beta\geq\omega^{\alpha}$. Define
\begin{align*}
g & =f\left(\beta\right) & h & =f\left(\varepsilon\right)^{-1}f\left(\beta\right)
\end{align*}
Then we have
\begin{align*}
g^{-1}f & =f\left(\beta\right)^{-1}f\\
 & =f\left(\beta\right)^{-1}f\left(\beta\right)f\left(\varepsilon\right)^{-1}f\\
 & =f\left(\varepsilon\right)^{-1}f\\
 & =f\left(\varepsilon\right)^{-1}f\left(\beta\right)f\left(\beta\right)^{-1}f\\
 & =hg^{-1}f
\end{align*}
Also, $d\left(h\right)=-\varepsilon+\beta\geq\omega^{\alpha}$, as
desired.
\end{proof}
\begin{cor}
\label{cor:non-returning-equiv}For $f\in\Lambda^{*}$ and $p\in\Lambda_{\alpha+1}r\left(f\right)$,
$pf$ is $\alpha$-cancellative if and only if $f$ is $\alpha$-cancellative.
\end{cor}

\begin{proof}
Suppose $pf$ is not $\alpha$-cancellative. By \prettyref{lem:returning-loops},
choose $g,h\in\Lambda_{\alpha+1}$ such that $hg^{-1}pf=g^{-1}pf$.
Since we may choose $d\left(h\right)\geq\omega^{\alpha}$, we may
assume without loss of generality by replacing $h$ with a larger
power $h^{n}$ that $d\left(h\right)\geq d\left(p\right)+\omega^{\alpha}$.
Then $p^{-1}ghg^{-1}pf=f$. We now split into two cases.
\begin{casenv}
\item If $d\left(g\right)\geq d\left(p\right)$, define $\beta=d\left(p^{-1}gh\right)=-d\left(p\right)+d\left(g\right)+d\left(h\right)$
and $\varepsilon=-d\left(p\right)+d\left(g\right)$. Then $f\left(\beta\right)f\left(\varepsilon\right)^{-1}f=p^{-1}ghg^{-1}pf=f$,
yet $-\varepsilon+\beta=d\left(h\right)\geq\omega^{\alpha}$.
\item If $d\left(g\right)<d\left(p\right)$, then $p^{-1}ghg^{-1}p$ is
a well-defined path since $d\left(p\right)<d\left(p\right)+\omega^{\alpha}\leq d\left(h\right)\leq d\left(gh\right)$.
Define $\beta=p^{-1}ghg^{-1}p$ and $\varepsilon=0$. Then $f\left(\beta\right)f\left(\varepsilon\right)^{-1}f=f$,
but $-\varepsilon+\beta=\beta\geq-d\left(p\right)+d\left(h\right)\geq\omega^{\alpha}$.
\end{casenv}
In any case, we see $f$ is not $\alpha$-cancellative.

On the other hand, if $pf$ is $\alpha$-cancellative and $0\leq\varepsilon\leq\beta<\omega^{\alpha+1}$
such that $f\left(\beta\right)f\left(\varepsilon\right)^{-1}f=f$,
then
\begin{align*}
pf & =pf\left(\beta\right)f\left(\varepsilon\right)^{-1}p^{-1}pf\\
 & =\left(pf\right)\left(d\left(p\right)+\beta\right)\left(pf\right)\left(d\left(p\right)+\varepsilon\right)^{-1}pf
\end{align*}
Thus $-\varepsilon-d\left(p\right)+d\left(p\right)+\beta=-\varepsilon+\beta<\omega^{\alpha}$,
and $f$ is $\alpha$-cancellative.
\end{proof}
\begin{lem}
\label{lem:loop-reg}If $v\in\Lambda_{0}$ is $\alpha$-regular, then
for each $g\in v\Lambda_{\alpha}\backslash\Lambda_{0}$, $s\left(g\right)\not=v$.
\end{lem}

\begin{proof}
Suppose $v\in\Lambda_{0}$ is $\alpha$-regular and $g\in v\Lambda_{\alpha}\backslash\Lambda_{0}$
such that $s\left(g\right)=r\left(g\right)=v$. Let $f\in v\Lambda^{\omega^{\alpha}}$
be arbitrary. Since $v\Lambda^{\omega^{\alpha}}$ is finite, $\left\{ g^{n}f:n\in\mathbb{N}\right\} $
is finite. Therefore, there exist distinct $n_{0},n_{1}\in\mathbb{N}$
such that $g^{n_{0}}f=g^{n_{1}}f$. Cancelling on the left, we have
$g^{n}f=f$ for some $n\in\mathbb{N}$. This would imply $\Lambda$
does not satisfy condition (C), contradicting \prettyref{assumptions:inductive-assumptions}.
Thus such $g$ does not exist.
\end{proof}
\begin{cor}
\label{cor:alpha-reg-non-returning}If $v\in\Lambda_{0}$ is $\alpha+1$-regular,
then every $f\in v\Lambda^{*}$ is $\alpha$-cancellative.
\end{cor}

\begin{proof}
If $f\in v\Lambda^{*}$ is not $\alpha$-cancellative, there are $g,h\in\Lambda_{\alpha+1}$
such that $hg^{-1}f=g^{-1}f$. Since $r\left(g\right)=r\left(f\right)$
is $\alpha+1$-regular, \prettyref{lem:regular-hereditary} implies
$s\left(g\right)=s\left(h\right)=r\left(h\right)$ is $\alpha+1$-regular,
but this contradicts \prettyref{lem:loop-reg}.
\end{proof}
\begin{prop}
\label{prop:good-boundary-paths-exist}For each $v\in\Lambda_{0}$,
there exists $f\in v\partial\Lambda$ such that for each $\alpha\in\mathrm{Ord}$,
either $L\left(f\right)\leq\omega^{\alpha+1}$ or $f$ is $\alpha$-cancellative.
\end{prop}

\begin{proof}
Define the following set of ordinals:
\[
S=\left\{ \alpha\in\mathrm{Ord}:\exists g,h\in\Lambda_{\alpha+1}\text{ such that }r\left(g\right)=v,s\left(g\right)=r\left(h\right)=s\left(h\right),\text{ and }d\left(h\right)\geq\omega^{\alpha}\right\} 
\]
If $S$ is empty, apply \prettyref{lem:boundary-paths-exist} to construct
$f\in v\partial\Lambda$. Then by \prettyref{lem:returning-loops},
$f$ is $\alpha$-cancellative for every $\alpha\in\mathrm{Ord}$,
and we are finished. Otherwise, $S$ has a least element $\alpha$,
and there exist $g,h\in\Lambda_{\alpha+1}$ satisfying the conditions
which define $S$. Now we define $h^{\omega}\in\Lambda^{*}$ such
that $L\left(h^{\omega}\right)=d\left(h\right)\cdot\omega=\omega^{\alpha+1}$.
If $\beta<\omega^{\alpha+1}$, then there is $n\in\mathbb{N}$ such
that $\beta<\omega^{\alpha}\cdot n$, and we define
\[
h^{\omega}\left(\beta\right)=h^{n}\left(\beta\right)
\]
Then for all $m\in\mathbb{N}$, $h^{n+m}\left(\beta\right)=h^{n}\left(\beta\right)$,
and $h^{\omega}$ is a well-defined element of $\Lambda^{*}$.

Next we prove $h^{\omega}\in\partial\Lambda$. Suppose $\beta<\omega^{\alpha+1}$
and $w=s\left(h^{\omega}\left(\beta\right)\right)=s\left(h^{n}\left(\beta\right)\right)$
is $\epsilon$-regular. We must show $L\left(h^{\omega}\right)=\omega^{\alpha+1}>\beta+\omega^{\epsilon}$.
If $\epsilon\leq\alpha$, then $\beta+\omega^{\epsilon}<\omega^{\alpha+1}$.
Thus we assume $\epsilon\geq\alpha+1$, and by \prettyref{lem:regular-hereditary},
$w$ is $\alpha+1$-regular. By choosing $n$ to be minimal, we may
assume without loss of generality there exist $\eta<d\left(h\right)$
such that $\beta=d\left(h\right)\cdot\left(n-1\right)+\eta$. Then
\[
h^{n}\left(\beta\right)=h^{n}\left(d\left(h\right)\cdot\left(n-1\right)+\eta\right)=\left(h^{n-1}h\right)\left(d\left(h^{n-1}\right)+\eta\right)=h^{n-1}h\left(\eta\right)
\]
Therefore $w=s\left(h\left(\eta\right)\right)=r\left(h\left(\eta\right)^{-1}h\right)$.
By \prettyref{lem:regular-hereditary}, since $h\left(\eta\right)^{-1}h\in\Lambda_{\alpha+1}$,
$s\left(h\left(\eta\right)^{-1}h\right)=s\left(h\right)$ is $\alpha+1$-regular.
Since $s\left(h\right)=r\left(h\right)$ and $h\in\Lambda_{\alpha+1}$,
this contradicts \prettyref{lem:loop-reg}. Therefore $\epsilon\leq\alpha$
and $h^{\omega}\in\partial\Lambda$.

By \prettyref{lem:compose-boundary}, $gh^{\omega}\in v\partial\Lambda$.
Moreover, $L\left(gh^{\omega}\right)=d\left(g\right)+L\left(h^{\omega}\right)=d\left(g\right)+\omega^{\alpha+1}=\omega^{\alpha+1}$.
If $\beta<\alpha$ and $gh^{\omega}$ is not $\beta$-cancellative,
then by \prettyref{lem:returning-loops}, $\beta\in S$; however,
this would contradict the minimality of $\alpha$. Therefore, $f=gh^{\omega}$
satisfies the desired requirements.
\end{proof}
\begin{lem}
\label{prop:shift-func}There exists a function $v:\Lambda^{*}\rightarrow\mathbb{Z}^{[0,\zeta)}$
such that for all $\alpha\in[0,\zeta)$,
\begin{enumerate}
\item If $f\in\Lambda^{*}$ and $L\left(f\right)>\omega^{\alpha+1}$, $v\left(f\right)_{\alpha}=v\left(f\left(\omega^{\alpha+1}\right)\right)_{\alpha}$.
\item If $f\in\Lambda$ and $d\left(f\right)\in\left[\omega^{\alpha}\cdot n,\omega^{\alpha}\cdot\left(n+1\right)\right)$,
$v\left(f\right)_{\alpha}=n$. In particular, $v\left(f\right)_{\alpha}=0$
if $f\in\Lambda_{\alpha}$.
\item If $f\in\Lambda^{*}$ is $\alpha$-cancellative, then for each $\beta<\omega^{\alpha+1}$,
$v\left(f\right)_{\alpha}=v\left(f\left(\beta\right)\right)_{\alpha}+v\left(f\left(\beta\right)^{-1}f\right)_{\alpha}$.
\end{enumerate}
\end{lem}

\begin{proof}
Let $\zeta$ and $\alpha$ be given. We define an equivalence relation
$\sim$ on $\Lambda^{*}$ by $f\sim g$ iff there exist $\beta,\gamma<\omega^{\alpha+1}$
such that $f\left(\beta\right)^{-1}f=g\left(\gamma\right)^{-1}g$.
If $f\sim g$ and $g\sim h$, there exist $\beta_{0},\beta_{1},\gamma_{0},\gamma_{1}<\omega^{\alpha+1}$
such that $f\left(\beta_{0}\right)^{-1}f=g\left(\gamma_{0}\right)^{-1}g$
and $g\left(\gamma_{1}\right)^{-1}g=h\left(\beta_{1}\right)^{-1}h$.
Assuming without loss of generality that $\gamma_{0}<\gamma_{1}$,
this gives
\[
g\left(\gamma_{0}\right)f\left(\beta_{0}\right)^{-1}f=g\left(\gamma_{1}\right)h\left(\beta_{1}\right)^{-1}h
\]
\[
f=f\left(\beta_{0}\right)g\left(\gamma_{0}\right)^{-1}g\left(\gamma_{1}\right)h\left(\beta_{1}\right)^{-1}h
\]
\[
f\left(\beta_{0}-\gamma_{0}+\gamma_{1}\right)^{-1}f=h\left(\beta_{1}\right)^{-1}h
\]
Thus $\sim$ is transitive, and both reflexivity and symmetry are
clear. For each $f\in\partial\Lambda$, choose a representative $c\left(f\right)$
of minimal length\emph{ }$L\left(c\left(f\right)\right)$\emph{ }such
that $f\sim c\left(f\right)$ and $c\left(f\right)=c\left(g\right)$
if $f\sim g$. In particular, for every $f\in\Lambda_{\alpha+1}$
we have $c\left(f\right)=s\left(f\right)$.

Now, let $f\in\Lambda^{*}$ be $\alpha$-cancellative. Note that this
includes the case for which $f\in\Lambda_{\alpha+1}$. We will proceed
to define $v\left(f\right)_{\alpha}$. If $L\left(f\right)>\omega^{\alpha+1}$,
define $v\left(f\right)_{\alpha}=v\left(f\left(\omega^{\alpha+1}\right)\right)_{\alpha}$
so that $v$ satisfies property 1. Otherwise, define $g=c\left(f\right)$,
and note that $g\sim f$. Thus we choose $\beta,\gamma<\omega^{\alpha+1}$
such that $g\left(\beta\right)^{-1}g=f\left(\gamma\right)^{-1}f$.
Choose $n,m\in\mathbb{N}$ such that $\beta\in\left[\omega^{\alpha}\cdot n,\omega^{\alpha}\cdot\left(n+1\right)\right)$
and $\gamma\in\left[\omega^{\alpha}\cdot m,\omega^{\alpha}\cdot\left(m+1\right)\right)$,
and define $v\left(f\right)_{\alpha}=m-n$. Now we show $v\left(f\right)_{\alpha}$
is well-defined. Suppose $g\left(\beta_{0}\right)^{-1}g=f\left(\gamma_{0}\right)^{-1}f$
and $g\left(\beta_{1}\right)^{-1}g=f\left(\gamma_{1}\right)^{-1}f$
for $\gamma_{0}<\gamma_{1}$. Then
\[
g\left(\beta_{0}\right)f\left(\gamma_{0}\right)^{-1}f=g\left(\beta_{1}\right)f\left(\gamma_{1}\right)^{-1}f
\]
\[
g\left(\beta_{1}\right)^{-1}g\left(\beta_{0}\right)f\left(\gamma_{0}\right)^{-1}f\left(\gamma_{1}\right)f\left(\gamma_{1}\right)^{-1}f=f\left(\gamma_{1}\right)^{-1}f
\]
If $\beta_{1}\leq\beta_{0}-\gamma_{0}+\gamma_{1}$, then $g\left(\beta_{1}\right)^{-1}g\left(\beta_{0}\right)f\left(\gamma_{0}\right)^{-1}f\left(\gamma_{1}\right)$
is a well-defined path. Thus if we set $p=g\left(\beta_{1}\right)^{-1}g\left(\beta_{0}\right)f\left(\gamma_{0}\right)^{-1}f\left(\gamma_{1}\right)$
and $q=f\left(\gamma_{1}\right)^{-1}f$, we have $pq=q$. Since $d\left(p\right)$,$\gamma_{1}<\omega^{\alpha+1}$,
\prettyref{cor:non-returning-equiv} and \prettyref{def:cancellative}
imply $d\left(p\right)=-\beta_{1}+\beta_{0}-\gamma_{0}+\gamma_{1}<\omega^{\alpha}$.
If $\beta_{j}\in\left[\omega^{\alpha}\cdot n_{j},\omega^{\alpha}\cdot\left(n_{j}+1\right)\right)$
and $\gamma_{j}\in\left[\omega^{\alpha}\cdot m_{j},\omega^{\alpha}\cdot\left(m_{j}+1\right)\right)$,
then we have $d\left(p\right)\geq\omega^{\alpha}\cdot\left(-m_{1}+m_{0}-n_{0}+n_{1}\right)$.
Hence $-m_{1}+m_{0}-n_{0}+n_{1}=0$, and $m_{0}-n_{0}=m_{1}-n_{1}$.
If $\beta_{1}>\beta_{0}-\gamma_{0}+\gamma_{1}$, then $p=f\left(\gamma_{1}\right)^{-1}f\left(\gamma_{0}\right)g\left(\beta_{0}\right)^{-1}g\left(\beta_{1}\right)$
is well-defined, $pq=q$, and the same argument implies $-m_{1}+m_{0}-n_{0}+n_{1}=0$.
Thus $v\left(f\right)_{\alpha}$ is well-defined.

To see property 2, let $f\in\Lambda_{\alpha+1}$ with $d\left(f\right)\in\left[\omega^{\alpha}\cdot m,\omega^{\alpha}\cdot\left(m+1\right)\right)$.
Then $g=c\left(f\right)=s\left(f\right)$, and $g\left(\beta\right)^{-1}g=f\left(\gamma\right)^{-1}f$
only if $\beta=0$ and $\gamma=d\left(f\right)$. Hence $v\left(f\right)_{\alpha}=m$,
as desired.

For property 3, let $f\in\Lambda^{*}$ such that $f$ is $\alpha$-cancellative.
For $\beta<\omega^{\alpha+1}$, let $e=f\left(\beta\right)$ and $g=f\left(\beta\right)^{-1}f$.
Suppose $\beta\in\left[\omega^{\alpha}\cdot n,\omega^{\alpha}\cdot\left(n+1\right)\right)$
for $n\in\mathbb{N}$. If $h=c\left(f\right)$, there are $\gamma,\delta<\omega^{\alpha+1}$
such that
\[
h\left(\gamma\right)^{-1}h=f\left(\delta\right)^{-1}f
\]
\[
h\left(\gamma\right)^{-1}h=f\left(\delta\right)^{-1}f\left(\beta\right)g
\]
If $\beta<\delta$, this gives
\[
h\left(\gamma\right)^{-1}h=g\left(-\beta+\delta\right)^{-1}g
\]
Thus if $\gamma\in\left[\omega^{\alpha}\cdot m,\omega^{\alpha}\cdot\left(m+1\right)\right)$
and $\delta\in\left[\omega^{\alpha}\cdot k,\omega^{\alpha}\cdot\left(k+1\right)\right)$,
then
\[
v\left(g\right)_{\alpha}=-n+k-m=-v\left(e\right)_{\alpha}+v\left(f\right)_{\alpha}
\]
as desired. When $\beta\geq\delta$, we have
\[
h=h\left(\gamma\right)f\left(\delta\right)^{-1}f\left(\beta\right)g
\]
\[
h\left(\gamma-\delta+\beta\right)^{-1}h=g
\]
Therefore
\[
v\left(g\right)_{\alpha}=-m+k-n=-v\left(e\right)_{\alpha}+v\left(f\right)_{\alpha}
\]

Finally, note that we have not defined $v\left(f\right)_{\alpha}$
for $f\in\Lambda^{*}$ which is not $\alpha$-cancellative. However,
properties 2 and 3 only apply when $f$ is $\alpha$-cancellative.
Hence we may arbitrarily define $v\left(f\right)_{\alpha}$ when $f$
is not $\alpha$-cancellative, as long as we ensure $v\left(f\right)_{\alpha}=v\left(f\left(\omega^{\alpha}\right)\right)_{\alpha}$
when $L\left(f\right)>\omega^{\alpha}$.
\end{proof}
\begin{prop}
\label{prop:shift-rep}Let $v$ denote the function constructed in
\prettyref{prop:shift-func}. For each ordinal graph $\Lambda$ and
$\zeta\in\mathrm{Ord}$, there exists a representation
\[
\pi:\mathcal{O}\left(\Lambda\right)\rightarrow B\left(\ell^{2}\left(\partial\Lambda\times\mathbb{Z}^{[0,\zeta)}\right)\right)
\]
defined such that for $e\in\Lambda^{\omega^{*}}$,
\[
\pi\left(T_{e}\right)\xi_{f,n}=\begin{cases}
\xi_{ef,m} & s\left(e\right)=r\left(f\right)\\
0 & \text{otherwise}
\end{cases}
\]
where
\[
m_{\alpha}=\begin{cases}
n_{\alpha}+v\left(e\right)_{\alpha} & L\left(ef\right)\leq\omega^{\alpha+1}\\
n_{\alpha}+v\left(ef\right)_{\alpha}-v\left(f\right)_{\alpha} & L\left(ef\right)>\omega^{\alpha+1}
\end{cases}
\]
\end{prop}

\begin{proof}
We proceed by showing $\left\{ \pi\left(T_{e}\right):e\in\Lambda^{\omega^{*}}\right\} $
satisfies the relations in \prettyref{thm:gen-relations}. First we
compute the adjoints. Note that the function $\left(f,n\right)\mapsto\left(ef,m\right)$
is injective, hence we have
\[
\pi\left(T_{e}\right)^{*}\xi_{g,m}=\begin{cases}
\xi_{e^{-1}g,n} & g\in e\partial\Lambda\\
0 & \text{otherwise}
\end{cases}
\]
where
\[
n_{\alpha}=\begin{cases}
m_{\alpha}-v\left(e\right)_{\alpha} & L\left(g\right)\leq\omega^{\alpha+1}\\
m_{\alpha}-v\left(g\right)_{\alpha}+v\left(e^{-1}g\right)_{\alpha} & L\left(g\right)>\omega^{\alpha+1}
\end{cases}
\]

Let $f\in\partial\Lambda$, $n\in\mathbb{Z}^{[0,\zeta)}$, and $\alpha<\zeta$
be fixed. Assuming that $s\left(e\right)=r\left(f\right)$, we see
$\pi\left(T_{e}\right)^{*}\pi\left(T_{e}\right)\xi_{f,n}=\xi_{ef,m}=\xi_{f,n}$.
Therefore $\pi\left(T_{e}\right)^{*}\pi\left(T_{e}\right)=\pi\left(T_{s\left(e\right)}\right)$,
and relation (1) is satisfied. For relation (2), let $e,g\in\Lambda^{\omega^{*}}$
such that $s\left(g\right)=r\left(e\right)$, $s\left(e\right)=r\left(f\right)$,
and $d\left(g\right)<d\left(e\right)$. Then $\pi\left(T_{g}\right)\pi\left(T_{e}\right)\xi_{f,n}=\pi\left(T_{g}\right)\xi_{ef,m}=\xi_{gef,k}$
for some $m,k\in\mathbb{Z}^{[0,\zeta)}$. If $L\left(ef\right)>\omega^{\alpha+1}$,
then $L\left(gef\right)>\omega^{\alpha+1}$, and
\[
k_{\alpha}=m_{\alpha}+v\left(gef\right)_{\alpha}-v\left(ef\right)_{\alpha}
\]
\[
=n_{\alpha}+v\left(ef\right)_{\alpha}-v\left(f\right)_{\alpha}+v\left(gef\right)_{\alpha}-v\left(ef\right)_{\alpha}
\]
Hence if $\pi\left(T_{ge}\right)\xi_{f,n}=\xi_{gef,p}$, we have $p_{\alpha}=k_{\alpha}$.
If, on the other hand, $L\left(ef\right)\leq\omega^{\alpha+1}$, then
$g,e\in\Lambda_{\alpha+1}$ and
\[
k_{\alpha}=m_{\alpha}+v\left(g\right)_{\alpha}=n_{\alpha}+v\left(e\right)_{\alpha}+v\left(g\right)_{\alpha}
\]
Since $e,g\in\Lambda_{\alpha+1}$, $e$ and $g$ are $\alpha$-cancellative.
Hence by \prettyref{prop:shift-func}, $v\left(ge\right)_{\alpha}=v\left(g\right)_{\alpha}+v\left(e\right)_{\alpha}$,
and $p_{\alpha}=k_{\alpha}$. Therefore relation (2) is satisfied.
For relation (3), let $e,g\in\Lambda^{\omega^{*}}$ with $e\Lambda\cap g\Lambda=\emptyset$
and $s\left(e\right)=r\left(f\right)$. Then $ef\not\in g\partial\Lambda$,
hence $\pi\left(T_{g}\right)^{*}\pi\left(T_{e}\right)=0$. Finally,
let $v\in\Lambda_{0}$ be $\alpha$-regular. Then by \prettyref{def:boundary-def},
for each $f\in v\partial\Lambda$, $L\left(f\right)>\omega^{\alpha}$
and $f=f\left(\omega^{\alpha}\right)f\left(\omega^{\alpha}\right)^{-1}f$.
Therefore
\[
\pi\left(T_{v}\right)=\sum_{e\in v\Lambda^{\omega^{\alpha}}}\pi\left(T_{e}\right)\pi\left(T_{e}\right)^{*}
\]
and relation (4) is satisfied.
\end{proof}
\begin{lem}
\label{lem:injective}$\pi\circ\rho_{\zeta}$ is injective. In particular,
$\rho_{\zeta}$ is injective.
\end{lem}

\begin{proof}
For each $\alpha<\zeta$, define $H_{\alpha}$ to be the following
closed subspace of $\ell^{2}\left(\partial\Lambda\times\mathbb{Z}^{[0,\zeta)}\right)$:
\[
H_{\alpha}=\overline{\mathrm{span}}\left\{ \xi_{f,n}:\text{for each }\beta\geq\alpha,L\left(f\right)\leq\omega^{\beta+1}\text{ or }f\text{ is }\beta\text{-cancellative}\right\} 
\]
Note that $H_{\alpha}$ is an invariant subspace for $\pi\circ\rho_{\alpha+1}$:
for $e\in\Lambda_{\beta+1}$, $L\left(ef\right)\leq\omega^{\beta+1}$
iff $L\left(f\right)\leq\omega^{\beta+1}$, and by \prettyref{cor:non-returning-equiv},
$ef$ is $\beta$-cancellative iff $f$ is $\beta$-cancellative.
Let $\pi_{\alpha}:\mathcal{O}\left(\Lambda_{\alpha+1}\right)\rightarrow B\left(H_{\alpha}\right)$
be the restriction of $\pi\circ\rho_{\alpha+1}$ to this subspace. 

We will use transfinite induction to show $\pi_{\alpha}$ is injective.
Suppose $\pi_{\beta}$ is injective for each $\beta<\alpha$. Since
$H_{\beta}\subseteq H_{\alpha}$, we also have $\pi_{\alpha}\circ\rho_{\beta+1}^{\alpha+1}$
is injective. If $\alpha$ is a successor ordinal (that is, $\alpha=\beta+1$
for some $\beta<\alpha$), then this implies $\pi_{\alpha}\circ\rho_{\alpha}^{\alpha+1}$
is injective. If, on the other hand, $\alpha$ is a non-zero limit
ordinal, then by \prettyref{prop:inductive-limit}, $\mathcal{O}\left(\Lambda_{\alpha}\right)$
is the inductive limit of the maps $\rho_{\beta}^{\beta+1}:\mathcal{O}\left(\Lambda_{\beta}\right)\rightarrow\mathcal{O}\left(\Lambda_{\beta+1}\right)$.
By the \prettyref{assumptions:inductive-assumptions}, each of these
maps is injective, and $\mathcal{O}\left(\Lambda_{\alpha}\right)=\overline{\cup_{\beta<\alpha}\mathcal{O}\left(\Lambda_{\beta+1}\right)}$.
Therefore each $\pi_{\alpha}\circ\rho_{\beta+1}^{\alpha+1}$ is an
isometry, which implies injectivity of $\pi_{\alpha}\circ\rho_{\alpha}^{\alpha+1}$.
Finally, if $\alpha=0$ then \prettyref{prop:good-boundary-paths-exist}
implies that for each $v\in\Lambda_{0}$ there exists $\xi_{f,n}\in H_{0}$
such that $\pi_{0}\left(T_{v}\right)\xi_{f,n}=\xi_{f,n}$. In particular,
$\pi_{0}\left(T_{v}\right)\not=0$ for each $v\in\Lambda_{0}$, and
since $\mathcal{O}\left(\Lambda_{0}\right)\cong c_{0}\left(\Lambda_{0}\right)$,
$\pi_{0}\circ\rho_{0}^{1}$ is injective.

In any case, we have $\pi_{\alpha}\circ\rho_{\alpha}^{\alpha+1}$
is injective. By \prettyref{assumptions:inductive-assumptions}, $\left(\psi_{\alpha},\rho_{\alpha}^{\alpha+1}\right)$
is a universal covariant representation of $X_{\alpha}$, so it suffices
now to construct a gauge action on the image of $\pi_{\alpha}$ and
apply the gauge-invariant uniqueness theorem. To do so, define for
each $z\in\mathbb{T}$ a unitary $U_{z}\in\mathcal{U}\left(H_{\alpha}\right)$
such that
\[
U_{z}\xi_{f,n}=z^{n_{\alpha}}\xi_{f,n}
\]
Then for $e\in\Lambda^{\omega^{\alpha}}$ and $\xi_{f,n}$ a basis
vector of $H_{\alpha}$ satisfying $r\left(f\right)=s\left(e\right)$,
if $f$ is $\alpha$-cancellative, $v\left(ef\right)_{\alpha}-v\left(f\right)_{\alpha}=v\left(e\right)_{\alpha}=1$.
If $f$ is not $\alpha$-cancellative, then $L\left(f\right)\leq\omega^{\alpha+1}$,
so in any case,
\[
U_{z}\pi_{\alpha}\left(T_{e}\right)U_{z}^{*}\xi_{f,n}=z^{-n_{\alpha}}U_{z}\pi_{\alpha}\left(T_{e}\right)\xi_{f,n}=z^{-n_{\alpha}}U_{z}\xi_{ef,m}=z\xi_{ef,m}=z\pi_{\alpha}\left(T_{e}\right)\xi_{f,n}
\]
where $m_{\alpha}=n_{\alpha}+1$. Similarly if $e\in\Lambda_{\alpha}$,
$\left(\mathrm{Ad}\:U_{z}\circ\pi_{\alpha}\right)\left(T_{e}\right)=\pi_{\alpha}\left(T_{e}\right)$.
Thus $\mathrm{Ad}\:U_{z}\circ\pi_{\alpha}=\pi_{\alpha}\circ\Gamma_{\alpha,z}$,
and since $\pi_{\alpha}\circ\rho_{\alpha}^{\alpha+1}$ is injective,
the gauge-invariant uniqueness theorem \cite[Theorem 6.4]{KATIDEAL2}
implies $\pi_{\alpha}$ is injective.

Thus $\pi\circ\rho_{\alpha+1}$ is injective for each $\alpha<\zeta$.
If $\zeta$ is a successor ordinal, this immediately implies $\pi\circ\rho_{\zeta}$
is injective. Otherwise, $\zeta$ is a limit ordinal, and $\mathcal{O}\left(\Lambda_{\zeta}\right)$
is the inductive limit of maps $\rho_{\alpha}^{\alpha+1}:\mathcal{O}\left(\Lambda_{\alpha}\right)\rightarrow\mathcal{O}\left(\Lambda_{\alpha+1}\right)$,
which are all injective by \prettyref{assumptions:inductive-assumptions}.
Thus $\mathcal{O}\left(\Lambda_{\zeta}\right)=\overline{\cup_{\alpha<\zeta}\mathcal{O}\left(\Lambda_{\alpha+1}\right)}$,
and since each $\pi\circ\rho_{\zeta}\circ\rho_{\alpha+1}^{\zeta}=\pi\circ\rho_{\alpha+1}$
is isometric, $\pi\circ\rho_{\zeta}$ is injective. Since $\pi\circ\rho_{\zeta}$
is injective, indeed $\rho_{\zeta}$ is injective. 
\end{proof}
As $\rho_{\zeta}$ is injective, we regard $\mathcal{O}\left(\Lambda_{\zeta}\right)$
as a subalgebra of $\mathcal{O}\left(\Lambda\right)$ for the rest
of this section.
\begin{cor}
\label{cor:extend-action}For $\alpha<\zeta$, define
\[
\mathcal{D}_{\alpha+1}=C^{*}\left(T_{p}T_{q}^{*}:p,q\in\Lambda_{\zeta},p\in\Lambda_{\alpha+1}\text{ iff }q\in\Lambda_{\alpha+1},p\Lambda^{*}\cup q\Lambda^{*}\subseteq C_{\alpha}\right)
\]
Then for each generator $T_{p}T_{q}^{*}$ and basis vector $\xi_{f,n}\in\ell^{2}\left(\partial\Lambda\times\mathbb{Z}^{[0,\zeta)}\right)$,
$\pi\left(T_{p}T_{q}^{*}\right)\xi_{f,n}=\xi_{pq^{-1}f,m}$ for $m\in\mathbb{Z}^{[0,\zeta)}$
satisfying $m_{\alpha}=n_{\alpha}-v\left(q\right)_{\alpha}+v\left(p\right)_{\alpha}$.
Moreover, there exists a continuous action $\overline{\Gamma}_{\alpha}:\mathbb{T}\rightarrow\mathrm{Aut}\left(\mathcal{D}_{\alpha+1}\right)$
defined by $\overline{\Gamma}_{\alpha,z}\left(T_{p}T_{q}^{*}\right)=z^{v\left(p\right)_{\alpha}-v\left(q\right)_{\alpha}}T_{p}T_{q}^{*}$
which agrees with $\Gamma_{\alpha}$ on $\mathcal{D}_{\alpha+1}\cap\mathcal{O}\left(\Lambda_{\alpha+1}\right)$.
\end{cor}

\begin{proof}
For each $z\in\mathbb{T}$, define a unitary $U_{z}$ on $H=\ell^{2}\left(\partial\Lambda\times\mathbb{Z}^{[0,\zeta)}\right)$
by $U_{z}\xi_{f,n}=z^{n_{\alpha}}\xi_{f,n}$. We will show that 
\[
U_{z}\pi\left(T_{p}T_{q}^{*}\right)U_{z}^{*}=\pi\left(z^{v\left(p\right)_{\alpha}-v\left(q\right)_{\alpha}}T_{p}T_{q}^{*}\right)
\]
for $p,q$ as in the definition of $\mathcal{D}_{\alpha+1}$. Since
$\pi$ is injective on $\mathcal{D}_{\alpha+1}\subseteq\mathcal{O}\left(\Lambda_{\zeta}\right)$
by \prettyref{lem:injective} and $z^{v\left(p\right)_{\alpha}-v\left(q\right)_{\alpha}}T_{p}T_{q}^{*}\in\mathcal{D}_{\alpha+1}$,
this will complete the proof that $\overline{\Gamma}_{\alpha}$ is
well-defined. Selecting an arbitrary basis vector $\xi_{f,n}$, we
have $\pi\left(T_{p}T_{q}^{*}\right)\xi_{f,n}=0$ unless $f\in q\partial\Lambda$,
in which case $f$ is $\alpha$-cancellative. Therefore, it suffices
to show that the operators agree on $\xi_{f,n}$ when $f$ is $\alpha$-cancellative.
Note that either $p$ and $q$ both belong to $\Lambda_{\alpha+1}$,
or $p$ and $q$ both belong to $\Lambda\backslash\Lambda_{\alpha+1}$.
We handle the two cases separately. If $p,q\in\Lambda_{\alpha+1}$,
then $v\left(pq^{-1}f\right)_{\alpha}=v\left(p\right)_{\alpha}-v\left(q\right)_{\alpha}+v\left(f\right)_{\alpha}$,
so
\begin{align*}
U_{z}\pi\left(T_{p}T_{q}^{*}\right)U_{z}^{*}\xi_{f,n} & =z^{-n_{\alpha}}U_{z}\pi\left(T_{p}T_{q}^{*}\right)\xi_{f,n}\\
 & =z^{-n_{\alpha}}U_{z}\xi_{pq^{-1}f,m}\\
 & =z^{-n_{\alpha}+n_{\alpha}+v\left(pq^{-1}f\right)_{\alpha}-v\left(f\right)_{\alpha}}\xi_{pq^{-1}f,m}\\
 & =z^{v\left(p\right)_{\alpha}-v\left(q\right)_{\alpha}}\pi\left(T_{p}T_{q}^{*}\right)\xi_{f,n}
\end{align*}
If $p,q\in\Lambda\backslash\Lambda_{\alpha+1}$, then $v\left(pq^{-1}f\right)_{\alpha}=v\left(p\right)_{\alpha}$,
and since $f\in q\partial\Lambda$, $v\left(f\right)_{\alpha}=v\left(q\right)_{\alpha}$.
Therefore, the same equations hold. Finally, to see that $\overline{\Gamma}_{\alpha}$
agrees with $\Gamma_{\alpha}$ on $D_{\alpha}\cap\mathcal{O}\left(\Lambda_{\alpha+1}\right)$,
note that $v\left(p\right)_{\alpha}=1$ for $p\in\Lambda^{\omega^{\alpha}}$.
\end{proof}
Next we tackle the problem of proving $J_{\zeta}$ is generated by
$\left\{ T_{v}:v\in\Lambda_{0}\text{ is }\zeta\text{-regular}\right\} $.
Since $\mathcal{O}\left(\Lambda_{0}\right)$ is fixed by $\Gamma_{\alpha}$
for every $\alpha$, the ideal generated by these $\zeta$-regular
vertices is gauge-invariant. Our strategy is to prove that $J_{\zeta}$
is also gauge-invariant and use the classification of gauge-invariant
ideals of Cuntz-Pimsner algebras.
\begin{lem}
\label{lem:awesome-lemma}If $\alpha<\zeta$, $a\in\mathcal{O}\left(\Lambda_{\alpha+1}\right)$,
$\varphi_{\zeta}\left(a\right)\in\mathcal{K}\left(X_{\zeta}\right)$,
and $h\in\Lambda_{\alpha+1}\backslash\Lambda_{\alpha}$ with $s\left(h\right)=r\left(h\right)$,
then $\left(\varphi_{\zeta}\left(aT_{h}^{n}\right)\right)_{n\in\mathbb{N}}$
has a convergent subsequence.
\end{lem}

\begin{proof}
It suffices to prove that for $\varepsilon>0$ there exists positive
$n\in\mathbb{N}$ such that for all positive $m\in\mathbb{N}$,
\[
\left\Vert \varphi_{\zeta}\left(aT_{h}^{n}\right)-\varphi_{\zeta}\left(aT_{h}^{mn}\right)\right\Vert \leq\varepsilon
\]
To see why, let $\left(\varepsilon_{k}\right)_{k\in\mathbb{N}}$ be
a sequence in $\mathbb{R}_{>0}$ converging to $0$, and for each
$\varepsilon_{k}$ select $n_{k}$ satisfying the above inequality.
Then the sequence
\[
\varphi_{\zeta}\left(aT_{h}^{n_{0}}\right),\varphi_{\zeta}\left(aT_{h}^{n_{0}n_{1}}\right),\varphi_{\zeta}\left(aT_{h}^{n_{0}n_{1}n_{2}}\right),\dots
\]
is Cauchy, and hence converges in $\mathcal{K}\left(X_{\zeta}\right)$.
Thus we fix $\varepsilon>0$ and focus on constructing $n$ satisfying
the desired inequality for all positive $m\in\mathbb{N}$.

Let $S=\varphi_{\zeta}\left(aT_{s\left(h\right)}\right)$. By \prettyref{cor:compactness},
there exists finite $F\subseteq\Lambda^{\omega^{\zeta}}$ such that
if 
\[
\eta\in\overline{\mathrm{span}}_{\mathcal{O}\left(\Lambda_{\zeta}\right)}\left\{ \delta_{g}:g\in\Lambda^{\omega^{\zeta}}\backslash F\right\} 
\]
 and $\left\Vert \eta\right\Vert \leq1$, $\left\Vert S\eta\right\Vert <\frac{1}{2}\varepsilon$.
By \prettyref{cor:compact-loop}, if $f\in\cap_{k\in\mathbb{N}}h^{k}\Lambda$
and $h^{k}f\not=f$ for all $k\in\mathbb{N}$, then $S\delta_{f}=0$.
Hence we may assume without loss of generality that $F$ contains
no such $f$. We also assume that each $f\in F$ satisfies $r\left(f\right)=s\left(h\right)$,
since otherwise $S\delta_{f}=0$. Hence for each $f\in F$, we have
one of two possibilities:
\begin{enumerate}
\item There exists $k>0$ such that $f\not\in h^{k}\Lambda$.
\item $f\in\cap_{k\in\mathbb{N}}h^{k}\Lambda$ and $h^{m}f=f$ for some
$m\in\mathbb{N}$.
\end{enumerate}
Suppose for $f\in F$, (1) occurs. Define $n_{f}>0$ such that $f\not\in h^{n_{f}}\Lambda$.
If (2) occurs for $f\in F$, define $n_{f}>0$ such that $h^{n_{f}}f=f$.
Select $n$ to be a positive common multiple of $\left\{ n_{f}:f\in F\right\} $.
Define $G=\left\{ g\in F:h^{m}g=g\text{ for some }m>0\right\} $.
Now we make the following observation:
\begin{claim*}
For $g\in\Lambda^{\omega^{\zeta}}$, $g\in G$ iff $h^{n}g\in G$.
Moreover, if $g\not\in G$, then $h^{n}g\not\in F$.
\end{claim*}
First assume $g\in G$ and $h^{m}g=g$ for $m>0$. Then $h^{mk}g=g$
for all $k\in\mathbb{N}$, and $g\in\cap_{k\in\mathbb{N}}h^{k}\Lambda$.
By construction of $n$, $h^{n}g=g\in G$. Now assume $g\not\in G$
and $f=h^{n}g\in F$. Since $f\in h^{n}\Lambda\subseteq h^{n_{f}}\Lambda$,
we must have by the construction of $n$ $f\in\cap_{k\in\mathbb{N}}h^{k}\Lambda$,
and hence that $h^{n}f=f$. Thus $h^{2n}g=h^{n}g$, and by left cancellation,
$h^{n}g=g\in F$. However, this contradicts the construction of $G$.
Since $G\subseteq F$, we also have $h^{n}g\not\in G$, proving the
claim.

Now define subspaces
\begin{align*}
C & =\overline{\mathrm{span}}_{\mathcal{O}\left(\Lambda_{\zeta}\right)}\left\{ \delta_{g}:g\in G\right\} \\
D & =\overline{\mathrm{span}}_{\mathcal{O}\left(\Lambda_{\zeta}\right)}\left\{ \delta_{g}:g\in\Lambda^{\omega^{\zeta}}\backslash G\right\} 
\end{align*}
By the above claim, $C$ and $D$ are closed orthogonal invariant
subspaces for the operator $Q=\varphi_{\zeta}\left(T_{h}^{n}\right)$.
For $\eta\in X_{\zeta}$ with $\left\Vert \eta\right\Vert \leq1$,
let $\eta=\eta_{1}+\eta_{2}$ with $\eta_{1}\in C$ and $\eta_{2}\in D$.
Since $C\perp D$, $\left\Vert \eta_{2}\right\Vert \leq1$. Then $\left\Vert Q\eta_{2}\right\Vert \leq1$
since $\left\Vert T_{h}^{n}\right\Vert \leq1$, and
\[
\left\{ Q\eta_{2},Q^{n}\eta_{2}\right\} \subseteq\overline{\mathrm{span}}_{\mathcal{O}\left(\Lambda_{\zeta}\right)}\left\{ \delta_{g}:g\in\Lambda^{\omega^{\zeta}}\backslash F\right\} 
\]
Then by the definition of $F$,
\begin{align*}
\left\Vert SQ\eta-SQ^{n}\eta\right\Vert  & \leq\left\Vert SQ\eta_{1}-SQ^{n}\eta_{1}\right\Vert +\left\Vert SQ\eta_{2}-SQ^{n}\eta_{2}\right\Vert \\
 & =\left\Vert SQ\eta_{2}-SQ^{n}\eta_{2}\right\Vert \\
 & \leq\left\Vert SQ\eta_{2}\right\Vert +\left\Vert SQ^{n}\eta_{2}\right\Vert \\
 & <\frac{\varepsilon}{2}+\frac{\varepsilon}{2}=\varepsilon
\end{align*}
\end{proof}
\begin{lem}
\label{lem:non-returning-zero}If $\alpha<\zeta$ and $a\in\mathcal{O}\left(\Lambda_{\alpha+1}\right)\cap J_{\zeta}$,
then for each $f\in\Lambda^{\omega^{\zeta}}$ which is not $\alpha$-cancellative
and $z\in\mathbb{T}$, $\left(\varphi_{\zeta}\circ\Gamma_{\alpha,z}\right)\left(a\right)\delta_{f}=0$.
\end{lem}

\begin{proof}
Let $\alpha<\zeta$ and $a\in\mathcal{O}\left(\Lambda_{\alpha+1}\right)\cap J_{\zeta}$
be given. Let $f\in\Lambda^{\omega^{\zeta}}$ such that $f$ is not
$\alpha$-cancellative, and by \prettyref{lem:returning-loops}, choose
$g,h\in\Lambda_{\alpha+1}$ such that $d\left(h\right)\geq\omega^{\alpha}$
and $hg^{-1}f=g^{-1}f$. Since $g\in\Lambda_{\alpha+1}$ and $J_{\zeta}$
is an ideal, we have $aT_{g}\in\mathcal{O}\left(\Lambda_{\alpha+1}\right)\cap J_{\zeta}$.
Moreover, $g^{-1}f$ is not $\alpha$-cancellative, and $\varphi_{\zeta}\left(aT_{g}\right)\delta_{g^{-1}f}=\varphi_{\zeta}\left(a\right)\delta_{f}$.
Thus it suffices to prove $\varphi_{\zeta}\left(a\right)\delta_{f}=0$
if $hf=f$, and we assume without loss of generality $hf=f$.

Applying \prettyref{lem:awesome-lemma}, we have a convergent sequence
$\varphi_{\zeta}\left(aT_{h}^{n_{k}}\right)_{k\in\mathbb{N}}$ with
$\left(n_{k}\right)_{k\in\mathbb{N}}$ increasing. Since $\varphi_{\zeta}$
is injective on $J_{\zeta}$, $\left(aT_{h}^{n_{k}}\right)_{k\in\mathbb{N}}$
converges to some $b\in\mathcal{O}\left(\Lambda_{\alpha+1}\right)$.
Since $\Gamma_{\alpha}$ defined in \prettyref{lem:alg-auts} is continuous,
we have
\begin{equation}
\lim_{\varepsilon\rightarrow0}\varepsilon^{-1}\int_{0}^{\varepsilon}\Gamma_{\alpha,e^{it}}\left(b\right)\:dt=\Gamma_{\alpha,1}\left(b\right)=b\label{eq:lim_int}
\end{equation}
On the other hand, $d\left(h\right)\geq\omega^{\alpha}$ so there
exists positive $m\in\mathbb{N}$ such that
\begin{align*}
\int_{0}^{\varepsilon}\Gamma_{\alpha,e^{it}}\left(b\right)\:dt & =\int_{0}^{\varepsilon}\Gamma_{\alpha,e^{it}}\left(\lim_{k\rightarrow\infty}aT_{h}^{n_{k}}\right)\:dt=\int_{0}^{\varepsilon}\lim_{k\rightarrow\infty}e^{imn_{k}t}\Gamma_{\alpha,e^{it}}\left(a\right)T_{h}^{n_{k}}\:dt\\
 & =\lim_{k\rightarrow\infty}\int_{0}^{\varepsilon}e^{imn_{k}t}\Gamma_{\alpha,e^{it}}\left(a\right)\:dt\:T_{h}^{n_{k}}=0
\end{align*}
Since $\left\Vert e^{imn_{k}t}\Gamma_{\alpha,e^{it}}\left(a\right)T_{h}^{n_{k}}\right\Vert =\left\Vert a\right\Vert $,
we apply the dominated convergence theorem to pull the limit out of
the integral. The final equality follows from the Riemann-Lebesgue
lemma. Hence by \prettyref{eq:lim_int}, $b=0$. We also have
\[
0=\Gamma_{\alpha,z}\left(b\right)=\lim_{k\rightarrow\infty}z^{mn_{k}}\Gamma_{\alpha,z}\left(a\right)T_{h}^{n_{k}}
\]
Therefore $\left(\Gamma_{\alpha,z}\left(a\right)T_{h}^{n_{k}}\right)_{k\in\mathbb{N}}$
also converges to $0$, and
\[
\left(\varphi_{\zeta}\circ\Gamma_{\alpha,z}\right)\left(a\right)\delta_{f}=\varphi_{\zeta}\left(\Gamma_{\alpha,z}\left(a\right)T_{h}^{n_{k}}\right)\delta_{f}=0
\]
\end{proof}
\begin{lem}
\label{lem:katsura-ideal-generators}For all $\alpha<\zeta$ and $V\subseteq[\alpha,\zeta)$,
$\mathcal{O}\left(\Lambda_{\alpha}\right)\cap\bigcap_{\gamma\in V}J_{\gamma}$
is the smallest ideal in $\mathcal{O}\left(\Lambda_{\alpha}\right)$
containing $\left\{ T_{v}:v\in\Lambda_{0}\text{ is }\gamma\text{-regular for all }\gamma\in V\right\} $.
\end{lem}

\begin{proof}
For each $\delta<\zeta$, define $\mathcal{J}_{\delta}=\mathcal{O}\left(\Lambda_{\delta}\right)\cap\bigcap_{\gamma\in V}J_{\gamma}$.
If $\delta$ is a successor ordinal, then $\mathcal{O}\left(\Lambda_{\delta}\right)$
is a Cuntz-Pimsner algebra and $\mathcal{J}_{\delta}$ gauge-invariant
by \prettyref{assumptions:inductive-assumptions}. We use transfinite
induction on the classification of gauge-invariant ideals to prove
that $\mathcal{J_{\alpha}}$ is the ideal in $\mathcal{O}\left(\Lambda_{\alpha}\right)$
generated by $\left\{ T_{v}:v\in\Lambda_{0}\text{ is }\gamma\text{-regular for all }\gamma\in V\right\} $.
Suppose $\beta\leq\alpha$ and for all $\delta<\beta$, $\mathcal{J_{\delta}}$
is the smallest gauge-invariant ideal of $\mathcal{O}\left(\Lambda_{\delta}\right)$
containing $\mathcal{J}_{0}$. We will show $\mathcal{J}_{\beta}$
is the smallest gauge-invariant ideal in $\mathcal{O}\left(\Lambda_{\beta}\right)$
containing $\mathcal{J}_{0}$. 

Clearly this is true if $\beta=0$, so assume $\beta>0$. If $\beta$
is limit ordinal, then $\mathcal{O}\left(\Lambda_{\beta}\right)=\overline{\bigcup_{\delta<\beta}\mathcal{O}\left(\Lambda_{\delta}\right)}$,
hence by \cite[II.8.2.4]{Encyclopaedia}, $\mathcal{J}_{\beta}=\overline{\bigcup_{\delta<\beta}\mathcal{J}_{\delta}}$.
For an ideal $\mathcal{I}$ of $\mathcal{O}\left(\Lambda_{\beta}\right)$
containing $\mathcal{J}_{0}$, we have $\mathcal{J}_{\delta}\subseteq\mathcal{I}$
for every $\delta<\beta$. Therefore $\mathcal{J}_{\beta}\subseteq\mathcal{I}$,
and $\mathcal{J}_{\beta}$ is generated in $\mathcal{O}\left(\Lambda_{\beta}\right)$
by $\mathcal{J}_{0}$.

Otherwise, let $\beta=\delta+1$ be a successor ordinal. By \cite[Theorem 8.6]{GAUGE-INV-IDEALS},
the gauge-invariant ideals $\mathcal{I}$ of $\mathcal{O}\left(\Lambda_{\beta}\right)$
are uniquely determined by $\mathcal{O}\left(\Lambda_{\delta}\right)\cap\mathcal{I}$
and $\mathcal{O}\left(\Lambda_{\delta}\right)\cap\left(\mathcal{I}+\psi_{\delta}^{(1)}\left(\mathcal{K}\left(X_{\delta}\right)\right)\right)$.
If $\mathcal{I}$ is gauge-invariant, then
\[
J_{\delta}=\psi_{\delta}^{(1)}\left(\mathcal{K}\left(X_{\delta}\right)\right)\cap\mathcal{O}\left(\Lambda_{\delta}\right)\subseteq\left(\mathcal{I}+\psi_{\delta}^{(1)}\left(\mathcal{K}\left(X_{\delta}\right)\right)\right)\cap\mathcal{O}\left(\Lambda_{\delta}\right)
\]
By \prettyref{assumptions:inductive-assumptions},
\[
\mathcal{O}\left(\Lambda_{\delta}\right)\cap\left(\mathcal{J}_{\beta}+\psi_{\delta}^{(1)}\left(\mathcal{K}\left(X_{\delta}\right)\right)\right)\subseteq J_{\delta}
\]
Therefore, $\mathcal{O}\left(\Lambda_{\delta}\right)\cap\left(\mathcal{J}_{\beta}+\psi_{\delta}^{(1)}\left(\mathcal{K}\left(X_{\delta}\right)\right)\right)$
is as small as possible, and $\mathcal{J}_{\beta}$ is the smallest
gauge-invariant ideal containing $\mathcal{J}_{\beta}\cap\mathcal{O}\left(\Lambda_{\delta}\right)=\mathcal{J}_{\delta}$.
Since $\mathcal{J}_{\delta}$ is the smallest gauge-invariant ideal
in $\mathcal{O}\left(\Lambda_{\delta}\right)$ containing $\mathcal{J}_{0}$,
$\mathcal{J}_{\beta}$ is the smallest gauge-invariant ideal containing
$\mathcal{J}_{0}$, and by \prettyref{assumptions:inductive-assumptions},
the smallest gauge-invariant ideal containing $\left\{ T_{v}:v\in\Lambda_{0}\text{ is }\gamma\text{-regular for all }\gamma\in V\right\} $.
This set is gauge-invariant, hence $\mathcal{J}_{\beta}$ is the smallest
ideal generated by $\mathcal{J}_{0}$. Then by transfinite induction,
$\mathcal{J}_{\alpha}$ is generated by $\left\{ T_{v}:v\in\Lambda_{0}\text{ is }\gamma\text{-regular for all }\gamma\in V\right\} $.
\end{proof}
\begin{prop}
\label{prop:katsura-ideal-form}For all $\alpha<\zeta$ and $V\subseteq[\alpha,\zeta)$,
\[
\mathcal{O}\left(\Lambda_{\alpha}\right)\cap\bigcap_{\gamma\in V}J_{\gamma}=\overline{\mathrm{span}}\left\{ T_{p}T_{q}^{*}:p,q\in\Lambda_{\beta},s\left(p\right)=s\left(q\right)\text{ is }\gamma\text{-regular for all }\gamma\in V\right\} 
\]
\end{prop}

\begin{proof}
Let $\mathcal{I}$ be the set $\overline{\mathrm{span}}\left\{ T_{p}T_{q}^{*}:p,q\in\Lambda_{\beta},s\left(p\right)=s\left(q\right)\text{ is }\gamma\text{-regular}\right\} $.
In particular, we have $T_{v}=T_{v}^{*}T_{v}\in\mathcal{I}$ for each
$\gamma\in V$ and $v\in\Lambda_{0}$ which is $\gamma$-regular.
By \prettyref{lem:katsura-ideal-generators}, $\mathcal{J}=\mathcal{O}\left(\Lambda_{\beta}\right)\cap\bigcap_{\gamma\in V}J_{\gamma}$
is the smallest ideal of $\mathcal{O}\left(\Lambda_{\beta}\right)$
containing $\left\{ T_{v}:v\in\Lambda_{0},v\text{ is }\gamma\text{-regular for all }\gamma\in V\right\} $,
so if we prove $\mathcal{I}$ is an ideal in $\mathcal{O}\left(\Lambda_{\beta}\right)$,
we have proven $\mathcal{J}\subseteq\mathcal{I}$.

Let $p,q,r\in\Lambda_{\beta}$ such that $s\left(p\right)=s\left(q\right)$
is $\gamma$-regular. If $T_{r}T_{p}T_{q}^{*}\not=0$, then $s\left(r\right)=r\left(p\right)$,
$T_{r}T_{p}T_{q}^{*}=T_{rp}T_{q}^{*}$, and $s\left(rp\right)=s\left(p\right)$
is $\gamma$-regular. Likewise, if $T_{p}T_{q}^{*}T_{r}\not=0$, then
either $r\in q\Lambda_{\beta}$ or $q\in r\Lambda_{\beta}$. Suppose
$r=qh$ for some $h\in\Lambda_{\beta}$. Then $T_{p}T_{q}^{*}T_{r}=T_{p}T_{h}=T_{ph}$.
Then since $\Lambda_{\beta}\subseteq\Lambda_{\gamma}$, \prettyref{lem:regular-hereditary}
implies $s\left(ph\right)=s\left(h\right)$ is $\gamma$-regular.
Moreover, if $q=rh$, then $T_{p}T_{q}^{*}T_{r}=T_{p}T_{h}^{*}$ and
$s\left(h\right)=s\left(q\right)$ is $\gamma$-regular. Since $T_{r}^{*}T_{p}T_{q}^{*}=\left(T_{q}T_{p}^{*}T_{r}\right)^{*}$
and $T_{p}T_{q}^{*}T_{r}^{*}=\left(T_{r}T_{q}T_{p}^{*}\right)^{*}$,
\prettyref{prop:ordgraph-closed-span} implies $\mathcal{I}$ is an
ideal and $\mathcal{J}\subseteq\mathcal{I}$. Finally, note that for
each $p,q\in\Lambda_{\beta}$ such that $s\left(p\right)=s\left(q\right)$
is $\gamma$-regular, $T_{p}T_{q}^{*}=T_{p}T_{s\left(p\right)}T_{q}^{*}$,
and therefore $\mathcal{I}\subseteq\mathcal{J}$.
\end{proof}
\begin{prop}
\label{prop:conditional-expectations}For fixed $\alpha<\zeta$, define
for each $\beta\in(\alpha,\zeta]$ the following subalgebras of $\mathcal{O}\left(\Lambda_{\zeta}\right)$
\begin{align*}
\mathcal{A}_{\beta}=\overline{\mathrm{span}}\{ & T_{p}T_{q}^{*}:p,q\in\Lambda_{\zeta},v\left(p\right)_{\theta}=v\left(q\right)_{\theta}\text{ for all }\theta\in[\beta,\zeta),\\
 & p\Lambda^{*}\cup q\Lambda^{*}\subseteq C_{\gamma}\text{ for all }\gamma+1\in[\alpha+1,\zeta)\}
\end{align*}
Then for all $\alpha<\gamma\leq\beta\leq\zeta$, there exist commuting
conditional expectations $E_{\gamma}^{\beta}:\mathcal{A}_{\beta}\rightarrow\mathcal{A}_{\gamma}$
satisfying for $p,q$ specified as above
\[
E_{\gamma}^{\beta}\left(T_{p}T_{q}^{*}\right)=\begin{cases}
T_{p}T_{q}^{*} & \text{if }v\left(p\right)_{\theta}=v\left(q\right)_{\theta}\text{ for all }\theta\in[\gamma,\beta)\\
0 & \text{otherwise}
\end{cases}
\]
\end{prop}

\begin{proof}
We find it useful to first define for $\beta\in(\alpha,\zeta]$
\begin{align*}
\mathcal{B}_{\beta}=\overline{\mathrm{span}}\{ & T_{p}T_{q}^{*}:p,q\in\Lambda_{\zeta},p\in\Lambda_{\theta}\text{ iff }q\in\Lambda_{\theta}\text{ for all }\theta\in[\beta,\zeta),\\
 & p\Lambda^{*}\cup q\Lambda^{*}\subseteq C_{\gamma}\text{ for all }\gamma+1\in[\alpha+1,\zeta)\}
\end{align*}
Note that $\mathcal{B}_{\beta+1}\subseteq\mathcal{D}_{\beta+1}$ for
the algebra $\mathcal{D}_{\beta+1}$ defined in \prettyref{cor:extend-action}.
Now we show $\mathcal{A}_{\beta}\subseteq\mathcal{B}_{\beta}$. Let
$p,q\in\Lambda_{\zeta}$ be chosen so that $T_{p}T_{q}^{*}$ is a
generator for $\mathcal{A}_{\beta}$. Then for all $\theta\in[\beta,\zeta)$,
$v\left(p\right)_{\theta}=v\left(q\right)_{\theta}$. Suppose for
some $\theta\in[\beta,\zeta)$, $p\in\Lambda_{\theta}$ and $q\in\Lambda\backslash\Lambda_{\theta}$.
Then there exists $\gamma\in[\theta,\zeta)$ such that $q\in\Lambda_{\gamma+1}\backslash\Lambda_{\gamma}$.
Then $v\left(p\right)_{\gamma}=0\not=v\left(q\right)_{\gamma}$, a
contradiction, and hence $\mathcal{A}_{\beta}\subseteq\mathcal{B}_{\beta}$. 

Now we show $\mathcal{A}_{\beta}$ is indeed a $C^{*}$-algebra. Let
$p_{1},q_{1},p_{2},q_{2}\in\Lambda_{\zeta}$ be given as in the definition
of $\mathcal{A}_{\beta}$. By the preceding argument, $p_{1}\in\Lambda_{\theta+1}$
iff $q_{1}\in\Lambda_{\theta+1}$. Assume without loss of generality
that $T_{p_{1}}T_{q_{1}}^{*}T_{p_{2}}T_{q_{2}}^{*}=T_{p_{1}q_{1}^{-1}p_{2}}T_{q_{2}}^{*}$.
Then assuming $\theta+1<\zeta$, $p_{2}\Lambda^{*}\subseteq C_{\theta}$.
Hence $p_{2}$ is $\theta$-cancellative, and by \prettyref{cor:non-returning-equiv}
\[
v\left(p_{1}q_{1}^{-1}p_{2}\right)_{\theta}=\begin{cases}
v\left(p_{1}\right)_{\theta} & p_{1},q_{1}\in\Lambda\backslash\Lambda_{\theta+1}\\
v\left(p_{1}\right)_{\theta}-v\left(q_{1}\right)_{\theta}+v\left(p_{2}\right)_{\theta} & p_{1},q_{1}\in\Lambda_{\theta+1}
\end{cases}
\]
In the first case, $p_{2}\in q_{1}\Lambda$, hence $v\left(p_{1}\right)_{\theta}=v\left(q_{1}\right)_{\theta}=v\left(p_{2}\right)_{\theta}=v\left(q_{2}\right)_{\theta}$.
Otherwise, $v\left(p_{1}q_{1}^{-1}p_{2}\right)_{\theta}=v\left(p_{2}\right)_{\theta}=v\left(q_{2}\right)_{\theta}$.
If $\theta+1=\zeta$, then by \prettyref{lem:cancellative-length-independent}
we again have $v\left(p_{1}q_{1}^{-1}p_{2}\right)_{\theta}=v\left(p_{1}\right)_{\theta}-v\left(q_{1}\right)_{\theta}+v\left(p_{2}\right)_{\theta}=v\left(q_{2}\right)_{\theta}$.
Since $p_{1}q_{1}^{-1}p_{2}C_{\theta}\subseteq p_{1}C_{\theta}$,
$T_{p_{1}}T_{q_{1}}^{*}T_{p_{2}}T_{q_{2}}^{*}\in\mathcal{A}_{\beta}$,
and $\mathcal{A}_{\beta}$ is a $C^{*}$-algebra densely spanned by
such $T_{p}T_{q}^{*}$.

Instead of defining each $E_{\gamma}^{\beta}$ on the domain $\mathcal{A}_{\beta}$,
we first define it as a map $E_{\gamma}^{\beta}:\mathcal{B}_{\beta}\rightarrow\mathcal{B}_{\gamma}$.
We construct each $E_{\gamma}^{\beta}$ using transfinite recursion
starting at $\alpha$. Since $E_{\alpha+1}^{\alpha+1}=\mathrm{id}$,
the base case is satisfied. Suppose for some $\beta$, $E_{\gamma}^{\theta}$
exists for all $\alpha<\gamma\leq\theta<\beta$. First we consider
the case in which $\beta=\nu+1$ is a successor ordinal and $\beta<\zeta$.
Then $\mathcal{B}_{\beta}\subseteq\mathcal{D}_{\beta}$, and by \prettyref{cor:extend-action},
there is an action $\overline{\Gamma}_{\nu}:\mathbb{T}\rightarrow\mathrm{Aut}\left(\mathcal{D}_{\beta}\right)$
defined by $\overline{\Gamma}_{\nu,z}\left(T_{p}T_{q}^{*}\right)=z^{v\left(p\right)_{\nu}-v\left(q\right)_{\nu}}T_{p}T_{q}^{*}$,
so we may define
\[
E_{\nu}^{\beta}\left(a\right)=\int_{\mathbb{T}}\overline{\Gamma}_{\nu,z}\left(a\right)\:dz
\]
where we use the normalized Haar measure on $\mathbb{T}$. If $\beta=\nu+1=\zeta$,
then we instead define $E_{\nu}^{\beta}\left(a\right)=\int_{\mathbb{T}}\Gamma_{\nu,z}\left(a\right)\:dz$
using the usual gauge action defined in \prettyref{lem:alg-auts}.
By \prettyref{assumptions:inductive-assumptions}, $\mathcal{O}\left(\Lambda_{\zeta}\right)$
is a Cuntz-Pimsner algebra, and this is the usual conditional expectation
onto the fixed point algebra. Note that $E_{\nu}^{\beta}$ indeed
maps into $\mathcal{B}_{\nu}$. To see why, let $T_{p}T_{q}^{*}$
be a generator of $\mathcal{B}_{\beta}$. Then $E_{\nu}^{\beta}\left(T_{p}T_{q}^{*}\right)=0$
unless $v\left(p\right)_{\nu}=v\left(q\right)_{\nu}$. In that case,
supposing $p\in\Lambda_{\nu}$, we have $v\left(p\right)_{\nu}=v\left(q\right)_{\nu}$.
Since $q\in\Lambda_{\nu+1}$, this implies $q\in\Lambda_{\nu}$, and
$T_{p}T_{q}^{*}\in\mathcal{B}_{\nu}$. 

Now we handle the case in which $\beta$ is a limit ordinal. If $p,q\in\Lambda_{\beta}$,
then there exists $\theta<\beta$ such that $p,q\in\Lambda_{\theta}$,
and therefore $\mathcal{B}_{\beta}=\overline{\cup_{\theta<\beta}\mathcal{B}_{\theta}}$.
For a generator $T_{p}T_{q}^{*}$ of $\mathcal{B}_{\theta}$, $\left(E_{\gamma}^{\nu}\left(T_{p}T_{q}^{*}\right)\right)_{\nu<\beta}$
is eventually constant, and thus we may define $E_{\gamma}^{\beta}\left(a\right)=\lim_{\nu<\beta}E_{\beta}^{\nu}\left(a\right)$
for $a\in\cup_{\theta<\beta}\mathcal{B}_{\theta}$. For $a$ in the
algebraic span of generators $T_{p}T_{q}^{*}$, there is $\nu<\beta$
such that $E_{\gamma}^{\beta}\left(a\right)=E_{\gamma}^{\nu}\left(a\right)$.
Since each $E_{\gamma}^{\nu}$ is contractive, it follows that $E_{\gamma}^{\beta}$
is contractive, and that $E_{\gamma}^{\beta}$ may be extended to
a map on $\mathcal{B}_{\beta}$.

Now that we have defined each $E_{\gamma}^{\beta}$, we restrict the
domain of $E_{\gamma}^{\beta}$ to $\mathcal{A}_{\beta}$. Then the
range of $E_{\gamma}^{\beta}$ is $\mathcal{A}_{\gamma}$, and each
map $E_{\beta}^{\beta+1}$ is a conditional expectation. It follows
that if $E_{\gamma}^{\beta}$ is a conditional expectation, then $E_{\gamma}^{\beta+1}=E_{\gamma}^{\beta}\circ E_{\beta}^{\beta+1}$
is a conditional expectation. Finally, if $\beta$ is a limit ordinal
and for each $\theta<\beta$, $E_{\gamma}^{\theta}$ is a conditional
expectation, then for each $a\in\cup_{\theta<\beta}\left(\mathcal{B}_{\theta}\cap\mathcal{\mathcal{A}_{\beta}}\right)$
and $b,c\in\mathcal{A}_{\gamma}$, 
\[
E_{\gamma}^{\beta}\left(bac\right)=\lim_{\nu<\beta}E_{\gamma}^{\nu}\left(bac\right)=\lim_{\nu<\beta}bE_{\gamma}^{\nu}\left(a\right)c=bE_{\gamma}^{\beta}\left(a\right)c
\]
Thus $E_{\gamma}^{\beta}$ is a conditional expectation for every
$\alpha<\gamma\leq\beta\leq\zeta$. 
\end{proof}
\begin{lem}
The conditional expectations $E_{\gamma}^{\beta}:\mathcal{A}_{\beta}\rightarrow\mathcal{A}_{\gamma}$
defined in \prettyref{prop:conditional-expectations} are faithful.
\end{lem}

\begin{proof}
Since the expectations of the form $E_{\beta}^{\beta+1}$ arise by
averaging over an action by $\mathbb{T}$, we would be done if not
for the case $E_{\gamma}^{\beta}$ for a limit ordinal $\beta$. Instead,
we first prove that for every $\beta\in[\alpha+1,\zeta)$, $a\in\mathcal{A}_{\beta}$,
and pair of basis vectors $\xi_{f,n},\xi_{g,m}\in\ell^{2}\left(\partial\Lambda\times\mathbb{Z}^{[0,\zeta)}\right)$,
\begin{equation}
\left\langle \xi_{f,n},\pi\left(E_{\gamma}^{\beta}\left(a\right)\right)\xi_{g,m}\right\rangle =\begin{cases}
\left\langle \xi_{f,n},\pi\left(a\right)\xi_{g,m}\right\rangle  & n_{\theta}=m_{\theta}\text{ for all }\theta\in[\gamma,\beta)\\
0 & \text{otherwise}
\end{cases}\label{eq:inner_product}
\end{equation}
Clearly this equation is true when $\beta=\gamma=\alpha+1$. In order
to apply transfinite induction, suppose $\theta\in(\alpha,\zeta]$
and for all $\alpha<\gamma\leq\beta<\theta$, $E_{\gamma}^{\beta}$
satisfies \prettyref{eq:inner_product}. By linearity and continuity,
it suffices to prove \prettyref{eq:inner_product} when $a=T_{p}T_{q}^{*}$
for $p,q\in\Lambda_{\zeta}$ satisfying $v\left(p\right)_{\eta}=v\left(q\right)_{\eta}$
for all $\eta\in[\theta,\zeta)$ and $p\Lambda^{*}\cup q\Lambda^{*}\subseteq C_{\gamma}$
for all $\gamma+1\in[\alpha+1,\zeta)$. If $\theta=\nu+1$, then since
$\mathcal{A}_{\theta}\subseteq\mathcal{D}_{\theta}$, \prettyref{cor:extend-action}
implies
\begin{align*}
\left\langle \xi_{f,n},\pi\left(E_{\nu}^{\theta}\left(T_{p}T_{q}^{*}\right)\right)\xi_{g,m}\right\rangle  & =\int_{\mathbb{T}}\left\langle \xi_{f,n},U_{z}\pi\left(T_{p}T_{q}^{*}\right)U_{z}^{*}\xi_{g,m}\right\rangle \:dz\\
 & =\int_{\mathbb{T}}z^{m_{\nu}}\left\langle U_{z}^{*}\xi_{f,n},\pi\left(T_{p}T_{q}^{*}\right)\xi_{g,m}\right\rangle \:dz\\
 & =\int_{\mathbb{T}}z^{m_{\nu}-n_{\nu}}\left\langle \xi_{f,n},\pi\left(T_{p}T_{q}^{*}\right)\xi_{g,m}\right\rangle \:dz\\
 & =\begin{cases}
\left\langle \xi_{f,n},\pi\left(T_{p}T_{q}^{*}\right)\xi_{g,m}\right\rangle  & m_{\nu}=n_{\nu}\\
0 & \text{otherwise}
\end{cases}
\end{align*}
And since $E_{\gamma}^{\theta}=E_{\gamma}^{\nu}\circ E_{\nu}^{\theta}$,
\prettyref{eq:inner_product} holds for all $E_{\gamma}^{\theta}$.
If $\theta$ is a limit ordinal, then since $p\in\Lambda_{\theta}$
iff $q\in\Lambda_{\theta}$, there exists a successor ordinal $\nu<\theta$
such that $p\in\Lambda_{\eta}$ iff $q\in\Lambda_{\eta}$ for all
$\eta\in[\nu,\theta)$. Thus $T_{p}T_{q}^{*}\in\mathcal{D}_{\eta}$
for all such $\eta$. Since $\left(E_{\gamma}^{\eta}\left(T_{p}T_{q}^{*}\right)\right)_{\eta}$
is eventually constant, we may choose $\eta$ such that $E_{\gamma}^{\eta}\left(T_{p}T_{q}^{*}\right)=E_{\gamma}^{\beta}\left(T_{p}T_{q}^{*}\right)$,
and applying the inductive hypothesis, we see \prettyref{eq:inner_product}
holds for $E_{\gamma}^{\beta}$ as well.

Now, by \prettyref{eq:inner_product}, we see
\begin{align*}
\left\langle \xi_{f,n},\pi\left(E_{\gamma}^{\beta}\left(a^{*}a\right)\right)\xi_{f,n}\right\rangle  & =\left\langle \xi_{f,n},\pi\left(a^{*}a\right)\xi_{f,n}\right\rangle \\
 & =\left\langle \pi\left(a\right)\xi_{f,n},\pi\left(a\right)\xi_{f,n}\right\rangle 
\end{align*}
Therefore, if $E_{\gamma}^{\beta}\left(a^{*}a\right)=0$, then $\pi\left(a\right)=0$,
and hence $a^{*}a=0$ by injectivity of $\pi$. All that remains is
the case for $\beta=\zeta$. If $\zeta=\nu+1$, then $E_{\gamma}^{\zeta}=E_{\gamma}^{\nu}\circ E_{\nu}^{\zeta}$,
and $E_{\nu}^{\zeta}$ is faithful. If on the other hand $\zeta$
is a limit ordinal, then \prettyref{eq:inner_product} holds by the
same arguments as before, and $E_{\gamma}^{\zeta}$ is faithful.
\end{proof}
\begin{lem}
\label{lem:norm-in-next-alg}If $\alpha<\zeta$ and $f\in\Lambda^{\omega^{\zeta}}$,
then for $\eta\in\overline{\mathrm{span}}_{\mathbb{C}}\left\{ \delta_{e}:e\in\Lambda^{\omega^{\alpha}}r\left(f\right)\right\} $,
\[
\|\eta\|=\|\varphi_{\zeta}\left(\psi_{\alpha}\left(\eta\right)\right)\delta_{f}\|
\]
\end{lem}

\begin{proof}
Let $f\in\Lambda^{\omega^{\zeta}}$ be arbitrary. It suffices to check
the equality for $\eta$ of the form
\[
\eta=\sum_{e\in\Lambda^{\omega^{\alpha}}r\left(f\right)}\lambda_{e}\delta_{e}
\]
with $\lambda_{e}\in\mathbb{C}$. Then
\begin{align*}
\left\Vert \eta\right\Vert  & =\left\Vert \sum_{e\in\Lambda^{\omega^{\alpha}}r\left(f\right)}\lambda_{e}\delta_{e}\right\Vert \\
\left\Vert \varphi_{\zeta}\left(\psi_{\alpha}\left(\eta\right)\right)\delta_{f}\right\Vert  & =\left\Vert \sum_{e\in\Lambda^{\omega^{\alpha}}r\left(f\right)}\lambda_{e}\delta_{ef}\right\Vert 
\end{align*}
For $e_{1},e_{2}\in\Lambda^{\omega^{\alpha}}$, note that $e_{1}f=e_{2}f$
implies $\left(e_{1}f\right)\left(\omega^{\alpha}\right)=e_{1}=\left(e_{2}f\right)\left(\omega^{\alpha}\right)=e_{2}$.
Thus by \prettyref{lem:left-action-isometry}, both norms are equal
to
\[
\left(\sum_{e\in\Lambda^{\omega^{\alpha}}r\left(f\right)}\left|\lambda_{e}\right|^{2}\right)^{1/2}
\]
\end{proof}
\begin{rem}
In the above proof, it is crucial that $\eta$ belongs to the $\mathbb{C}$-span
of $\left\{ \delta_{e}:e\in\Lambda^{\omega^{\alpha}}r\left(f\right)\right\} $,
and not just the $\mathcal{O}\left(\Lambda_{\alpha}\right)$-span.
\end{rem}

\begin{lem}
\label{lem:gauge-containment}For each $\alpha<\zeta$, $\mathcal{O}\left(\Lambda_{\alpha}\right)\cap\left(J_{\zeta}+\psi_{\alpha}^{(1)}\left(\mathcal{K}\left(X_{\alpha}\right)\right)\right)\subseteq J_{\alpha}$.
In particular, $\mathcal{O}\left(\Lambda_{\alpha}\right)\cap J_{\zeta}\subseteq J_{\alpha}$.
\end{lem}

\begin{proof}
Let $\varepsilon>0$ be given and $a=b+c\in\mathcal{O}\left(\Lambda_{\alpha}\right)$
for $b\in J_{\zeta}$ and $c=\psi_{\alpha}^{(1)}\left(S\right)$.
We apply \prettyref{cor:compactness} part (3) to prove $\varphi_{\alpha}\left(a\right)P_{v}\in\mathcal{K}\left(X_{\alpha}\right)$
and then part (2) to prove $\varphi_{\alpha}\left(a\right)\in\mathcal{K}\left(X_{\alpha}\right)$.
Choose finite $F\subseteq\Lambda^{\omega^{\zeta}}$ such that for
all $\eta\in\mathrm{span}_{\mathbb{C}}\left\{ \delta_{f}:f\in\Lambda^{\omega^{\zeta}}\backslash F\right\} $
with $\left\Vert \eta\right\Vert \leq1$, $\left\Vert \varphi_{\zeta}\left(b\right)\eta\right\Vert <\frac{1}{2}\varepsilon$.
In addition, choose finite $G\subseteq\Lambda^{\omega^{\alpha}}$
such that for all $\mu\in\mathrm{span}_{\mathbb{C}}\left\{ \delta_{g}:g\in\Lambda^{\omega^{\alpha}}\backslash G\right\} $
with $\left\Vert \mu\right\Vert \leq1$, $\left\Vert S\mu\right\Vert <\frac{1}{2}\varepsilon$.
Define $H=\left\{ f\left(\omega^{\alpha}\right):f\in F\right\} $
and let $\xi\in\mathrm{span}_{\mathbb{C}}\left\{ \delta_{h}:h\in\Lambda^{\omega^{\alpha}}\backslash\left(G\cup H\right)\right\} $
with $\left\Vert \xi\right\Vert \leq1$. If $v\in\Lambda_{0}$ and
$v\Lambda^{\omega^{\zeta}}=\emptyset$, then $T_{h}\in\ker\varphi_{\zeta}$
for all $h\in\Lambda^{\omega^{\alpha}}v$. Hence 
\begin{align*}
\left\Vert \psi_{\alpha}\left(\varphi_{\alpha}\left(a\right)P_{v}\xi\right)\right\Vert  & =\left\Vert \left(b+c\right)\psi_{\alpha}\left(P_{v}\xi\right)\right\Vert \\
 & =\left\Vert c\psi_{\alpha}\left(P_{v}\xi\right)\right\Vert \\
 & =\left\Vert \psi_{\alpha}\left(SP_{v}\xi\right)\right\Vert <\frac{1}{2}\varepsilon
\end{align*}
On the other hand, if $q\in v\Lambda^{\omega^{\zeta}}$, then by \prettyref{lem:norm-in-next-alg},
\begin{align}
\left\Vert \varphi_{\alpha}\left(a\right)P_{v}\xi\right\Vert  & =\left\Vert \varphi_{\zeta}\left(a\psi_{\alpha}\left(P_{v}\xi\right)\right)\delta_{q}\right\Vert \nonumber \\
 & \leq\left\Vert \varphi_{\zeta}\left(b\right)\varphi_{\zeta}\left(\psi_{\alpha}\left(P_{v}\xi\right)\right)\delta_{q}\right\Vert +\left\Vert \varphi_{\zeta}\left(c\psi_{\alpha}\left(P_{v}\xi\right)\right)\delta_{q}\right\Vert \nonumber \\
 & <\frac{1}{2}\varepsilon+\left\Vert S\xi\right\Vert <\varepsilon\label{eq:left-action-ineq}
\end{align}
since $\left\Vert \varphi_{\zeta}\left(\psi_{\alpha}\left(P_{v}\xi\right)\right)\delta_{q}\right\Vert \leq\left\Vert \xi\right\Vert \leq1$.
Because $G\cup H$ is finite, $\varphi_{\alpha}\left(a\right)P_{v}\in\mathcal{K}\left(X_{\alpha}\right)$.

Let $V=\left\{ s\left(g\right):g\in G\right\} \cup\left\{ s\left(f\left(\omega^{\alpha}\right)\right):f\in F\right\} $
and suppose $v\in\Lambda_{0}\backslash V$, $\xi\in\mathrm{span}\left\{ \delta_{h}:h\in\Lambda^{\omega^{\alpha}}v\right\} $,
and $\left\Vert \xi\right\Vert \leq1$. If $v\Lambda^{\omega^{\zeta}}=\emptyset$,
then the same argument as before shows $\left\Vert \varphi_{\alpha}\left(a\right)\xi\right\Vert <\varepsilon$.
Moreover, if $q\in v\Lambda^{\omega^{\zeta}}$ and $h\in\Lambda^{\omega^{\alpha}}v$,
then $hq\not\in F$ and $h\not\in G$. Thus by \prettyref{eq:left-action-ineq},
$\left\Vert \varphi_{\alpha}\left(a\right)\xi\right\Vert <\varepsilon$.
Since $V$ is finite, we have that the map $v\mapsto\left\Vert \varphi_{\alpha}\left(a\right)P_{v}\right\Vert $
belongs to $c_{0}\left(\Lambda_{0}\right)$, and $\varphi_{\alpha}\left(a\right)\in\mathcal{K}\left(X_{\alpha}\right)$.

Finally, let $p\in\ker\varphi_{\alpha}$ be given. If $h\in\Lambda^{\omega^{\zeta}}$,
then
\[
\left\Vert \varphi_{\zeta}\left(p\right)\delta_{h}\right\Vert =\left\Vert \varphi_{\zeta}\left(pT_{h\left(\omega^{\alpha}\right)}\right)\delta_{h\left(\omega^{\alpha}\right)^{-1}h}\right\Vert =\left\Vert \varphi_{\zeta}\left(\psi_{\alpha}\left(\varphi_{\alpha}\left(p\right)\delta_{h\left(\omega^{\alpha}\right)}\right)\right)\delta_{h\left(\omega^{\alpha}\right)^{-1}h}\right\Vert =0
\]
Therefore $p\in\ker\varphi_{\zeta}$, and since $b\in J_{\zeta}$,
\[
ap=bp+cp=cp=\psi_{\alpha}^{(1)}\left(S\varphi_{\alpha}\left(p\right)\right)=0
\]
By definition, this tells us $a\in\left(\ker\varphi_{\alpha}\right)^{\perp}$. 
\end{proof}
\begin{lem}
\label{lem:katsura-sandwich}Suppose $\alpha<\zeta$, and let $\mathcal{A}_{\zeta}$
be the algebra defined in \prettyref{prop:conditional-expectations}.
Then for all $a\in\mathcal{O}\left(\Lambda_{\alpha+1}\right)\cap J_{\zeta}$
and $z\in\mathbb{T}$, $\Gamma_{\alpha,z}\left(a^{*}\right)\mathcal{O}\left(\Lambda_{\zeta}\right)\Gamma_{\alpha,z}\left(a\right)\subseteq\mathcal{A}_{\zeta}$
and $\Gamma_{\alpha,z}\left(a\right)\in\mathcal{A}_{\alpha+1}$.
\end{lem}

\begin{proof}
Since $\mathcal{O}\left(\Lambda_{\zeta}\right)$ is densely spanned
by elements $T_{p_{1}}T_{q_{1}}^{*}$ for $p_{1},q_{1}\in\Lambda_{\zeta}$,
we select such paths and prove $aT_{p_{1}}T_{q_{1}}^{*}a\subseteq\mathcal{A}_{\zeta}$.
Then there exists $\gamma<\zeta$ such that $p_{1},q_{1}\in\Lambda_{\theta}$
for all $\theta>\gamma$, and in particular, $v\left(p_{1}\right)_{\theta}=v\left(q_{1}\right)_{\theta}=0$.
Therefore, $T_{p_{1}}T_{q_{1}}^{*}\in\mathcal{A}_{\zeta}$ if $p_{1}\Lambda^{*}\cup q_{1}\Lambda^{*}\subseteq C_{\gamma}$
for all $\gamma+1\in[\alpha+1,\zeta)$.

For each $\gamma$ such that $\gamma+1\in[\alpha+1,\zeta)$, $J_{\zeta}\cap\mathcal{O}\left(\Lambda_{\alpha+1}\right)\subseteq J_{\gamma+1}$
by \prettyref{lem:gauge-containment}. Therefore $a\in\mathcal{O}\left(\Lambda_{\alpha+1}\right)\cap\bigcap_{\gamma+1\in[\alpha+1,\zeta)}J_{\gamma+1}$,
and by \prettyref{lem:katsura-ideal-generators}, this ideal is generated
by vertex projections. Thus $\Gamma_{\alpha,z}\left(a\right)$ also
belongs to this ideal, and by \prettyref{prop:katsura-ideal-form}
the ideal is densely spanned by $T_{p_{2}}T_{q_{2}}^{*}$ for $p_{2},q_{2}\in\Lambda_{\alpha+1}$
such that $s\left(p_{2}\right)=s\left(q_{2}\right)$ is $\gamma+1$-regular
for all $\gamma+1\in[\alpha+1,\zeta)$. Therefore by symmetry it suffices
to prove that for such $p_{1},q_{1},p_{2},q_{2}$, $T_{p_{1}}T_{q_{1}}^{*}T_{p_{2}}T_{q_{2}}^{*}=T_{p_{3}}T_{q_{3}}^{*}$
where $q_{3}\in\Lambda_{\zeta}$ satisfies $q_{3}\Lambda^{*}\subseteq C_{\gamma}$
for all $\gamma+1\in[\alpha+1,\zeta)$. If $T_{p_{1}}T_{q_{1}}^{*}T_{p_{2}}T_{q_{2}}^{*}\not=0$,
then either $p_{2}\in q_{1}\Lambda$ or $q_{1}\in p_{2}\Lambda$.
If $p_{2}\in q_{1}\Lambda$, then $T_{p_{1}}T_{q_{1}}^{*}T_{p_{2}}T_{q_{2}}^{*}=T_{p_{1}q_{1}^{-1}p_{2}}T_{q_{2}}^{*}$
and $s\left(q_{2}\right)=s\left(p_{2}\right)$ is $\gamma+1$-regular
for all $\gamma+1\in[\alpha+1,\zeta)$. By \prettyref{cor:alpha-reg-non-returning}
and \prettyref{cor:non-returning-equiv}, $q_{2}\Lambda^{*}\subseteq C_{\gamma}$
for every $\gamma+1\in[\alpha+1,\zeta)$. On the other hand, if $q_{1}\in p_{2}\Lambda$,
$T_{p_{1}}T_{q_{1}}^{*}T_{p_{2}}T_{q_{2}}^{*}=T_{p_{1}}T_{q_{2}p_{2}^{-1}q_{1}}^{*}$
and $q_{2}p_{2}^{-1}q_{1}\Lambda^{*}\subseteq q_{2}\Lambda^{*}\subseteq C_{\gamma}$
for every $\gamma+1\in[\alpha+1,\zeta)$.

To see $\Gamma_{\alpha,z}\left(a\right)\in\mathcal{A}_{\alpha+1}$,
note that for $p,q\in\Lambda_{\alpha+1}$ with $s\left(p\right)=s\left(q\right)$
$\gamma+1$-regular for all $\gamma+1\in[\alpha+1,\zeta)$, $p\Lambda^{*}\cup q\Lambda^{*}\subseteq C_{\gamma}$
for every $\gamma+1\in[\alpha+1,\zeta)$. Since $p,q\in\Lambda_{\alpha+1}$,
we also have $v\left(p\right)_{\theta}=v\left(q\right)_{\theta}=0$
for all $\theta\geq\alpha+1$. Hence $T_{p}T_{q}^{*}\in\mathcal{A}_{\alpha+1}$,
and by preceeding comments, $\Gamma_{\alpha,z}\left(a\right)$ is
approximated by linear combinations of such $T_{p}T_{q}^{*}$.
\end{proof}
\begin{lem}
\label{lem:expectation-kernel}Suppose $\alpha+1<\zeta$, and define
for each $z\in\mathbb{T}$ and $\theta\in[\alpha,\zeta)$ unitaries
$U_{\theta,z}\in\mathcal{U}\left(X_{\zeta}\right)$ by $U_{\theta,z}\delta_{f}=z^{v\left(f\right)_{\theta}}\delta_{f}$.
Then for all $\theta+1<\zeta$ and $b\in\mathcal{A}_{\theta+1}$,
\[
U_{\theta,z}\varphi_{\zeta}\left(b\right)U_{\theta,z}^{*}=\varphi_{\zeta}\left(\overline{\Gamma}_{\theta,z}\left(b\right)\right)
\]
Similarly, if $b\in\mathcal{A}_{\zeta}\cap J_{\zeta}$ and $\theta+1=\zeta$,
then $\mathrm{Ad}\:U_{\theta,z}\circ\varphi_{\zeta}=\varphi_{\zeta}\circ\Gamma_{\theta,z}$.
Furthermore, if $b\in\mathcal{A}_{\zeta}$ is positive and $b\in\ker\varphi_{\zeta}$,
then $E_{\alpha+1}^{\zeta}\left(b\right)\in\ker\varphi_{\zeta}$.
\end{lem}

\begin{proof}
Let $\theta+1<\zeta$, and suppose $p,q\in\Lambda^{\omega^{\zeta}}$
such that $p\in\Lambda_{\theta+1}$ iff $q\in\Lambda_{\theta+1}$
and $p\Lambda^{*}\cup q\Lambda^{*}\subseteq C_{\theta}$. If $p,q\in\Lambda_{\theta+1}$
and $\varphi_{\zeta}\left(T_{p}T_{q}^{*}\right)\delta_{f}\not=0$,
then $f$ is $\theta$-cancellative and $v\left(pq^{-1}f\right)_{\theta}=v\left(p\right)_{\theta}-v\left(q\right)_{\theta}+v\left(f\right)_{\theta}$.
If $p,q\in\Lambda\backslash\Lambda_{\theta+1}$, then $v\left(pq^{-1}f\right)_{\theta}=v\left(p\right)_{\theta}$
and $v\left(f\right)_{\theta}=v\left(pq^{-1}f\right)_{\theta}=v\left(q\right)_{\theta}$.
In any case,
\begin{align*}
U_{\theta,z}\varphi_{\zeta}\left(T_{p}T_{q}^{*}\right)U_{\theta,z}^{*}\delta_{f} & =z^{-v\left(f\right)_{\theta}}U_{\theta,z}\varphi_{\zeta}\left(T_{p}T_{q}^{*}\right)\delta_{f}\\
 & =z^{v\left(pq^{-1}f\right)_{\theta}-v\left(f\right)_{\theta}}\varphi_{\zeta}\left(T_{p}T_{q}^{*}\right)\delta_{f}\\
 & =\varphi_{\zeta}\left(\overline{\Gamma}_{\theta,z}\left(T_{p}T_{q}^{*}\right)\right)\delta_{f}
\end{align*}
from which the first claim follows. If $\theta+1=\zeta$ and $b\in\mathcal{A}_{\zeta}\cap J_{\zeta}$,
then by \prettyref{lem:non-returning-zero}, $\varphi_{\zeta}\left(b\right)\delta_{f}=0$
unless $f$ is $\theta$-cancellative, and the same calculations show
$\mathrm{Ad}\:U_{\theta,z}\circ\varphi_{\zeta}=\varphi_{\zeta}\circ\Gamma_{\theta,z}$.

To prove the final claim, we first prove by transfinite induction
that for every $b\in\mathcal{A}_{\beta}$ and $f,g\in\Lambda^{\omega^{\zeta}}$,
\begin{equation}
\left\langle \delta_{f},\varphi_{\zeta}\left(E_{\gamma}^{\beta}\left(b\right)\right)\delta_{g}\right\rangle =\begin{cases}
\left\langle \delta_{f},b\delta_{g}\right\rangle  & v\left(f\right)_{\theta}=v\left(g\right)_{\theta}\text{ for all }\theta\in[\gamma,\beta)\\
0 & \text{otherwise}
\end{cases}\label{eq:left-action-expectation}
\end{equation}
Clearly the claim is true if $\beta=\gamma$. Moreover, if $\beta=\nu+1$
and the claim is true for $E_{\gamma}^{\nu}$, then
\begin{align*}
\left\langle \delta_{f},\varphi_{\zeta}\left(E_{\gamma}^{\beta}\left(b\right)\right)\delta_{g}\right\rangle  & =\left\langle \delta_{f},\varphi_{\zeta}\left(E_{\gamma}^{\nu}\circ E_{\nu}^{\beta}\left(b\right)\right)\delta_{g}\right\rangle \\
 & =\begin{cases}
\left\langle \delta_{f},\varphi_{\zeta}\left(E_{\nu}^{\beta}\left(b\right)\right)\delta_{g}\right\rangle  & v\left(f\right)_{\theta}=v\left(g\right)_{\theta}\text{ for all }\theta\in[\gamma,\nu)\\
0 & \text{otherwise}
\end{cases}
\end{align*}
Computing $\left\langle \delta_{f},\varphi_{\zeta}\left(E_{\nu}^{\beta}\left(b\right)\right)\delta_{g}\right\rangle $
in the first case, we have
\begin{align*}
\left\langle \delta_{f},\varphi_{\zeta}\left(E_{\nu}^{\beta}\left(b\right)\right)\delta_{g}\right\rangle  & =\int_{\mathbb{T}}\left\langle \delta_{f},\varphi_{\zeta}\left(\overline{\Gamma}_{\nu,z}\left(b\right)\right)\delta_{g}\right\rangle \:dz\\
 & =\int_{\mathbb{T}}\left\langle \delta_{f},U_{\nu,z}\varphi_{\zeta}\left(b\right)U_{\nu,z}^{*}\delta_{g}\right\rangle \:dz\\
 & =\int_{\mathbb{T}}\left\langle U_{\nu,z}^{*}\delta_{f},\varphi_{\zeta}\left(b\right)U_{\nu,z}^{*}\delta_{g}\right\rangle \:dz\\
 & =\int_{\mathbb{T}}z^{v\left(g\right)_{\nu}-v\left(f\right)_{\nu}}\left\langle \delta_{f},\varphi_{\zeta}\left(b\right)\delta_{g}\right\rangle \:dz\\
 & =\begin{cases}
\left\langle \delta_{f},\varphi_{\zeta}\left(b\right)\delta_{g}\right\rangle  & v\left(g\right)_{\nu}=v\left(f\right)_{\nu}\\
0 & \text{otherwise}
\end{cases}
\end{align*}
Thus $E_{\gamma}^{\beta}$ satisfies \prettyref{eq:left-action-expectation}.
If $\beta$ is a limit ordinal and $T_{p}T_{q}^{*}$ is a generator
for $\mathcal{A}_{\beta}$, then either $E_{\gamma}^{\beta}\left(T_{p}T_{q}^{*}\right)=0$,
in which case \prettyref{eq:left-action-expectation} holds, or $\left(E_{\gamma}^{\eta}\left(T_{p}T_{q}^{*}\right)\right)_{\eta}$
is eventually constant. In this case, we may choose $\eta$ such that
$E_{\gamma}^{\eta}\left(T_{p}T_{q}^{*}\right)=E_{\gamma}^{\beta}\left(T_{p}T_{q}^{*}\right)$.
Therefore \prettyref{eq:left-action-expectation} is satisfied for
$b=T_{p}T_{q}^{*}$, and hence for arbitrary $b\in\mathcal{A}_{\beta}$.

Let $b\in\mathcal{A}_{\beta}\cap\ker\varphi_{\zeta}$ be positive.
Then
\[
\left\langle \delta_{f},\varphi_{\zeta}\left(E_{\alpha+1}^{\beta}\left(b\right)\right)\delta_{f}\right\rangle =\left\langle \delta_{f},\varphi_{\zeta}\left(b\right)\delta_{f}\right\rangle =0
\]
for all $f\in\Lambda^{\omega^{\zeta}}$. Since $E_{\alpha+1}^{\beta}\left(b\right)$
is postive, this implies $E_{\alpha+1}^{\beta}\left(b\right)\in\ker\varphi_{\zeta}$.
\end{proof}
\begin{prop}
\label{prop:katsura-ideal-gauge-invariant}For each $\alpha<\zeta$,
$\mathcal{O}\left(\Lambda_{\alpha+1}\right)\cap J_{\zeta}$ is a gauge-invariant
ideal of $\mathcal{O}\left(\Lambda_{\alpha+1}\right)$.
\end{prop}

\begin{proof}
For $z\in\mathbb{T}$, define $U_{\alpha,z}\in\mathcal{U}\left(X_{\zeta}\right)$
by $U_{\alpha,z}\delta_{f}=z^{v\left(f\right)_{\alpha}}\delta_{f}$
as in \prettyref{lem:expectation-kernel}. Then for $a\in\mathcal{O}\left(\Lambda_{\alpha+1}\right)\cap J_{\zeta}$,
$\Gamma_{\alpha,z}\left(a\right)\in\mathcal{A}_{\alpha+1}$ by \prettyref{lem:katsura-sandwich},
so $U_{\alpha,z}\varphi_{\zeta}\left(a\right)U_{\alpha,z}^{*}=\varphi_{\zeta}\left(\Gamma_{\alpha,z}\left(a\right)\right)$.
Hence $\varphi_{\zeta}\left(\Gamma_{\alpha,z}\left(a\right)\right)\in\mathcal{K}\left(X_{\zeta}\right)$
since $\varphi_{\zeta}\left(a\right)\in\mathcal{K}\left(X_{\zeta}\right)$.

All that remains is to show for $b\Gamma_{\alpha,z}\left(a\right)=0$
for all $b\in\ker\varphi_{\zeta}$. Letting $\left(e_{\lambda}\right)$
be an approximate unit for $J_{\zeta}$, this is equivalent to showing
$\Gamma_{\alpha,z}\left(a^{*}e_{\lambda}^{*}\right)b^{*}b\Gamma_{\alpha,z}\left(e_{\lambda}a\right)=0$.
Applying \prettyref{lem:katsura-sandwich}, define
\begin{align*}
c & =E_{\alpha+1}^{\zeta}\left(\Gamma_{\alpha,z}\left(a^{*}e_{\lambda}^{*}\right)b^{*}b\Gamma_{\alpha,z}\left(e_{\lambda}a\right)\right)\\
 & =\Gamma_{\alpha,z}\left(a^{*}\right)E_{\alpha+1}^{\zeta}\left(\Gamma_{\alpha,z}\left(e_{\lambda}^{*}\right)b^{*}b\Gamma_{\alpha,z}\left(e_{\lambda}\right)\right)\Gamma_{\alpha,z}\left(a\right)
\end{align*}
and note that since $b\in\ker\varphi_{\zeta}$, \prettyref{lem:expectation-kernel}
implies $c\in\ker\varphi_{\zeta}$. Therefore $\varphi_{\zeta}\left(\overline{\Gamma}_{\alpha,\overline{z}}\left(c\right)\right)=U_{\alpha,z}^{*}\varphi_{\zeta}\left(c\right)U_{\alpha,z}=0$,
and $\overline{\Gamma}_{\alpha,z}\left(c\right)\in\ker\varphi_{\zeta}$.
However,
\[
\overline{\Gamma}_{\alpha,\overline{z}}\left(c\right)=a^{*}\overline{\Gamma}_{\alpha,\overline{z}}\left(E_{\alpha+1}^{\zeta}\left(\Gamma_{\alpha,z}\left(e_{\lambda}^{*}\right)b^{*}b\Gamma_{\alpha,z}\left(e_{\lambda}\right)\right)\right)a
\]
Since $a\in\left(\ker\varphi_{\zeta}\right)^{\perp}$, $\overline{\Gamma}_{\alpha,\overline{z}}\left(c\right)=0$,
and $c=0$. Since $E_{\alpha+1}^{\zeta}$ is faithful by \prettyref{prop:conditional-expectations},
$\Gamma_{\alpha,z}\left(a^{*}e_{\lambda}^{*}\right)b^{*}b\Gamma_{\alpha,z}\left(e_{\lambda}a\right)=0$,
as desired.
\end{proof}
\begin{lem}
\label{lem:katsura-ideal-vertices}For $v\in\Lambda_{0}$, $T_{v}\in J_{\zeta}$
if and only if $v$ is $\zeta$-regular.
\end{lem}

\begin{proof}
First suppose $v\in\Lambda_{0}$ and $T_{v}\in J_{\zeta}$. By \prettyref{cor:compactness}
there exists finite $F\subseteq\Lambda^{\omega^{\zeta}}$ such that
for all $\xi\in\overline{\mathrm{span}}\left\{ \delta_{f}:f\in\Lambda^{\omega^{\zeta}}\backslash F\right\} $
with $\|\xi\|\leq1$, $\left\Vert \varphi_{\zeta}\left(T_{v}\right)\xi\right\Vert <1$.
Observe that for each $f\in v\Lambda^{\omega^{\zeta}}$, $\left\Vert \varphi_{\zeta}\left(T_{v}\right)\delta_{f}\right\Vert =\left\Vert \delta_{f}\right\Vert =\left\Vert T_{s\left(f\right)}\right\Vert =1$,
and this implies $f\in F$. Hence $v\Lambda^{\omega^{\zeta}}\subseteq F$,
and $v$ is $\zeta$-row-finite. Now suppose $e\in v\Lambda_{\zeta}$.
If $s\left(e\right)\Lambda^{\omega^{\zeta}}=\emptyset$, then $T_{e}\in\ker\varphi_{\zeta}$
and $T_{v}T_{e}=T_{e}=0$, which is a contradiction. Therefore $s\left(e\right)\Lambda^{\omega^{\zeta}}\not=\emptyset$,
and $v$ is $\zeta$-regular.

For the other implication, let $v\in\Lambda_{0}$ be $\zeta$-regular.
Define $F=v\Lambda^{\omega^{\zeta}}$, which is finite. Observe that
$\varphi_{\zeta}\left(T_{v}\right)\delta_{f}=0$ for all $f\in\Lambda^{\omega^{\zeta}}\backslash F$.
Hence if $\xi\in\overline{\mathrm{span}}\left\{ \delta_{f}:f\in\Lambda^{\omega^{\zeta}}\backslash F\right\} $,
$\varphi_{\zeta}\left(T_{v}\right)\xi=0$, and $\varphi_{\zeta}\left(T_{v}\right)\in\mathcal{K}\left(X_{\zeta}\right)$
by \prettyref{cor:compactness}. Now let $b\in\ker\varphi_{\zeta}$
be given. Then
\[
\rho_{\zeta}\left(bT_{v}\right)=\rho_{\zeta}\left(b\right)\sum_{f\in v\Lambda^{\omega^{\zeta}}}T_{f}T_{f}^{*}=\sum_{f\in v\Lambda^{\omega^{\zeta}}}\psi_{\zeta}\left(\varphi_{\zeta}\left(b\right)\delta_{f}\right)T_{f}^{*}=0
\]
Since $\rho_{\zeta}$ is injective by \prettyref{lem:injective},
we have $bT_{v}=0$, and $T_{v}\in\left(\ker\varphi_{\zeta}\right)^{\perp}$.
\end{proof}
\begin{prop}
\label{prop:katsura-ideal-generators}$J_{\zeta}$ is the smallest
ideal of $\mathcal{O}\left(\Lambda_{\zeta}\right)$ containing $\left\{ T_{v}:v\in\Lambda_{0}\text{ is }\zeta\text{-regular}\right\} $.
\end{prop}

\begin{proof}
We prove by transfinite induction that $J_{\zeta}$ is the smallest
ideal of $\mathcal{O}\left(\Lambda_{\zeta}\right)$ containing $J_{\zeta}\cap\mathcal{O}\left(\Lambda_{0}\right)$,
which is sufficient by \prettyref{lem:katsura-ideal-vertices}. Suppose
$\beta\leq\zeta$ and for all $\alpha<\beta$, $J_{\zeta}\cap\mathcal{O}\left(\Lambda_{\alpha}\right)$
is the smallest ideal of $\mathcal{O}\left(\Lambda_{\alpha}\right)$
containing $J_{\zeta}\cap\mathcal{O}\left(\Lambda_{0}\right)$. We
will show $J_{\zeta}\cap\mathcal{O}\left(\Lambda_{\beta}\right)$
is the smallest ideal of $\mathcal{O}\left(\Lambda_{\beta}\right)$
containing $J_{\zeta}\cap\mathcal{O}\left(\Lambda_{0}\right)$. Clearly
this is true if $\beta=0$, so we now assume $\beta>0$. If $\beta$
is a limit ordinal, then by \prettyref{prop:inductive-limit},
\[
\mathcal{O}\left(\Lambda_{\beta}\right)=\overline{\bigcup_{\alpha<\beta}\mathcal{O}\left(\Lambda_{\alpha}\right)}
\]
Hence by \cite[II.8.2.4]{Encyclopaedia},
\[
\mathcal{O}\left(\Lambda_{\beta}\right)\cap J_{\zeta}=\overline{\bigcup_{\alpha<\beta}J_{\zeta}\cap\mathcal{O}\left(\Lambda_{\alpha}\right)}
\]
For each ideal $\mathcal{I}$ of $\mathcal{O}\left(\Lambda_{\beta}\right)$
containing $J_{\zeta}\cap\mathcal{O}\left(\Lambda_{0}\right)$, the
indutive hypothesis implies $J_{\zeta}\cap\mathcal{O}\left(\Lambda_{\alpha}\right)\subseteq\mathcal{I}\cap\mathcal{O}\left(\Lambda_{\alpha}\right)$;
hence $\mathcal{O}\left(\Lambda_{\beta}\right)\cap J_{\zeta}\subseteq\mathcal{I}$,
proving the result. Now suppose $\beta$ is a successor ordinal and
$\beta=\alpha+1$. By \prettyref{prop:katsura-ideal-gauge-invariant},
$J_{\zeta}\cap\mathcal{O}\left(\Lambda_{\alpha+1}\right)$ is a gauge-invariant
ideal of $\mathcal{O}\left(\Lambda_{\alpha+1}\right)$. By \cite[Theorem 8.6]{GAUGE-INV-IDEALS},
the gauge-invariant ideals $\mathcal{I}$ of $\mathcal{O}\left(\Lambda_{\alpha+1}\right)$
are determined by $\mathcal{I}\cap\mathcal{O}\left(\Lambda_{\alpha}\right)$
and $\left(\mathcal{I}+\psi_{\alpha}^{(1)}\left(\mathcal{K}\left(X_{\alpha}\right)\right)\right)\cap\mathcal{O}\left(\Lambda_{\alpha}\right)$.
Note that for each gauge-invariant ideal $\mathcal{I}$ of $\mathcal{O}\left(\Lambda_{\alpha+1}\right)$,
\[
J_{\alpha}=\psi_{\alpha}^{(1)}\left(\mathcal{K}\left(X_{\alpha}\right)\right)\cap\mathcal{O}\left(\Lambda_{\alpha}\right)\subseteq\left(\mathcal{I}+\psi_{\alpha}^{(1)}\left(\mathcal{K}\left(X_{\alpha}\right)\right)\right)\cap\mathcal{O}\left(\Lambda_{\alpha}\right)
\]
Moreover, \prettyref{lem:gauge-containment} tells us
\[
\left(J_{\zeta}\cap\mathcal{O}\left(\Lambda_{\alpha+1}\right)+\psi_{\alpha}^{(1)}\left(\mathcal{K}\left(X_{\alpha}\right)\right)\right)\cap\mathcal{O}\left(\Lambda_{\alpha}\right)\subseteq J_{\alpha}
\]
Therefore $J_{\zeta}\cap\mathcal{O}\left(\Lambda_{\alpha+1}\right)$
is the smallest gauge-invariant ideal containing $J_{\zeta}\cap\mathcal{O}\left(\Lambda_{\alpha}\right)$.
Since $J_{\zeta}\cap\mathcal{O}\left(\Lambda_{\alpha}\right)$ as
an ideal of $\mathcal{O}\left(\Lambda_{\alpha}\right)$ is generated
by $J_{\zeta}\cap\mathcal{O}\left(\Lambda_{0}\right)$ and $J_{\zeta}\cap\mathcal{O}\left(\Lambda_{0}\right)$
is fixed by every gauge action, $J_{\zeta}\cap\mathcal{O}\left(\Lambda_{\alpha+1}\right)$
is the smallest ideal containing $J_{\zeta}\cap\mathcal{O}\left(\Lambda_{0}\right)$.
By induction, this proves $J_{\zeta}$ is the smallest ideal of $\mathcal{O}\left(\Lambda_{\zeta}\right)$
containing $J_{\zeta}\cap\mathcal{O}\left(\Lambda_{0}\right)$.
\end{proof}
\bibliographystyle{plain}
\bibliography{references/refs}

\end{document}